\documentclass[11pt]{imsart}

\usepackage{amssymb, amsmath, amsthm}
\usepackage{tikz}
\usepackage{appendix}
\usepackage[normalem]{ulem}
\usepackage{comment}
\usepackage{framed}
\usepackage{epstopdf}
\usepackage{dsfont}
\usepackage[linesnumbered]{algorithm2e}

\SetKwInput{KwInput}{Input}                
\SetKwInput{KwOutput}{Output}              

\SetKwComment{Comment}{/* }{ */}
\RestyleAlgo{ruled}
\usepackage{xcolor}
\definecolor{blueR}{RGB}{34, 151, 230}
\definecolor{purpleR}{RGB}{206,164,255}

\usepackage[
    left=1.0in,       
    right=1.0in,      
    top=0.7in,        
    bottom=0.7in,     
    includehead, includefoot, 
]{geometry}
\RequirePackage[colorlinks,citecolor=blue,urlcolor=blue]{hyperref}

\newcommand{\quotient}[2]{{\raisebox{.2em}{$#1$}\big/\raisebox{-.2em}{$#2$}}}
\startlocaldefs

\newtheorem {Proposition}{Proposition}[section]
\newtheorem {Lemma}[Proposition] {Lemma}
\newtheorem {Theorem}[Proposition]{Theorem}
\newtheorem {Corollary}[Proposition]{Corollary}
\newtheorem {Remark}[Proposition]{Remark}
\newtheorem {Example}{Example}[section]
\newtheorem {Definition}{Definition}[section]
\newtheorem {Assumption}{Assumption}[section]

\usepackage{amssymb}
 \def\N{\mathbb{N}}

\def\R{\mathbb{R}}
 
\def\E{\mathbb{E}}

\def\x{\mathbf{x}}
\def\y{\mathbf{y}}

\def\zar{\mathcal Z}

\def\E{\mathbb E}

\def\N{\mathbb N}

\def\x{\mathbf x}

\def\y{\mathbf y}

\newcommand{\eps}{\varepsilon}

\renewcommand{\vec}{\operatorname{vec}}
\newcommand{\sym}{\operatorname{sym}}
\providecommand{\nelem}{ {\kappa_{d,g}}}

\providecommand{\Cov}{\operatorname{Cov}}
\renewcommand{\y}{y}
\renewcommand{\x}{x}

\endlocaldefs

\begin{document}


\begin{frontmatter}
\title{Estimation of Algebraic Sets: Extending PCA Beyond Linearity}
\runtitle{Estimation of Algebraic Sets}
\runauthor{Gonz\'alez-Sanz, Mordant, Samperio and Sen}

\begin{aug}
\author[A]{\fnms{Alberto}~\snm{González-Sanz}\ead[label=e1]{ag4855@columbia.edu}},
\author[M]{\fnms{Gilles}~\snm{Mordant}\ead[label=e2]{gilles.mordant@yale.edu}}, 
\author[S]{\fnms{Álvaro}~\snm{Samperio}\thanks{Partially supported by grants PID2021-122501NB-I00 and PID2022-138906NB-C21 funded by MICIU/AEI/ 10.13039/501100011033 and by ERDF/EU. Partially supported by the ALAMA network (RED2022-134176-T (MCI/AEI)).}$^,$\ead[label=e3]{alvaro.samperio@upc.edu}}
\and
\author[B]{\fnms{Bodhisattva}~\snm{Sen}\thanks{Supported by NSF grants DMS-2311062 and DMS-2515520.}$^,$\ead[label=e4]{b.sen@columbia.edu}}
\address[A]{Department of Statistics,
Columbia University, New York, USA,\\ \printead{e1}}
\address[M]{Applied \& Computational Mathematics, Yale University, New Haven, USA \\ \printead{e2}}
\address[S]{Department of Mathematics, Universitat Politècnica de Catalunya - BarcelonaTech (UPC), Barcelona, Spain \printead{e3}}

\address[B]{Department of Statistics, Columbia University, New York, USA, \printead{e4}}

\end{aug}

\begin{abstract}
An algebraic set is defined as the zero locus of a system of real polynomial equations. In this paper we address the problem of recovering an unknown algebraic set $\mathcal{A}$ from noisy observations of latent points lying on 
$\mathcal{A}$---a task that extends principal component analysis, which corresponds to the purely linear case. Our procedure consists of three steps: (i) constructing the {\it moment matrix} from the Vandermonde matrix associated with the data set and the degree of the fitted polynomials, (ii) debiasing this moment matrix to remove the noise-induced bias, (iii) extracting its kernel via an eigenvalue decomposition of the debiased moment matrix. These steps yield $n^{-1/2}$-consistent estimators of the coefficients of a set of generators for the ideal of polynomials vanishing on $\mathcal{A}$. To reconstruct $\mathcal{A}$ itself, we propose three complementary strategies:
(a) compute the zero set of the fitted polynomials; (b) build a semi-algebraic approximation that encloses $\mathcal{A}$;
(c) when structural prior information is available, project the estimated coefficients onto the corresponding constrained space. We prove (nearly) parametric asymptotic error bounds and show that each approach recovers $\mathcal{A}$ under mild regularity conditions.

\end{abstract}



\end{frontmatter}

\maketitle


\section{Introduction}
\label{intro}

Recovering low-dimensional structures corrupted by noise is a central theme in modern statistics. In this work we consider the classical signal-plus-noise model 
\begin{equation}\label{Model}
X_i = \theta_i + \eps_i, \quad i=1,\ldots, n, \ldots  \qquad {\rm with}\quad  \{\eps_i\}_{i\ge 1} \stackrel{\textrm{iid}}{\sim} \mathcal{N}(0,\Sigma),
\end{equation}
where the observed data is $\{X_i \}_{i=1}^n \subseteq \R^d$, $\Sigma$ is a nonnegative definite $d \times d$ covariance matrix, and  the latent (unobserved) signal sequence $\{\theta_i\}_{i\geq 1}$  lies on a ``lower dimensional'' {\it algebraic set} $\mathcal{A} \subseteq \R^d$. By definition, $\mathcal{A}$ is 
the set of common zeros of a collection of polynomials, i.e., 

\begin{equation}\label{eq:A}
\mathcal{A} := \{x \in \R^d: P_j(x) = 0, \text{ for } j = 1,\ldots, k\} 
\end{equation}
for some polynomials $P_1,\ldots, P_k$ (for a given $k$). Lines, circles, curves, surfaces,  higher-dimensional hypersurfaces, etc., can be examples of the algebraic set $\mathcal{A}$; in the following we give two prototypical examples of algebraic sets that we will study in this paper. 
\begin{figure}[h!]
    \centering
\includegraphics[width=0.32\linewidth]{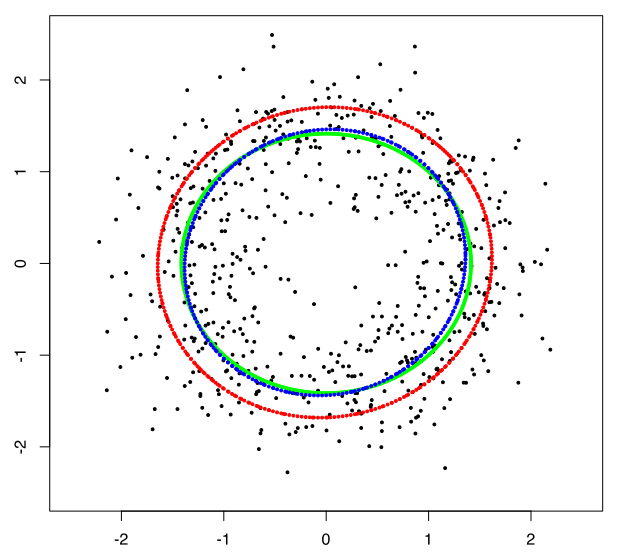}    
\includegraphics[width=0.32\linewidth]{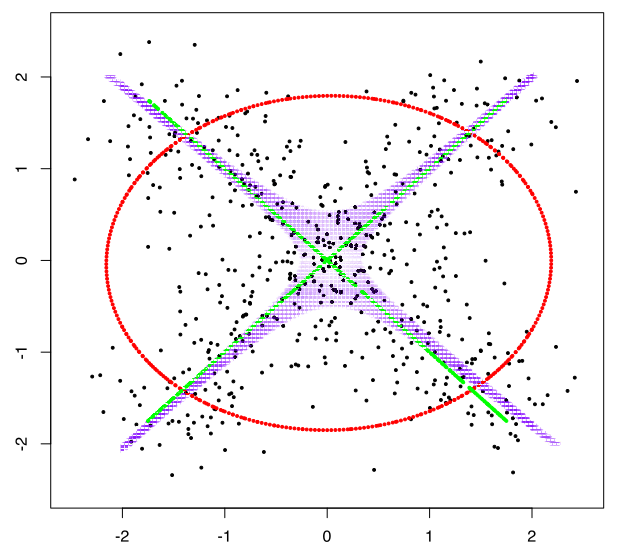}
\includegraphics[width=0.32\linewidth]{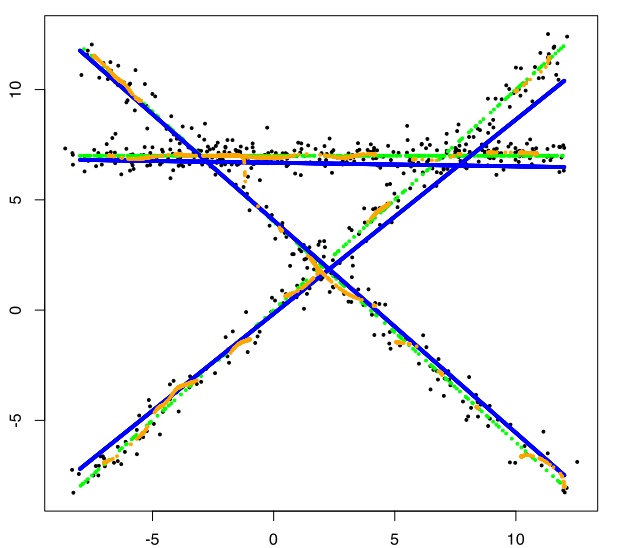}
     \fbox{ 
 \textcolor{green}{ $\bullet$} {\footnotesize $\theta_i$'s}\quad
 \textcolor{black}{$\bullet$} {\footnotesize $X_i$'s}\quad
 \textcolor{blue}{\bf\textemdash} {\footnotesize our method }\quad
    \textcolor{red}{\bf\textemdash} {\footnotesize naive approach }\quad
    \textcolor{purpleR}{\bf\textemdash} {\footnotesize semi algebraic set }\quad
    \textcolor{orange}{\bf\textemdash} {\footnotesize mean shift algo. }}
        \caption{\label{fig:IntroCircle} In each plot the latent $\theta_i$'s are displayed in green, while the observations $X_i$'s are denoted by black dots (we take $\Sigma = (0.4)^2 I_2$). {\bf Left}: The red and the blue curves denote the estimated algebraic set using the inconsistent naive approach and our proposed methodology (a), respectively. {\bf Middle}: A similar plot for Example~\ref{ex: 2}. The semi-algebraic tube estimator is displayed in purple. {\bf Right}: Here the true algebraic set is `three-intersecting-lines' and the estimated algebraic set, obtained by projecting the coefficients of the estimated polynomials to incorporate this known structure, is given in blue. The output of the mean-shift algorithm~\cite{400568} is depicted in orange (the bandwidth matrix for the kernel density estimator in the method is set to the actual noise $\sigma I$). In all the above simulations we take $n=600$ and $\sigma=0.4$. }
\end{figure}
\begin{Example}[Circle in $\R^2$]
\label{ex: 1} 
Let $\{\theta_i\}_{i\ge 1}$ lie on the unit circle $\mathcal{S}^1 :=\{ (x, y)\in \R^2: \ x^2+y^2-1=0\}$,
and we observe $\{{ X}_i\}_{i=1}^n$ from model~\eqref{Model}. The circle is a simple yet instructive algebraic set. A classical example arises from an undamped harmonic oscillator: the pair 
(position,velocity)
traces out a circle or an ellipse, so noisy measurements of that pair may produce the scatter plot shown in the left panel of Figure~\ref{fig:IntroCircle}.

\end{Example}
\begin{Example}[`Cross' in $\R^2$]\label{ex: 2} 
    Consider data points $\{{ X}_i\}_{i=1}^n$ generated from $\{\theta_i\}_{i=1}^n$ lying on the ``cross''  $\{(x, y)\in \R^2 :\  (x+y) (x-y)=0\}$  in $\R^2$ (see the center plot in Figure~\ref{fig:IntroCircle}). A general pair of intersecting lines in $\R^2$ can be represented as an algebraic set of the form  
    $$ \big\{(x, y)\in \R^2 :\  (ax+by+c)(a'x+b'y+c')=0\big\}  $$
    with $a,b,c,a', b', c'\in \R $ and $ |a|+|b|, \ |a'|+|b'|>0.$ Cross-shaped ``filament'' structures arise in astronomy, where galaxies align along intersecting ridges of the cosmic web \cite{genovese2012geometry,chen2017detecting}. Statistically, they correspond to a mixture of two linear regressions with measurement error in both variables \cite{Yao18042015}. Our method extends naturally to any finite mixture of such regressions, as illustrated in the right panel of Figure~\ref{fig:IntroCircle}.

\end{Example}

\begin{center}
{\emph{The statistical task we consider in this paper is to recover the smallest algebraic set\footnote{In particular, our interest is in the Zariski closure (see Definition~\ref{defn:Zariski}) of the set $\{\theta_i\}_{i \ge 1}$.} $\mathcal{A}$ containing all the $\theta_i$'s given data from model~\eqref{Model}}.}
\end{center}

Estimating algebraic sets extends principal component analysis (PCA) from linear to nonlinear geometry. PCA fits the best-approximating linear subspace $L$ to the data, which is itself the common zero set of $k :=d - \mathrm{dim}(L)$ linear equations (say, $b_j^\top x =0$, for $j = 1,\ldots, k$). In other words, PCA already fits an algebraic set---but only of degree 1. Allowing higher-degree polynomials yields far richer shapes.

Interest in learning such algebraic sets has surged since the seminal work of Breiding and Sturmfels~\cite{breiding2018learning,breiding2024metric}. Algebraic sets can represent a wide array of structures~\cite{breiding2018learning} and have found applications in systems biology \cite{gross2016numerical,dickenstein2020algebraic}, quantum chemistry \cite{faulstich2023algebraic}, power-grid modeling \cite{samperio2023sparse}, and many other fields.

Before explaining our approach to recovering the algebraic set $\mathcal{A}$ in the noisy setting, let us comment on the noiseless case; i.e., when $\Sigma =0$ in~\eqref{Model}, which was studied in \cite{pauwels2021data}. 
Consider the example of the unit circle $\{(x,y)\in \R^2:\ x^2+y^2=1\}$. As we assume that each latent point $\theta_i=(\theta_{i,1}, \theta_{i,2})\in \R^2$ lies on the unit circle, it holds that 
\[
\begin{pmatrix} 
1 &  \theta_{1,1} &  \theta_{1,2} &  \theta_{1,1}\theta_{1,2} &  \theta_{1,1}^2 &  \theta_{1,2}^2 \\
1 &  \theta_{2,1} &  \theta_{2,2} &  \theta_{2,1}\theta_{2,2} &  \theta_{2,1}^2 &  \theta_{2,2}^2 \\
 \vdots & \vdots & \vdots & \vdots & \vdots & \vdots \\
1 &  \theta_{n,1} &  \theta_{n,2} &  \theta_{n,1}\theta_{n,2} &  \theta_{n,1}^2 &  \theta_{n,2}^2
\end{pmatrix} \begin{pmatrix}
-1\\ 0\\ 0\\ 0\\ 1\\ 1 
\end{pmatrix} =\begin{pmatrix}
0\\ 0\\ \vdots\\ 0  
\end{pmatrix} \in \R^n. 
\]
Said differently, the vector $(-1,0,0,0,1,1)^\top \in \R^6$ in the kernel of the matrix in the above display\textemdash called a \emph{(multivariate) Vandermonde matrix}\textemdash is the set of coefficients of the polynomial $x^2+y^2-1$ and determines the algebraic set. 

More generally, for any algebraic set 
$\mathcal{A}$ the coefficients of all polynomials that vanish on the latent points 
$\{\theta_i\}_{i \ge 1}$ lie in the kernel of an appropriately chosen Vandermonde matrix. Under mild conditions, this kernel---and hence the defining equations of $\mathcal{A}$---can be recovered from sufficiently many noiseless samples~\cite{pauwels2021data}.

In the case of noisy data, \cite{breiding2018learning,breiding2024metric} suggest recovering the underlying algebraic set by forming the Vandermonde matrix based on the noisy observations $X_i$'s and approximating its kernel with numerical linear–algebra techniques.

Their empirical results show the method can tolerate mild noise, but a rigorous stability analysis remains open. In fact, we show that the presence of noise {\it biases} the empirical Vandermonde matrix (see Lemma~\ref{Lemma:StableIdeal} and Corollary~\ref{Coro:faliureOfDeterministic}) which prevents this method from consistently recovering the coefficients of the underlying polynomials, even as the sample size grows. As a result, when the noise level is not sufficiently small, the algebraic set recovered by this approach may deviate substantially from the true set $\mathcal{A}$.

Figure \ref{fig:IntroCircle} (left) illustrates this mismatch. The red curve plots the estimated algebraic set obtained by approximating the kernel of the Vandermonde matrix constructed from the $X_i$'s; 
the green curve shows the ground-truth algebraic set, the unit circle, which is noticeably different from the red curve. Also see the middle plot of Figure \ref{fig:IntroCircle} which considers Example~\ref{ex: 2}; here the true algebraic set (in green) and the estimated algebraic set (in red) have very different structure.

---------------------------------------------------------------------------------------------------------------------------------------------------------------------------------
In this paper, we address the problem of constructing estimators that consistently recover the algebraic set 
$\mathcal{A}$ as the sample size increases, when we have noisy data from model~\eqref{Model}. Our estimation strategy proceeds as follows: (i) forming the {\it moment matrix} from the Vandermonde matrix associated with the data and the degree of the fitted polynomials,\footnote{If we denote by $\mathbb{V}$ the Vandermonde matrix obtained from the $X_i$'s, then the moment matrix is $\mathbb{V}^\top \mathbb{V}/n$; it has the same kernel as $\mathbb{V}$.} (ii) appropriately debiasing this moment matrix, and (iii) performing an eigenvalue decomposition of this debiased moment matrix to `approximate' its kernel. More precisely, we use the eigenvectors corresponding to the `near-zero'  eigenvalues of the debiased moment matrix to define the polynomials whose zeros approximate the true algebraic set $\mathcal{A}$ (see Algorithm~\ref{alg:two}). This procedure yields a $n^{-1/2}$-consistent estimator of the coefficients of the polynomials vanishing at the latent sequence $\{\theta_i\}_{i\ge 1}$ (see Theorem~\ref{Theorem:main}).

Estimating the coefficients of the vanishing polynomials does not immediately yield consistent estimators of their zeros (i.e., the algebraic set). Consider the sequence of polynomials $P_n(x)=x^2+\frac{1}{n}$ which obviously converges to $P(x) = x^2$ but $P_n$ does not vanish anywhere in $\R$ while $P$ vanishes at $0$. To estimate the true algebraic set $\mathcal{A}$, we propose three estimation strategies, each suited to different assumptions about the structure of 
$\mathcal{A}$: 
\begin{itemize}
    \item[(a)] 
    {\it Zero set of the estimated polynomials}. Suppose that Algorithm~\ref{alg:two} yields the following estimated polynomials $\hat P_1, \ldots, \hat P_{\hat k}$---whose coefficients are the eigenvectors corresponding to the `near-zero' eigenvalues of the debiased moment matrix. A natural estimator of the true algebraic set $\mathcal{A}$ is then $\hat{\mathcal{A}} := \{x \in \R^d: \hat P_j (x) = 0 \mbox{ for } \; j = 1,\ldots, \hat k\}$ (cf.~\eqref{eq:A}). Under strong assumptions on $\mathcal{A}$, this approach leads to a provably consistent 
    estimator of $\mathcal{A}$. Furthermore, locally, around regular points (see Definition~\ref{defn:regular-pt}) in $\mathcal{A}$, this approach leads to the parametric ($n^{-1/2}$) rate of convergence, as the sample size increases (see Theorem~\ref{Theorem:recoveringSet}). The left plot of Figure~\ref{fig:IntroCircle} illustrates this approach; as evident from the plot, the blue curve (obtained from this procedure) nearly coincides with the true algebraic set (shown in green).

    \item[(b)] 
    {\it Semi-algebraic ``tube'' estimator.} After estimating polynomials---$\hat P_1, \ldots, \hat P_{\hat k}$---we consider the semi-algebraic set $\{x\in \R^d: |\hat P_i (x)|\leq \lambda \text{ for all $i=1, \dots, {\hat k} $} \}$, for a suitable $\lambda >0$. Thus the semi-algebraic set consists of all `approximate' zeros of the estimated polynomials. This estimated set is a consistent estimator of $\mathcal{A}$ and converges locally around regular points, at an almost parametric rate (see Theorem~\ref{Theorem:completeRecovery}) under mild assumptions. This semi-algebraic `tube' estimator is shown in purple in the middle plot of Figure~\ref{fig:IntroCircle} which corresponds to Example~\ref{ex: 2}.

    \item[(c)] 
    {\it Structure-aware projection.} When we have strong prior knowledge about the shape of $\mathcal{A}$, we can refine the initial polynomial estimate by projecting its coefficients onto the set that respects that structure. In the ``three-intersecting-lines'' example (Figure \ref{fig:IntroCircle}, right), the true set is the zero locus of a cubic that factorizes into three affine terms; an unconstrained estimate need not share this factorization. By forcing the projected coefficients to satisfy the factorization constraint, we obtain a sharper, structure-compatible estimate. Section \ref{sec: Prior} details this method, and Theorem \ref{Theorem:recoveringSet-Prior} proves its consistency under appropriate  assumptions.
    
\end{itemize}

To the best of our knowledge, this is the first work to study estimation and inference for algebraic sets within the general statistical framework of \eqref{Model}. Our estimation procedure is deliberately simple: it reduces to an eigen-decomposition of a suitably debiased moment matrix, thereby generalizing the familiar mechanics of PCA. Crucially, the method remains reliable in {\it high-noise} regimes \cite{lo2024method}---situations where kernel density based approaches often break down. For example, the rightmost panel of Figure~\ref{fig:IntroCircle} shows the mean-shift algorithm \cite{400568} (orange) failing under heavy noise, whereas our estimator retains its structure. Moreover, we obtain estimators of $\mathcal{A}$ that converge at the parametric rate (up to logarithmic factors in some cases), underscoring both the statistical efficiency and practical appeal of the approach.

Let us now discuss some of the technical novelties of our work. Following the approach of~\cite{pauwels2021data}, in contrast to \cite{breiding2018learning}, we shift the focus from the Vandermonde matrix to the {\it empirical moment matrix}, which is more amenable to statistical analysis. Our key contribution is an exact characterization of the bias in the empirical moment matrix which leads to the construction of an explicit debiased version $\hat{\mathbb{M}}_g$ (see~\eqref{eq: MomDec} and~\eqref{eq: MomDec-2} and Theorem~\ref{theorem:representation}). Further, we prove that, by the concentration of measure phenomenon, the kernel of $\hat{\mathbb{M}}_g$ is close to that of the moment matrix based on the latent points. 

Our procedure builds on the method of moments approach with an isomorphism between the Vandermonde matrix and collections of tensor moments and is formalized in Lemma~\ref{lem: isomorp}; see Remark~\ref{ex: Circle} for an illustration for Example~\ref{ex: 1}. This representation allows us to prove a central limit theorem for the debiased moment matrix (Section~\ref{sec: Cons+CLT}, Theorem~\ref{Theorem:CLTMatrix}), which in turn underpins our main statistical guarantees, including the consistency of the estimated kernel of $\hat{\mathbb{M}}_g$ (Theorem~\ref{Theorem:main}).

Section \ref{subsection:recoverSet} presents, to our knowledge, the first consistency results for algebraic-set estimators in the Painlevé–Kuratowski topology, a standard notion of set convergence in variational analysis~\cite[Chapter~4]{RockafellarWets}. Further, we derive a parametric rate of convergence (up to logarithmic factors in some cases) in local Hausdorff metric around regular points on the algebraic set $\mathcal{A}$ (i.e., points where the algebraic set behaves locally like a smooth  manifold; see Definition~\ref{defn:regular-pt}). Our proofs leverage the implicit function theorem 
and results from algebraic geometry (see Section~\ref{sec: introAlgGeo} for a review of this fascinating field of mathematics).

Our estimation procedure requires the user to specify the noise covariance $\Sigma$---we assume throughout the paper that $\Sigma$ is known. In the isotropic case, i.e., $\Sigma=\sigma^{2}I_d$, the variance parameter $\sigma^{2}$ is, in principle, identifiable (see Remark~\ref{identifiable}).  Further, for some of the theoretical results in Section~\ref{subsection:recoverSet} we assume knowledge of the minimal degree $g^*$ (see Definition~\ref{def:minDegree}). A full treatment of how to estimate either $\Sigma$ (or $\sigma^{2}$) or $g^*$ is beyond the present scope and we leave it to future work.

\subsection*{Related work}

The estimation strategy we develop builds on the classical method of moments, but leverages higher-order tensor moments to capture the algebraic structure hidden in noise. Similar moment-tensor ideas have yielded optimal procedures in multireference alignment \cite{perry2019sample,bandeira2023estimation} and Gaussian-mixture learning \cite{wu2020optimal,pereira2022tensor}; our contribution is to adapt this strategy to polynomial-set recovery by exploiting the Vandermonde–tensor isomorphism.

Related work on low-dimensional structure recovery includes principal-curve and piecewise-linear fits \cite{hastie1989principal,ozertem2011locally,lo2024method}, union-of-subspace models \cite{ma2008estimation}, and manifold denoising \cite{fefferman2023fitting,landa2023robust, zhang2023manifold,wu2024design}. Topological data analysis offers complementary, topology-based summaries of point clouds and its methods are quite different from ours~\cite{chazal2021introduction,fasy2014confidence}. Vanishing Component Analysis (VCA) \cite{VCALivini2013,Zhang2018IVC,masuya2022vanishingcomponentanalysiscontrastive} also seeks vanishing polynomials but, unlike our approach, does not address measurement noise. Our moment-based, debiasing procedure therefore extends the existing toolkit by providing noise-robust, statistically principled recovery together with an explicit polynomial representation.

As already mentioned, the influential work of Breiding et al.~\cite{breiding2018learning} forms the starting point for our study. Building on it, Pauwels et al.~\cite{pauwels2021data} analyzed the noiseless setting from a statistical perspective and, using heuristic arguments, raised questions about the stability of the recovered kernel when small amounts of noise are present. We advance this line of research by developing a principled method that remains valid for arbitrary noise levels under a Gaussian noise model, thereby extending the scope and practical relevance of the earlier results.

Recovering a smooth manifold underlying noisy data has been studied from geometric, topological, and statistical perspectives, giving rise to a substantial body of work; see, for example, \cite{niyogi2008finding, genovese2012geometry,Aamari-AoS-2019,Li-JRSSB-2022} and the references cited therein.

\subsection*{Organization of the work}

The remainder of the paper is organized as follows. Section~\ref{sec: introAlgGeo} reviews the basics on algebraic geometry and tensor products. In Section~\ref{Sec:Algebra} we
examine the effect of additive noise on the Vandermonde embedding and explain why the naive approach fails. We also introduce our debiasing approach and describe our estimation procedure in Section~\ref{Sec:Algebra}. Section~\ref{sec: Cons+CLT} develops the main theory, proving consistency, rates of convergence, and central-limit theorems for the debiased moment matrix and its estimated kernel. The three approaches to the recovery of the true algebraic set, along with their theoretical properties, are detailed in Section~\ref{subsection:recoverSet}. Section~\ref{sec: conclusion} concludes with a summary and a brief discussion of open questions. The Appendix gives the proofs of the main results and some simulation studies. The core functions used to implement our methodology and generate the figures in this paper are available on \href{https://github.com/gmordant/AlgebraicSetDenoising}{GitHub}.


\section{Algebraic geometry: an introduction}
\label{sec: introAlgGeo}

In this section, we present the key elements of algebraic geometry that will be relevant to our work. For readers seeking introductory texts on the subject, we recommend \cite{Hartshorne1977,Shafarevich1974}. Before delving into the main discussion, we first establish some notation and provide necessary background information.

\subsection{Preliminaries}

In this paper, 
$\N_0 := \{0,1,\ldots\}$ will denote the set of nonnegative integers and $\N := \{1,2,\ldots\}$ the set of natural numbers. The abstract probability space where all our random variables will be defined is denoted by $(\Omega, \mathcal{F},\mathbb{P})$.
The expectation of a real-valued random variable $X$ is denoted by 
$ \E[X] $. The expectation of a vector or matrix-valued random element is to be understood component-wise. For a sequence of random variables $\{X_n\}_{n\in \N}$ taking values in a separable metric space $(\mathcal{X},\mathfrak{d})$, we will use the usual notions of stochastic convergence. In particular, we use $ \overset{\mathbb{P}}{\to}$, $\xrightarrow{\ a.s.  }$, and $\xrightarrow{\ w\ }$ to denote convergence in probability, almost sure convergence, and convergence in distribution, respectively. The space of all Borel probability measures over $\R^d$ is denoted by $\mathcal{P}(\R^d)$. For any $k \in \N$, $\mathcal{P}_{k}(\R^d)$ denotes the space of all Borel probability measures over $\R^d$ with finite $k$-th moment.

The space of matrices of dimension $d_1\times d_2$ over the reals is denoted by $\mathcal{M}_{d_1\times d_2}(\R)$. For a $d_1\times d_2$ matrix $A$ we denote by $ {\rm Ker}(A)$ its kernel, i.e., ${\rm Ker}(A) := \{x \in \R^{d_2}: Ax = 0\}$. 
The Moore--Penrose pseudo-inverse of $A\in\mathcal{M}_{d_1\times d_2}(\R) $ is denoted by $A^\dagger$.  For a matrix $A = (a_{i,j}) \in\mathcal{M}_{d_1\times d_2}(\R) $ and a matrix $B\in\mathcal{M}_{d_3\times d_4}(\R) $ we denote by $A\otimes_K B$ the Kronecker product between $A$ and $B$. That is, $A\otimes_K B$ is the $(d_1d_3)\times (d_2d_4) $ real matrix with block entries
$(a_{i,j}B)$.  For a matrix $A=(a_{i,j}) \in \mathcal{M}_{p\times p}(\R)$,  we define
\(
{\rm vec}(A) :=(a_{1,1},\dots, a_{p,1}, \dots, a_{1,p}, \dots, a_{p,p} ) \in \R^{p^2}.
\)
The Frobenius inner product between two matrices $A,B\in \mathcal{M}_{p\times p}(\R)$ is given by \( \langle A, B\rangle_{{\rm Fr}} :=\operatorname{tr} (A^\top B)= \vec(A)^\top \vec(B) \)
and the associated norm is denoted by $\|\cdot\|_{{\rm Fr}}$.

We briefly introduce some notation on tensor products. Let $U$ and $V$ be two vector spaces over $\R$.  The tensor product $U\otimes V$ is a vector space over $\R$ equipped with a bilinear map $$ \otimes: U\times V \ni (u,v)\mapsto u\otimes v \in U\otimes V $$ such that if $\mathcal{B}_U$ is any basis of $U$ and $\mathcal{B}_V$ is any basis of $V$ then 
    $ \{ u \otimes v: \ u\in \mathcal{B}_U, \ v\in \mathcal{B}_V \} $ is a basis of $U\otimes V$. The space $U\otimes V$ is unique up to isomorphisms. 
The tensor product of a vector space \(V\) with itself is denoted by \(V \otimes V\). The tensor product of \(V\) with itself $g\in \N$ times is denoted by \[V^{\otimes_g}= \underbrace{V\otimes\cdots \otimes V}_{g\ {\rm times}}.\]
Note that $ V^{\otimes_g}$ is a vector space over $\R$ equipped with a multilinear  map
$$\otimes_g: V^{g} \ni (v_1, \dots, v_g) \mapsto \otimes_g(v_1, \dots, v_g)\in V^{\otimes_g}$$
such that 
    $ \{ \otimes_g(v_1, \dots, v_g):  \ v_1, \dots, v_g\in \mathcal{B}_V \} $ is a basis of $V^{\otimes_g}$.
Further, we use the notation $v^{\otimes_g}=\otimes_g(v, \dots, v) \in V^{\otimes_g}$ for the $g$-th order tensor of $v\in V$.

\subsection{Algebraic sets and the Zariski closure}
\label{Sec:AlgebraicCurves}
Let $\R[x_1, \dots, x_d]$ be the set of all polynomials in $d$ variables over $\R$. The set of all polynomials of degree less than or equal to $g \in \N$ is denoted by $\R_{\leq g}[x_1, \dots, x_d]$. The degree of a polynomial $P$ is denoted by ${\rm deg}(P)$. An {\it ideal} of polynomials is a set $I\subset \R[x_1, \dots, x_d]$ such that if $P_1,P_2\in I$ and $P_3\in \R[x_1, \dots, x_d]$, then $P_1+P_2\in I$ and $P_1P_3\in I$. The ideal $<S>$ generated by a set $S\subset \R[x_1, \dots, x_d]$ is the smallest ideal (in the sense of containment) containing $S$.
    For $ S\subset \R[x_1, \dots, x_d] $ we denote by $$\mathcal{V}(S) := \{ \y\in  \R^d: P(\y)=0, \quad \text{for all }P\in S\}  $$
the set of zeros of all polynomials belonging to $S$. 
A set $V \subset \R^d$  is said to be {\it algebraic} if there exists $ S\subset \R[x_1, \dots, x_d] $ such that $V=\mathcal{V}(S)$. An algebraic set $V$ is {\it reducible} if it can be written as the union of two algebraic sets different from $V$.  An algebraic set $V$ is a {\it variety} if it is not reducible.  For instance, a `cross', as in Example~\ref{ex: 2}, is a reducible algebraic set, whereas the circle in Example~\ref{ex: 1} is an algebraic variety (also see Figure~\ref{fig:IntroCircle}).

\begin{Definition}[Vanishing ideal]\label{def:VanishingIdeal}
Let $V \subset \R^d$ be an algebraic set.  We denote by 
\begin{equation}\label{eq:Ideal}
I(V) :=\{ P\in \R[x_1, \dots, x_d]: \ P({ v})=0, \ \mathrm{for \; all }\;\;{ v}\in V \}
\end{equation}
the {\it vanishing ideal} of $V$. 
\end{Definition}
Thus, the vanishing ideal of $V$ collects all polynomials whose values are identically zero on $V$. Note that $ \mathcal{V}(I(V))=V$. More generally,  for any set $A\subset \R^d$,  $\mathcal{V}(I(A))$ is equal to the Zariski closure (see  Section~\ref{Sec:AlgebraicCurves} below for its definition) of $A$ (cf.~\cite[Proposition~1.2]{Hartshorne1977}). We provide one more example of an algebraic set, which was not covered in the introduction.
\begin{Example}[Low-rank matrices)]
\label{ex: 3} Let $\mathcal{M}_{d_1\times d_2}(\R)$ be the space of matrices of dimension $d_1\times d_2$ with real coefficients. As the determinant of a matrix $A\in \mathcal{M}_{d\times d}(\R)$ is a polynomial, the set of non-injective matrices
$$ \{ A\in \mathcal{M}_{d\times d}(\R): \ {\rm Ker}(A)\neq \{0\}\}=\{ A\in \mathcal{M}_{d\times d}(\R): \ \det(A)=0\}$$
is an algebraic subset of the real vector space $\mathcal{M}_{d\times d}(\R)$.  For $ r\in \{1, \dots, \min(d_1, d_2)\}$ the set 
$$ \{ A\in \mathcal{M}_{d_1\times d_2}(\R): \ {\rm dim}({\rm Range}(A)) \leq r\}=\{ A\in \mathcal{M}_{d_1\times d_2}(\R): \ {\rm Rank}(A)\leq r\}$$
is an algebraic set on the real vector space $\mathcal{M}_{d_1\times d_2}(\R)$. This can be seen by relating the rank of a matrix with its characteristic polynomial, $\det( \lambda I - A)$. The problem of estimating a high-dimensional low-rank matrix corrupted by noise has received a lot of attention in the recent statistics literature; it is often called the spiked model~\cite{johnstone2001distribution,Montanari21}. However, in our situation, we envisage simultaneous estimation of many low-rank matrices observed with error.

\end{Example}

\begin{Definition}[Zariski closure~\cite{Hartshorne1977}]\label{defn:Zariski}
A set $U\subset \R^d$ is closed in the Zariski topology $\zar $ if and only if it is algebraic. The closure of a set $U\subset \R^d$ in the Zariski  topology, denoted by $\overline{U}^{\zar}$, is the smallest algebraic set containing $U$, i.e., 
$$ \overline{U}^{\zar} :=\bigcap \; \{A: \ U\subset A\;\; {\rm and }\;\; A \; \mathrm{is \; algebraic}\}.$$ The Zariski topology on $\R^d$ is defined by taking these closed sets as the closed sets of the topology.
\end{Definition}

Suppose that $\{\theta_i\}_{i\ge 1} \subset \R^2$ is a sequence of distinct points that lies on the upper semicircle of the unit circle (i.e., the $y$-coordinate is always $\ge 0$). The Zariski closure of this sequence is the \emph{entire} unit circle, not just the upper half; thus the Zariski topology cannot ``see'' the difference between a semicircle and the full circle.

Recall  model~\eqref{Model}. The Zariski closure of the sequence $\{\theta_i\}_{i\ge 1} \subset \R^d$ is the main object of interest in this paper and was denoted by $\mathcal{A}$ in~\eqref{eq:A}, which could be a curve, surface, or higher-dimensional algebraic set.\footnote{Note that it is important to look at the Zariski closure of the entire sequence  $\{\theta_i\}_{i\in \N}$ and not that of the finite set of points $\{\theta_i\}_{i =1}^{n}$---the Zariski closure of the finite set $\{\theta_i\}_{i\leq n}$ is itself.}  Equivalently,  $\mathcal{A}$ is the set of zeroes of the ideal   
$$ {\rm I}_\infty :=  \{ P\in \R[x_1, \dots, x_d]: \ P(\theta_i)=0,\quad \mathrm{for \; all } \;\; i\in \N\}$$
of polynomials vanishing at $\{\theta_i\}_{i\ge 1}$
The Hilbert basis theorem implies that every ideal of polynomials is finitely generated.\footnote{The ideal $I$ generated by polynomials $P_1,\ldots, P_k$ can be expressed as $I = \{P \in \R[x_1,\ldots, x_d]: P = \sum_{i=1}^k Q_i P_i, \;\; \mbox{for} \;\; Q_i \in \R[x_1,\ldots, x_d]\}$. In other words, every element of 
$I$ is a finite $\R[x_1,\ldots, x_d]$-linear combination of the generators $P_1,\ldots, P_k$.} Thus, finding any finite set of generators is enough to recover the ideal ${\rm I}_\infty$. Hence, $\mathcal{A}$ can be recovered from some finitely many polynomials $P_1, \dots, P_k \in \R[x_1, \dots, x_d]$, for some $k$, such that 
${\rm I}_\infty=\langle\,\{P_1, \dots, P_k\}\, \rangle $ and the following relationships hold:
\begin{equation}
    \label{eq:AandGenerators}
    \mathcal{A}=\mathcal{V}({\rm I}_\infty)=\mathcal{V}(\langle\, \{P_1, \dots, P_k\}\,\rangle)=\mathcal{V}(\{P_1, \dots, P_k\} ).
\end{equation} 
Hence, we only need to estimate the finite-dimensional vector spaces  of vanishing polynomials
\begin{equation}
    \label{eq:Ig}
    {\rm I}_g := \{ P\in \R_{\leq g}[x_1, \dots, x_d]: \ P(\theta_i)=0\quad \mathrm{for \; all } \;\; i\in \N\}, \quad \text{for } g\in \N,
\end{equation}
which for any
$g \ge \max_{i=1,\ldots, k}\mathrm{deg}(P_i)$ generates the ideal ${\rm I}_\infty$ of the sequence $\{\theta_i\}_{i \in \N}$, i.e., $ {\rm I}_\infty = \langle\, {\rm I}_g\, \rangle$. 
This motivates the following definition.
\begin{Definition}[Minimum degree of $\{\theta_i\}_{i\in \N}$] \label{def:minDegree}
    The smallest $g$ such that ${\rm I}_\infty=\langle \, {\rm I}_g \, \rangle$ is called the minimum degree of $\{\theta_i\}_{i\in \N}$ and is denoted by $g^*$. We note that $\mathcal{V}( {\rm I}_g )=\mathcal{A}$,  for all $g\geq g^*$ (see \cite[Proposition~1.2]{Hartshorne1977}). Further, $g^* \le \max_{i=1,\ldots, k} \mathrm{deg}(P_i)$. 
\end{Definition}

\subsection{The Vandermonde matrix and the Veronese map} 
In the sequel we will use the multi-index notation, i.e., for ${\bf i}=(i_0, i_1, \dots, i_d)\in \N^{d+1}_0$ and $x=(x_1, \dots, x_d)\in \R^d$ we write
\begin{equation}\label{eq:Multi-Idx}
\tilde{x} := (1, x_1, \ldots, x_{d}) \in \R^{d+1} \qquad \mbox{and} \qquad \tilde{x}^{\bf i}:=  1^{i_0} x_1^{i_1}\cdots x_d^{i_d} \in \R
\end{equation}
with the convention $0^0=1$.  Similarly, we consider a subscript version of the multi-index notation, i.e., $a_{\bf i}:= a_{i_0, i_1, \dots, i_d} \in \R. $ For a degree $g\in \N$, let
$\leq_g $ be a total order on the set 
\begin{equation}\label{eq:B-gN}
\mathcal{B}_{d,g} := \Big\{{\bf i}=(i_0, i_1, \dots, i_d)\in \N^{d+1}_0: \  \sum_{j=0}^d i_j= g\Big\}
\end{equation}
and 
\begin{equation}\label{eq:kappa}
\nelem := |\mathcal{B}_{d,g}| = \binom{d+g}{d}
\end{equation}
denote its cardinality.
This total order defines a \emph{monomial order}, which helps characterize polynomials through their coefficients.
\begin{Definition}[Veronese map]\label{defn:Veronese}
Fix $g,d\in \N$.   We call the following mapping the Veronese map~\cite[p.~62]{perrin2008algebraic}: For $\y = (y_1,\ldots, y_d) \in \R^d$,
\begin{equation}
    \label{eq:DefPhi}
      \phi \equiv \phi_{d,g} : \y \;\mapsto \; \phi(y) := (\tilde{ y}^{{\bf i}})_{{\bf i}\in \mathcal{B}_{d,g} }=(y_1^{i_1} \cdots y_d^{i_d})_{(i_0, \dots, i_d)\in \mathcal{B}_{d,g}} 
\in \R^{\nelem},
\end{equation}
where the latter must be understood as
a column vector in $\R^{\nelem}$ with entries ordered according to the total order $\leq_g $.   
\end{Definition}

\begin{Definition}[Vandermonde matrix]\label{defn:Vandermonde}
    The (multivariate) Vandermonde matrix of degree $g\in \N$ (see \cite{breiding2018learning}) of a set  $\{x_1, \dots, x_n\} \subset \R^d$ is 
$$  \mathbb{V}_g(x_1, \dots, x_n):=\left( \begin{array}{ccc}
     \phi(x_1), &
     \cdots,&
      \phi(x_n) 
\end{array} \right)^\top \in \mathcal{M}_{n \times \nelem}(\R).$$
\end{Definition}

For instance,  the Vandermonde matrix of degree $g = 2$ at $\left\{x_1, ...,x_n \right\} \subset \mathbb{R}^2$, with $x_i = (x_{i,1},x_{i,2})$, is
$$
\begin{pmatrix}
1 \; & \; x_{1,1} \; &\; x_{1,2} \; &\; x_{1,1}x_{1,2} \; & \; x_{1,1}^2 \; & \; x_{1,2}^2\\
\vdots & \vdots & \vdots & \vdots & \vdots& \vdots\\
1 \; & \;x_{n,1} \; & \;x_{n,2} \; & \; x_{n,1}x_{n,2} \; & \;x_{n,1}^2 \; & \;x_{n,2}^2
\end{pmatrix} \in \mathcal{M}_{n \times 6}(\R).
$$

\begin{Remark}[Vandermonde matrix and kernel PCA]
\label{rem: kerPCA}

The construction of the Vandermonde matrix (recall \eqref{eq:DefPhi}) and performing kernel PCA both involve an embedding into a higher dimensional space. In the case of interest for this paper, a simple computation establishes that the kernel
\(
 K(\x, \y) := (1+ \x^\top\y)^g 
\)
is, up to multiplicative combinatorial coefficients, the Euclidean inner product between the rows---$\phi(x)$ and $\phi(y)$---of the Vandermonde matrix of degree $g$. For example, when $d=g=2$, 
\begin{align*}
 (1+ \x^\top\y)^2 &= 1 + 2x_1y_1 + 2x_2 y_2 + 2 x_1x_2 y_1y_2 + x_1^2y_1^2 +  x_2^2y_2^2\\
 & =\big(1, \sqrt{2} x_1,\ \sqrt{2} x_2,\ \sqrt{2}x_1x_2, \ x_1^2, \ x_2^2\big)^\top  \big(1, \sqrt{2} y_1,\ \sqrt{2} y_2,\ \sqrt{2}y_1y_2, \ y_1^2,\ y_2^2\big).
\end{align*}

\end{Remark}

Note that the set $ {\rm I}_g$ (see~\eqref{eq:Ig}) is in bijective correspondence with 
\begin{equation}\label{eq:J_g}
{\rm J}_g :=\Big\{ (a_{\bf i})_{{\bf i}\in \mathcal{B}_{d,g} }:  \sum_{{\bf i}\in \mathcal{B}_{d,g}} a_{\bf i}{\tilde{\theta}_n}^{{\bf i}}=0\ \quad \mathrm{for \; all } \;\; n\in \N\Big\}, 
\end{equation}
which expresses the polynomials in $ {\rm I}_g$ via their coefficients (as vectors in $\R^\nelem$). This equivalent way of representing polynomials will be useful in practice.

\section{Estimation strategy}\label{Sec:Algebra}
\subsection{The failure of the naive approach}
As noted in the Introduction, the kernel of the Vandermonde matrix constructed from the observations $X_i$'s does not yield the coefficients of the polynomials defining the algebraic set $\mathcal{A}$. In this section, we make this precise. We begin by introducing some notation. Proofs of all the results stated in this section are provided in Appendix~\ref{Appendix:ProofSect3}.

\begin{Definition}[Moment matrix of order $g$]
For any distribution $\mu \in \mathcal{P}_{2g}(\R^d)$, we define the moment matrix of $\mu$ of order $g$ as 
\begin{equation}\label{eq:M_g}
 \mathbb{M}_{g}(\mu ):=\E[\phi(Y) \phi(Y)^\top],
\end{equation}
where $Y \sim \mu$ and $\phi \equiv \phi_{d,g}$ is the Veronese map (as defined in~\eqref{eq:DefPhi}).

\end{Definition}

Let $\nu_n := \frac{1}{n} \sum_{i=1}^n \delta_{\theta_i}$ be the empirical distribution of the first $n$ latent points. The moment matrix of $\nu_n$ of order $g$  will play a crucial role in our analysis, and is defined as: 
\begin{equation}
    \label{eq:defM}
     \mathbb{M}_{g}(\nu_n) := \frac{1} {n}\sum_{i=1}^{n}  [\phi(\theta_i)] \phi(\theta_i)^\top \in \mathcal{M}_{\nelem\times \nelem}(\R).
\end{equation}
Note that the kernel of the moment matrix $\mathbb{M}_{g}(\nu_n)$ is the same as that of the Vandermonde matrix $\mathbb{V}_g(\theta_1, \dots, \theta_n)$ (recall Definition~\ref{defn:Vandermonde}). The kernel of  $\mathbb{V}_g(\theta_1, \dots, \theta_n)$ is the set 
\begin{equation}\label{eq:J_g_n}
    {\rm J}_g^{(n)}\equiv {\rm J}_g(\theta_1, \dots, \theta_n) :=\Big\{ (a_{\bf i})_{{\bf i}\in \mathcal{B}_{d,g} }:  \sum_{{\bf i}\in \mathcal{B}_{d,g}} a_{\bf i}{\tilde{\theta}_j}^{{\bf i}}=0\ \quad \forall\,  j=1, \dots, n\Big\},
\end{equation}
(recall~\eqref{eq:Multi-Idx}) which is in bijective correspondence with the set 
\begin{equation}
    \label{eq:Ign}
    {\rm I}_g^{(n)}\equiv {\rm I}_g(\theta_1, \dots, \theta_n):=\{ P\in \R_{\leq g}[x_1, \dots, x_d]:\ P(\theta_j) =0\quad \forall \, j=1, \dots, n\}.
\end{equation}
More precisely, $ (a_{\bf i})_{{\bf i}\in \mathcal{B}_{d,g} } \in {\rm J}_g^{(n)} $ if and only if the polynomial $Q(x)=\sum_{{\bf i}\in \mathcal{B}_{d,g}} a_{\bf i}\tilde{x}^{{\bf i}}$ belongs to  ${\rm I}_g^{(n)}$. This fact is useful for the estimation of $\mathcal{A}$ due to the following lemma, which connects the unknown algebraic set $\mathcal{A}$ with the linear subspace ${\rm I}_g^{(n)}$. 
\begin{Lemma}[Stabilization of ${\rm I}_g^{(n)}$]\label{Lemma:StableIdeal}
    For any $g\in \N$ there exists\footnote{Under mild assumptions and in a setting slightly different from ours, \cite{pauwels2021data} provides a precise description of the integer $n_g$ linking it to the rank of the moment matrix $\mathbb{M}_{g}(\nu_n)$. } $n_g\in \N$ such that for all $n\geq n_g$, ${\rm I}_g^{(n)}={\rm I}_g.$ Moreover, if $g\geq g^*$ (recall Definition~\ref{def:minDegree}) then ${\rm I}_\infty=\langle\, {\rm I}_g\, \rangle$ and  $\mathcal{V}({\rm I}_g)=\mathcal{A}$.  
\end{Lemma}  
\begin{Remark}[Connecting $\mathcal{A}$ with ${\rm Ker}
(\mathbb{M}_g(\nu_n))$]\label{rem:Connect-A}
Suppose that $v_1,\ldots, v_k \in \R^\nelem$ form a basis for ${\rm Ker}(\mathbb{M}_g(\nu_n))$. Then, by the above lemma, the polynomials $P_1,\ldots, P_k$ constructed with the $v_i$'s as coefficients (respecting the total order $\le_g$), for $n \ge n_g$ and $g \ge g^*$, satisfy $$\mathcal{A} = \mathcal{V}(\langle P_1,\ldots, P_k \rangle).$$ 
\end{Remark}

The preceding remark suggests recovering 
$\mathcal{A}$ by estimating the kernel of the moment matrix $\mathbb{M}_g(\nu_n)$. Since the latent $\theta_i$'s are unobserved, a seemingly natural alternative is to replace $\nu_n$ with the empirical measure of the noisy observations, $\mu_n :=\frac{1}{n}\sum_{i=1}^n \delta_{X_i}$. 
The next result shows, however, that this naive substitution approach is bound to fail.

\begin{Corollary}\label{Coro:faliureOfDeterministic}
If $\Sigma$ in~\eqref{Model} is positive definite, the following holds with probability one: for all $g\in \N$, there exists  $n_g\in \N$ such that for all $n\geq n_g$, $\langle \, {\rm I}_g(X_1, \dots, X_n)\,  \rangle = \{0\} $ and  $\mathcal{V}( {\rm I}_g(X_1, \dots, X_n))=\R^d$. 
\end{Corollary}
\begin{proof}
    Apply Lemma~\ref{Lemma:StableIdeal} to the dataset $X_1, \dots, X_n$ and note that, due to the non-degenerate Gaussian noise, the Zariski closure of $\{X_n\}_{n\ge 1} $ is $\R^d$ with probability one.  
\end{proof}

\begin{Remark}[PCA: isotropic noise and \(g=1\)]
\label{rem:NoisyPCA}
The preceding result implies that the limiting moment matrix \(\mathbb{M}_g(\mu_n)\) has a trivial kernel for any \(g \in \mathbb{N}\). In the special case of PCA—corresponding to linear polynomials, i.e., \(g=1\)—this remains true.  When the noise is isotropic, i.e., \(\Sigma = \sigma^{2} I_{d}\), and \(g=1\), the eigenvalues of \(\E[\mathbb{M}_g(\mu_n)]\) and \(\mathbb{M}_g(\nu_n)\) differ only by a constant shift. Consequently, the true algebraic set \(\mathcal{A}\) (an affine subspace in this setting) can be recovered from the eigenvectors associated with the smallest eigenvalues of \(\mathbb{M}_g(\mu_n)\), yielding a consistent estimator. The isotropy of the noise is crucial: if the noise variances differ across directions, the empirical principal directions can diverge substantially from the population ones, even for large samples.  

For \(g>1\), however, the expression of \(\E[\mathbb{M}_g(\mu_n)] - \mathbb{M}_g(\nu_n)\) becomes more intricate (see Remark~\ref{ex: Circle} for an illustration), and the kernels of the two matrices no longer bear such a straightforward relationship.
\end{Remark}

Corollary~\ref{Coro:faliureOfDeterministic}  illustrates that the kernel of $\mathbb{M}_g(\mu_n)$ does not directly provide the sought information about the algebraic set $\mathcal{A}$. In a sense, one needs to `correct' the embedded data in a non-linear fashion. This is the object of the next subsection. 

\subsection{Debiasing the moment matrix $\mathbb{M}_g(\mu_n)$}

Owing to the noise in the observations $X_i$'s and the nonlinearity in the Veronese map (recall Definition~\ref{defn:Veronese}), 
the moment matrix computed on the data (i.e., $\mathbb{M}_g(\mu_n)$) will yield an inconsistent estimator of $\mathbb{M}_g(\nu_n)$, the moment matrix based on the $\theta_i$'s. The goal of this section is to find a closed-form expression for the incurred bias. Moreover, we want this expression to be computable, assuming knowledge of the noise covariance $\Sigma$. It will be convenient to work in tensor spaces, as this will allow us to express formulas more succinctly.
Let $[[g]]$ be the set of permutations of $\{1, \dots, g\}$.
The {\it symmetrization map} 
$$ \sym:(\R^{d+1})^{\otimes_g}\to (\R^{d+1})^{\otimes_g} $$
is the  unique linear  map such that
$$ \sym\left( \otimes_g(\x^{(1)}, \dots, \x^{(g)}) \right) := \frac{1}{g! }\sum_{\tau\in [[g]]}\otimes_g(\x^{(\tau(1))}, \dots, \x^{(\tau(g))}) , $$
for all $(\x^{(1)}, \dots, \x^{(g)})\in (\R^{d+1})^{g}$. 
The range ${\rm sym}((\R^{d+1})^{\otimes_g})$ is called the space of symmetric tensors. The symmetrization map is simply the operation that takes an ordinary tensor in $(\R^{d+1})^{\otimes_g}$ and ``averages'' it over all possible reorderings of its factors, thereby producing a tensor that is symmetric. The symmetrization map in the case of matrices, i.e., for $g=2$, has the particular form 
\(
\sym: A \mapsto (A+A^\top)/{2},  
\)
for any square matrix $A$.

It is well-known that the space of polynomials over $\R^d$ with degree $g$ (identified with $\R^\nelem$) and ${\rm sym}((\R^{d+1})^{\otimes_g})$ are isomorphic (cf.\ \cite{ComonetAl.SIMAA.2008}). For the sake of completeness, we formally prove this in Appendix~\ref{Appendix:ProofSect3}, where we provide an explicit construction of the isomorphism.
\begin{Lemma}
\label{lem: isomorp}
There exists a linear bijective map $\gamma_g: {\rm sym}((\R^{d+1})^{\otimes_g})  \longrightarrow  \R^\nelem$ such that 
\begin{equation}
    \label{gammag}
    \phi(\x)=\gamma_g \,\tilde{\x}^{\otimes_g} ,\quad \text{for all }\x\in \R^d.
\end{equation}

\end{Lemma}
To better understand the various notation used in the above result, let us illustrate the concepts for the case $g=d=2$.  In this case, for $ x=(x_1, x_2) \in \R^2$,  $ \tilde{x} =(1,x_1, x_2) \in \R^3$, and $$\tilde{\x}^{\otimes_2} = \begin{pmatrix}
1 & x_{1} & x_{2} \\
x_{1} & x_{1}^2 & x_{1} x_{2} \\
x_{2} & x_{2} x_{1} &  x_{2}^2 \\
\end{pmatrix} \qquad \mathrm{and}\qquad \gamma_2 \,\tilde{\x}^{\otimes_2} = (1,x_{1}, x_{2}, x_1 x_2, x_1^2, x_2^2)^\top.$$ In words, $\gamma_g \,\tilde{\x}^{\otimes_g}$ takes the $g$-tensor $\tilde{\x}^{\otimes_g}$ and spits out the ``upper triangular" part of the tensor in the form of a vector in $\R^\nelem$ with entries ordered by the total order $\le_2$. In our example above, the total order $\le_2$ is $$ (2,0,0) \leq_2 (1,1,0) \leq_2  (1,0,1)\leq_2 (0,1,1)\leq_2  (0,2,0) \leq_2 (0,0,2).$$

As the spaces $\mathcal{M}_{k\times k}(\R)$ and   $\R^k\otimes \R^k$ are isomorphic (see e.g.,~\cite[p.~10]{Ryan.2002}) there exists a linear isomorphism
\begin{equation}
    \label{ExampleTranspose}
    h_k:\R^k\otimes \R^k\to \mathcal{M}_{k\times k}(\R) \qquad \mathrm{ such \; that }  \qquad h_k({ u}\otimes { v}) = { u}\, { v}^\top   \;\; \mathrm{ for \;all }\;\; { u}, { v}\in \R^k.
\end{equation}
  Using this fact and Lemma~\ref{lem: isomorp}, we derive the following result which gives an alternate expression for $\mathbb{M}_g(\mu)$ (cf.~\eqref{eq:M_g}), for $\mu \in \mathcal{P}_{2g}(\R^d)$, based on the notation introduced above.

\begin{Lemma} 
\label{lem: RepSq}
Set $g\in \N$ and  $ \mu\in \mathcal{P}_{2g}(\R^d) $. Recall that for $x \in \R^d$, $\tilde x := (1,x) \in \R^{d+1}$. Then, 
 \(
 \mathbb{M}_g(\mu) = h_{\nelem}\,(\gamma_g\otimes \gamma_g)\, \left(\E_{X\sim \mu}[\tilde{X}^{\otimes_{2g}}] \right).
 \)
\end{Lemma}


In the next theorem we compute the expectation of the moment matrix for $\mu_n = \frac{1}{n} \sum_{i=1}^n \delta_{X_i}$ where $X_i$'s follow model~\eqref{Model}. In particular, the result highlights the impact of the (unobserved) noise $\{\eps_i\}_{i=1}^n$, and enables us to propose an unbiased estimator of $\mathbb{M}_g(\nu_n)$.

\begin{Theorem}\label{theorem:representation} Let $\{X_i, \eps_i, \theta_i \}_{i=1}^n$ be as in \eqref{Model},  $\mu_n=\frac{1}{n}\sum_{i=1}^n \delta_{X_i}$ and $\nu_n=\frac{1}{n}\sum_{i=1}^n \delta_{\theta_i}$. Let $\tilde{\theta}_i = (1,\theta_{i,1},\ldots, \theta_{i,d}) \in \R^{d+1}$ and 
$ \tilde{\Sigma}=\begin{pmatrix}
 1& 0\\
 0& \Sigma
\end{pmatrix}\in \mathcal{M}_{(d+1)\times (d+1)}(\R)$.
 Then,
\[\E\left[\mathbb{M}_g(\mu_n)\right] 
    \\
    = \frac{1}{n}\sum_{i=1}^n\sum_{k=0}^{ g} C_{2g,k} h_{\nelem}(\gamma_g\otimes \gamma_g)\left(  \sym \left( \tilde{\theta}_i^{\otimes_{(2g-2k)}} \otimes \tilde{\Sigma}^{\otimes_k} \right) \right),\]
where \(
C_{g,k} := \binom{g}{2k} \frac{(2k)!}{k!2^k}.
\)  Moreover,  
\begin{equation}
\label{eq: MomDec}
     \hat{\mathbb{M}}_g :=\frac{1}{n}\sum_{i=1}^n M_{n,i}, 
     \end{equation}
where
\begin{equation}
\label{eq: MomDec-2}
M_{n,i} :=\sum_{k=0}^{g} C_{2g,k}(-1)^k h_{\nelem}(\gamma_g\otimes \gamma_g)\left(  \sym \left( \tilde{X}_i^{\otimes_{(2g-2k)}} \otimes \tilde{\Sigma}^{\otimes_k} \right) \right),
\end{equation}
is an unbiased estimator of $\mathbb{M}_g(\nu_n)$, i.e.,  
  $ \E[ \hat{\mathbb{M}}_g ]=\mathbb{M}_g(\nu_n). $
\end{Theorem}

The symmetrization map is required to exploit the bijectivity between \emph{symmetric} tensors and homogeneous polynomials, on which our proof ultimately relies, see \cite[Prop.~2.8]{pereira2022tensor}.

\begin{Corollary}
[Explicit formula when \(d = 2\)]
\label{prop: EasyFormula}
When \(d = 2\), each entry of the moment matrix \(\mathbb{M}_g(\nu_n)\) takes the form
$\frac{1}{n} \sum_{i=1}^n \theta_{i,1}^K \theta_{i,2}^L,$
for \(K, L \in \mathbb{N}\). An unbiased estimator of this quantity, based on the observed data \(\{X_i\}_{i=1}^n\), is given by
\[
\frac{1}{n} \sum_{i=1}^n \sum_{j=0}^{\lfloor K/2 \rfloor} \sum_{\lambda=0}^{\lfloor L/2 \rfloor} 
C_{K,j} C_{L,\lambda} (-1)^{j+\lambda} \sigma^{2j+2\lambda} 
X_{i,1}^{K-2j} X_{i,2}^{L-2\lambda},
\]
where \(C_{K,j}\) and \(C_{L,\lambda}\) are suitable combinatorial coefficients. 
\end{Corollary}
While the above expression is relatively simple, the general case (\(d > 2\)) is more intricate. It relies on the connection between tensor moments and the entries of the moment matrix---a relationship that plays a central role in the theoretical analysis and proofs of our main results.

\begin{Remark}[The circle example]
\label{ex: Circle}
To illustrate the implications of Theorem~\ref{theorem:representation}  let us consider the estimation of the coefficients of the unit circle as discussed in Example~\ref{ex: 1}. We assume that 
$\theta_{i,1}^2 + \theta_{i,2}^2 =1$ and $\eps_i \sim \mathcal{N}(0, \sigma^2 I_2)$, for  $i = 1, \ldots, n.$ 
Then, the upper triangular entries of $\mathbb{M}_g(\mu_n)$ is given by
\[
 \frac1n
\begin{pmatrix}
\vspace{2mm}
n \; &\; \sum_{i=1}^n X_{i,1} &\; \sum_{i=1}^n X_{i,2} &\; \sum_{i=1}^n X_{i,1}X_{i,2} &\; \sum_{i=1}^n X_{i,1}^2 &\; \sum_{i=1}^n X_{i,2}^2  \\ \vspace{2mm}
 &\; \sum_{i=1}^n X_{i,1}^2 &\; \sum_{i=1}^n X_{i,1}X_{i,2} &\; \sum_{i=1}^n X_{i,1}^2X_{i,2} &\; \sum_{i=1}^n X_{i,1}^3 &\; \sum_{i=1}^n X_{i,2}^2X_{i,1} \\
 \vspace{2mm}
 &  &\; \sum_{i=1}^n X_{i,2}^2 &\; \sum_{i=1}^n X_{i,1}X_{i,2}^2 &\; \sum_{i=1}^n X_{i,1}^2X_{i,2} &\; \sum_{i=1}^n X_{i,2}^3 \\
  \vspace{2mm}
 &  &  &\; \sum_{i=1}^n X_{i,1}^2X_{i,2}^2 &\; \sum_{i=1}^n X_{i,1}^3X_{i,2} &\;  \sum_{i=1}^n X_{i,1}X_{i,2}^3 \\
  \vspace{2mm}
 &  &  &  &\; \sum_{i=1}^n X_{i,1}^4 &\;  \sum_{i=1}^n X_{i,1}^2X_{i,2}^2 \\
  \vspace{2mm}
 &  &  &  & &\;  \sum_{i=1}^n X_{i,2}^4 \\
\end{pmatrix}.
\]
As $X_i = \theta_i +\eps_i$, using the above matrix, one can easily compute $\E\left[\mathbb{M}_g(\mu_n)\right]$ and compare it with $\mathbb{M}_g(\nu_n)$.  One can see that the terms in $\mathbb{M}_g(\mu_n)$ with even powers may give rise to a nonzero bias. In particular, setting 
$$
\bar\theta_{1}^j := \frac1n \sum_{i=1}^n \theta_{i, 1}^j, \;\; \mathrm{ for } \;\;j =1,2, \qquad  \text{ and }\qquad  \overline{\theta_1\theta_2} :=\frac1n \sum_{i=1}^n \theta_{i,1}\theta_{i,2},
$$
we get that the upper triangular part of the bias matrix (i.e., $\E\left[\mathbb{M}_g(\mu_n)\right] - \mathbb{M}_g(\nu_n)$) is given by 
\[
\sigma^2 \begin{pmatrix}
0 & \;0 &\; 0 &\; 0 & \;1 & \; 1 \\
& \; 1 &\; 0 &\; \bar\theta_2 &\; 3 \bar \theta_1 &\;  \bar\theta_1\\
&  &\; 1 &\; \bar\theta_1 &\; \bar \theta_2 &\;  3 \bar \theta_2\\
&  &  &\; \sigma^2 + \bar\theta_1^2 +  \bar\theta_2^2 &\; 3\overline{\theta_1\theta_2}  &\; 3\overline{\theta_1\theta_2} \\
&  &  &  &\; 6 \bar\theta_1^2 + 3 \sigma^2    &\; \;\sigma^2 + \bar\theta_1^2 +  \bar\theta_2^2 \\
&  &  &  &   &\; 6 \bar\theta_2^2 + 3 \sigma^2  \\
\end{pmatrix}.
\]
Assuming that the noise parameter $\sigma^2$ is known, the same kind of computations suggest to estimate the bias with 
\[
\frac{\sigma^2}{n} \begin{pmatrix}
0 & \; 0 & \;0 & 0 &1 & 1 \vspace{2mm} \\ 
& \; 1 &\; 0 & \sum_{i=1}^n X_{i,2} & 3  \sum_{i=1}^n X_{i,1} &   \sum_{i=1}^n X_{i,1} \vspace{2mm}\\
&  & \; 1 &  \sum_{i=1}^n X_{i,1} &  \sum_{i=1}^n X_{i,2} &  2 \sum_{i=1}^n X_{i,2} \vspace{2mm}\\
&  &  & - {n \sigma^2} + \sum_{i=1}^n X_{i,1}^2 + X_{i,2}^2  & 3 \sum_{i=1}^n X_{i,1}X_{i,2} & 3 \sum_{i=1}^n X_{i,1}X_{i,2}\vspace{2mm} \\
&  &  &  & \quad 6  \sum_{i=1}^n X_{i,1}^2 - 3 {n \sigma^2} \quad   & - n\sigma^2 + \sum_{i=1}^n X_{i,1}^2 + X_{i,2}^2 \vspace{2mm} \\
&  &  &  &   & 6 \sum_{i=1}^n X_{i,2}^2 - 3 {n  \sigma^2}\vspace{2mm}  \\
\end{pmatrix}.
\]
\end{Remark}

\subsection{The estimation procedure}\label{subsec:estimationProc}
Our estimation strategy for recovering the latent algebraic set $\mathcal{A}$ is outlined in Algorithm~\ref{alg:two}---it is essentially as simple as implementing PCA via singular value decomposition (SVD), except for a prior debiasing step to compute $\hat{\mathbb{M}}_g$ (cf.~Remark~\ref{rem:Connect-A}). Theorem~\ref{theorem:representation} is the key step that yields an unbiased estimator of the moment matrix $\mathbb{M}_g(\nu_n)$.

\begin{algorithm}
\KwInput{${X_1, \ldots, X_n}$ (data); $\Sigma$ (noise covariance); $g$ (degree); $r_n$ (cutoff); }
\KwOutput{$\hat q$ (the coefficients of the polynomials in the generator).}
 Compute $\hat {\mathbb{M}}_g $ based on ${X_1, \ldots, X_n}, g$ and $\Sigma$\Comment*[r]{Get the debiased moment matrix (Thm.~\ref{theorem:representation})}
  $\Lambda, U\gets$ SVD($\hat {\mathbb{M}}_g)$\Comment*[r]{Extract the eigenvalues and eigenvectors of the matrix $ \hat {\mathbb{M}}_g$ (SVD)}
  $ \hat k \gets \# \{\lambda_k< r_n\}$\Comment*[r]{ As $\Lambda = \operatorname{diag}(\lambda_1, \ldots, \lambda_{\nelem})$ and $\lambda_1 \le \cdots \le \lambda_{\nelem}$}
 $\hat q \gets [u_1, \ldots, u_{\hat k} ]$\Comment*[r]{ $u_k$ is the eigenvector of $\hat {\mathbb{M}}_g$ corr. to $\lambda_k$}
\caption{\label{alg:two}Algorithm to estimate coefficients of the generating polynomials of $\mathcal{A}$.}
\end{algorithm}

Note that our algorithm requires us to specify a choice of the degree $g$ and the noise covariance matrix $\Sigma$. Typically, $g$ will be small; but for the results on the recovery of the algebraic set in Section~\ref{subsection:recoverSet} we need $g \ge g^*$---the minimum degree of the sequence $\{\theta_i\}_{i \in \N}$ (see Definition~\ref{def:minDegree}). We assume throughout this paper that $\Sigma$ is known.

\begin{Remark}[Identifiability of the noise covariance matrix]\label{identifiable}
If $\mathcal A \ne \mathbb R^{d}$, then in the isotropic case, i.e., $\Sigma=\sigma^{2}I_d$, the variance parameter $\sigma^{2}$ is, in principle, identifiable.
Informally, suppose that the data admit two representations: 
$X_{i}=\theta_{i}+\varepsilon_{i}$ and $X_{i}=\tilde{\theta}_{i}+\tilde{\varepsilon}_{i}$, with $\varepsilon_{i}\stackrel{\text{i.i.d.}}{\sim}\mathcal N(0,\sigma_{1}^{2}I_d)$, $\tilde{\varepsilon_{i}}\stackrel{\text{i.i.d.}}{\sim}\mathcal N(0,\sigma_{2}^{2}I_d)$ and $\sigma_{1}^{2}<\sigma_{2}^{2}$ where $\{\theta_i\}_{i \ge 1}$, $\{\tilde{\theta}_i\}_{i \ge 1}$ belong to an algebraic set $\mathcal{A} \ne \mathbb R^{d}$. Under mild conditions, the empirical measure $\frac1n\sum_{i=1}^{n}\delta_{X_{i}}$ converges (along a subsequence) to a limit $\mu$ that solves the heat equation at both times $\sigma_{1}^2$ and $\sigma_{2}^2$ for two initial distributions $\nu$ and $\tilde{\nu}$ supported on $\mathcal A\subsetneq\mathbb R^{d}$, i.e., $\mu= P_{\sigma_{1}^2}*  \nu $ and $\mu= P_{\sigma_{2}^2} *\tilde\nu $, where $ P_{t}$ denotes the heat semigroup (the convolution with the Gaussian kernel with variance $t$). Since $\{P_t\}_{t\geq 0} $ is a semigroup, it  follows that $$ P_{\sigma_{1}^2} * \nu= \mu= P_{\sigma_{2}^2} * \tilde\nu = P_{\sigma_{1}^2} * P_{\sigma_{2}^2-\sigma_{1}^2}*\tilde\nu.   $$
As the heat semigroup is injective, we get $\nu=P_{\sigma_{2}^2-\sigma_{1}^2}\tilde\nu$, which is a measure with support $\R^d$. This  contradicts the fact that  $\mathcal{A} \neq \R^d$, and hence  the variance parameter $\sigma^2$ is uniquely determined.
\end{Remark}
To implement our algorithm we need to choose and fix a total order $\le_g$ on $\mathcal{B}_{d,g}$ (see~\eqref{eq:B-gN}). Such an order determines the construction of the (debiased) moment matrix and underlies the order of  the entries of the vectors returned  by Algorithm~\ref{alg:two}. The total order $\leq_g$ induces a ranking on the elements of  $\mathcal{B}_{d,g}$ which can be represented by the bijection $\tau:\mathcal{B}_{d,g}\to \{1, \dots, \nelem\}$ defined as $\tau: {\bf i} \mapsto \vert \{ {\bf j }\in \mathcal{B}_{d,g}:\  {\bf j } \le_g {\bf i }\} \vert.$
Thus, $\tau^{-1}$) specifies  the exact monomial corresponding to each entry of the $\nelem$-dimensional vector of coefficients returned in the list $\hat q$ (from Algorithm~\ref{alg:two}). Each column of the returned $\hat q$ contains the coefficients of a polynomial in the recovered generator $\{\hat P_1,\ldots, \hat P_{\hat k}\}$, i.e., $\hat q$ gives rise to the  polynomials 
\[
\hat P_j(x) := \sum_{{\bf i} \in \mathcal{B}_{d,g}} (u_j)_{\tau ({\bf i })} (\tilde x)^{{\bf i}} , \qquad \mbox{for} \;\;1\le j\le \hat k. 
\]

Given an estimated set of coefficients for the fitted polynomials, visualizing the corresponding algebraic set requires solving the associated system of polynomial equations (see Section~\ref{subsection:recoverSet}). This task can be addressed using numerical algebraic geometry tools, such as the homotopy-continuation solver Bertini \cite{bates2013,harris2021}.
If one wishes to find a point on the algebraic set that is closest to a given sample point, the problem reduces to a nearest-point projection, which in turn becomes a polynomial optimization problem. The worst-case complexity of this problem is analyzed in \cite{draisma2016}.

\section{Estimation accuracy of the debiased moment matrix $ \hat{\mathbb{M}}_g$ and its kernel}
\label{sec: Cons+CLT}
In this section we study the limiting behavior of the debiased moment matrix $\hat{\mathbb{M}}_g$ (Section~\ref{Sec:Asym-hat-M_g}). We further use the thresholded eigenvalues of $\hat{\mathbb{M}}_g$ to estimate the zero eigenvalues of population moment matrix $\mathbb{M}_g(\nu_n)$ (Section~\ref{subsec:AnalysisEigen}). We use this to estimate the kernel of $\mathbb{M}_g(\nu_n)$ by considering the eigenspace of $\hat{\mathbb{M}}_g$ corresponding to the ``near-zero'' eigenvalues (Section~\ref{sec:Est-Kernel}). Finally, we end this section with a result on the uniform convergence on compacta of the estimated polynomials, indexed by vectors in ${\rm J}_g^{(n)}$ (see~\eqref{eq:J_g_n}), in a suitable topology (Section~\ref{SubSec:EstimationVanasingPoly}).

All the proofs of this section are deferred to Appendix~\ref{Appendix:ProofSect4}. Unless stated, in this section we consider any degree $g \in \N$.

\subsection{Asymptotic analysis of $ \hat{\mathbb{M}}_g$}\label{Sec:Asym-hat-M_g}
Our first result shows that the unbiased estimator $\hat{\mathbb{M}}_g$, see~\eqref{eq: MomDec}, of $\mathbb{M}_g(\nu_n)$ is consistent. Furthermore, we obtain a limiting distribution result describing the fluctuations of $ \hat{\mathbb{M}}_g$ around its mean, as the sample size $n \to \infty$. Let us start by stating our main assumption. 
\begin{Assumption}\label{Assumption:CLT}
Recall the definition of $M_{n,i}$ from~\eqref{eq: MomDec-2}.   Assume the following two conditions:
   \begin{enumerate}
       \item  there exists a matrix $\mathfrak{S}$ such that   $ \frac{1}{n} \sum_{i=1}^n \Cov(\vec(M_{n,i})) \to \mathfrak{S}, $ as $n \to \infty$, and 
       \item for every $\eta>0$,  
       $ \frac{1}{n}\sum_{i=1}^n \| \theta_i\|^{4g} \mathds{1}_{\{\| \theta_{i}\|^{2g}> \sqrt{n}\eta\}}\to 0, \quad \textrm{ as } \;\; n \to \infty.$
   \end{enumerate}
\end{Assumption}
\begin{Remark}[On Assumption~\ref{Assumption:CLT}]
Assumption~\ref{Assumption:CLT} requires that the empirical measures $\nu_n :=\frac{1}{n}\sum_{i=1}^n \delta_{\theta_i}$ stabilize and has `light' tails. In particular, it holds in the following scenarios: (i) $ \{\theta_n\}_{n\in \N}$ is an i.i.d.\ sequence with finite $4g$-th order moment;  (ii) $ \{\theta_n\}_{n\in \N}$ is such that 
the associated empirical measure $\nu_n$ converges to some $\nu$ in the sense that $\nu_n\xrightarrow{w} \nu$ and  $\int \|\theta\|^{4g}d\nu_n(\theta)\to \int \|\theta\|^{4g}d\nu(\theta)$, as $n \to \infty$. 
\end{Remark}

\begin{Theorem}
    \label{Theorem:CLTMatrix}
    Let $\{X_i, \eps_i, \theta_i \}_{i=1}^n$ be as in~\eqref{Model}, and let $\nu_n :=\frac{1}{n}\sum_{i=1}^n \delta_{\theta_i}$. Assume further that Assumption \ref{Assumption:CLT} holds. Then, as $ n\to \infty$,
    $$  \|\hat{\mathbb{M}}_g -\mathbb{M}_g(\nu_n)\|_{{\rm Fr}}\xrightarrow{\ a.s.\ }0    \quad \text{  and  } \quad \sqrt{n}\left(  \vec(\hat{\mathbb{M}}_g)-  \vec(\mathbb{M}_g(\nu_n)) \right)\xrightarrow{\ w\ } \mathcal{N}(0, \mathfrak{S}). $$
\end{Theorem}
Theorem \ref{Theorem:CLTMatrix} shows that the debiased moment matrix $\hat{\mathbb{M}}_g$ is asymptotically Gaussian with mean $\mathbb{M}_g(\nu_n)$ and variance of order $ n^{-1}$. Its proof invokes the Lindeberg–Feller CLT, relying crucially on the independence of the $X_i$'s. This result forms the theoretical bedrock for the estimation procedures developed below. In Section \ref{subsec:AnalysisEigen} we leverage it to establish a dichotomy for the spectrum of $\hat{\mathbb{M}}_g$: each eigenvalue either converges to zero at the standard rate  $ n^{-{1}/{2}}$ or remains bounded away from zero.

\subsection{Asymptotic analysis of the eigenvalues of $ \hat{\mathbb{M}}_g$}\label{subsec:AnalysisEigen}

In this subsection, we examine the asymptotic behavior of the eigenvalues of the debiased moment matrix $ \hat{\mathbb{M}}_g$. We begin by introducing some notation. Let $ {\bf H}_{\nelem}\subset \mathcal{M}_{\nelem\times \nelem}(\R)$ be the space of {\it symmetric matrices} of dimension $\nelem\times \nelem$ and  define, for each $k\in \{1, \dots, \nelem\}$, the function $\lambda_k:  {\bf H}_{\nelem}\to \R $ which maps a symmetric matrix $H \in {\bf H}_{\nelem}$ to its $k$-th smallest eigenvalue, counting multiplicities. That is, for any symmetric matrix $H \in {\bf H}_{\nelem}$, the eigenvalues satisfy $$\lambda_1(H) \le \lambda_2(H) \le \ldots \le \lambda_\nelem(H).$$
Weyl’s inequalities (see \cite[Eq.~(1.63)]{tao2012topics}) imply that, for any $k \in \{1, \dots, \nelem\}$, 
$$ \vert \lambda_k(H_1)-\lambda_k(H_2)\vert \leq \max\left( \lambda_\nelem(H_1-H_2), \lambda_\nelem(H_2-H_1)\right)=: \|H_1-H_2\|_2.$$
Thus, the absolute difference between the $k$-th smallest eigenvalues of two symmetric matrices is controlled by the spectral norm (or operator 2-norm) of their difference. This result, in turn, allows us to derive the following proposition as a direct application of Theorem~\ref{Theorem:CLTMatrix}.
\begin{Proposition}\label{Prop:ConsistencyEigenValues}
 Let $\{X_i, \eps_i, \theta_i \}_{i=1}^n$ be as in~\eqref{Model}, and let $\nu_n :=\frac{1}{n}\sum_{i=1}^n \delta_{\theta_i}$. Assume further that Assumption \ref{Assumption:CLT} holds. Then, for every $k\in \{1, \dots, \nelem\} $,
$$ |\lambda_k( \hat{\mathbb{M}}_g)-\lambda_k\left(\mathbb{M}_g(\nu_n)\right)|\xrightarrow{a.s.} 0 \qquad \text{ and } \qquad 
 |\lambda_k( \hat{\mathbb{M}}_g)-\lambda_k\left(\mathbb{M}_g(\nu_n)\right)|=\mathcal{O}_{\mathbb{P}}(n^{-\frac{1}{2}}) .$$
\end{Proposition}

This result shows that the eigenvalues of of \(\hat{\mathbb{M}}_g\) lie within $\mathcal{O}_{\mathbb{P}}(n^{-\frac{1}{2}})$ of the corresponding eigenvalues of \(\mathbb{M}_g(\nu_n)\). Note, moreover, that the kernel of \(\mathbb{M}_g(\nu_n)\) eventually stabilizes; this follows from Lemma~\ref{Lemma:StableIdeal} and the fact that \({\rm I}_g^{(n)}\) (see \eqref{eq:Ig}) is in isomorphic correspondence with ${\rm J}_g^{(n)} = {\rm Ker}(\mathbb{M}_g(\nu_n))$  (see~\eqref{eq:J_g_n}). Combining this correspondence with Proposition \ref{Prop:ConsistencyEigenValues} yields the dichotomy discussed above.

\begin{Lemma}\label{Lemma:StableKernel}  Let $\{X_i, \eps_i, \theta_i \}_{i=1}^n$ be as in~\eqref{Model}, and  set $k_g := {\rm dim}\left({\rm I}_g\right)$. Suppose that Assumption \ref{Assumption:CLT} holds.
    Then,
\begin{enumerate}
       \item  for every $k\leq  k_g$,  
\begin{equation}
    \label{eq:convergenceLambdaKernel}
    \lambda_k( \hat{\mathbb{M}}_g)\xrightarrow{a.s.} 0 \qquad {\rm and}\qquad  \lambda_k( \hat{\mathbb{M}}_g)=\mathcal{O}_{\mathbb{P}}\left(n^{-\frac{1}{2}}\right), \quad \mbox{and}
\end{equation} 
\item for $k>  k_g$, there exists $\Delta>0$ such that 
  \begin{equation}
      \label{eq:BoundBelowLambdas}
       \mathbb{P}\left( \text{there exists } n_0>0: \  \lambda_k( \hat{\mathbb{M}}_g)\geq \Delta \quad \text{for all }n\geq n_0\right)=1.
  \end{equation}
   \end{enumerate}
\end{Lemma}
In words, Lemma~\ref{Lemma:StableKernel} states that an eigenvalue \(\lambda_k(\hat{\mathbb{M}}_g)\), for $k\leq  k_g$, converges to zero at the rate \(n^{-1/2}\); the other eigenvalues remain bounded away from zero almost surely. Furthermore, the vanishing eigenvalues correspond precisely to the zero eigenvalues of \(\mathbb{M}_g(\nu_n)\), for sufficiently large $n$.

\subsection{Estimation of the kernel of $\mathbb{M}_g(\nu_n)$}\label{sec:Est-Kernel}

 Our goal is to recover the kernel of the population moment matrix $\mathbb{M}_{g}(\nu_n)$, which is expressed as ${\rm J}_g^{(n)}$ (cf.~\eqref{eq:J_g_n}). Note that Lemma~\ref{Lemma:StableIdeal} states that the decreasing chain of kernels $\{{\rm J}_g^{(n)}\}_{n\in \N}$ stabilizes as $n\to \infty$ meaning that there exists $n_g\in \N$ such that ${\rm J}_g^{(n)}={\rm J}_g^{(n+1)}$ for all $n\geq n_g$. We denote this (limiting) kernel by $ {\rm J}_g$.  
 
Although the debiased moment matrix $\hat{\mathbb{M}}_g$ and its  spectrum are asymptotically close to those of $\mathbb{M}_{g}(\nu_n)$, its kernel may not necessarily be nontrivial. 
Nevertheless, by Proposition~\ref{Prop:ConsistencyEigenValues}, we know that $\hat{\mathbb{M}}_g$ must have some eigenvalues close to zero. Consequently, the corresponding eigenspace can be expected to be close to ${\rm J}_g$. We analyze the convergence of eigensubspaces by introducing an appropriate notion of distance between them.

\begin{Definition}[Subspace distance]\label{defn:Subspace-dist}
Let $E,F$ be linear subspaces of a normed vector space $(V, \|\cdot\|)$. The chordal (or Frobenius‐projection) distance between $E$ and $F$ is defined as
$$ {\rm d}( E, F ) := \|\Pi_E- \Pi_F \|_{{\rm Fr}}= {\rm Trace}((\Pi_E- \Pi_F)(\Pi_E- \Pi_F)^\top),  $$
where $\Pi_E$ and $\Pi_F$ denote the orthogonal projection operators onto $E$ and $F$, respectively.
\end{Definition}
For a symmetric matrix $A\in \mathbf{H}_\nelem$, we define the eigenspace associated with its $j$'th smallest eigenvalue $\lambda_j(A)$ by
$$ {E}_{j}(A) :=\{  \x\in \R^\nelem: \ (A-\lambda_j(A)) \x=0\}, \quad \mbox{ for } \;j =1, \ldots,\nelem. $$
Fix a sequence $\{r_n\}_{n \in \N}$ such that $r_n\to 0$ and $n^{{1}/{2}} r_n \to \infty$, as $n \to \infty$. We estimate ${\rm J}_g$ by aggregating the eigenspaces of $\hat{\mathbb{M}}_g$ whose eigenvalues do not exceed $r_n$:
\begin{equation}\label{eq:hat-J_g}
\hat{\rm J}_g:=\bigoplus_{j\in j_n} E_j( \hat{\mathbb{M}}_g) \qquad \mbox{for} \qquad j_n :=\{ j\in \{1, \dots \nelem\}: \ \lambda_j( \hat{\mathbb{M}}_g)\leq r_n\}; 
\end{equation}
here $\oplus$ denotes the direct sum of subspaces.
Theorem~\ref{Theorem:main} below establishes the rate of convergence of the estimated eigenspace $\hat{\rm J}_g$.

\begin{Theorem}\label{Theorem:main}
Let $\{X_i, \eps_i, \theta_i \}_{i=1}^n$ be as in~\eqref{Model} and $\{r_n\}_{n\in \N}$ be a sequence such that $r_n\to 0$ and $n^{{1}/{2}} r_n \to \infty$, as $n \to \infty$. Moreover, suppose that Assumption~\ref{Assumption:CLT} holds. Then, as $n \to \infty$, we have: $$ {\rm d}\left(  \hat{\rm J}_g, {\rm J}_g \right) =\mathcal{O}_{\mathbb{P}}\left( n^{-\frac{1}{2}}\right).$$
\end{Theorem}
The above result essentially states that the estimated kernel $\hat{\rm J}_g$ of the debiased moment matrix $\hat{\mathbb{M}}_g$ is close to the limiting kernel ${\rm J}_g$, in the appropriate topology. Note that the result holds even when the eigenvalues of $\mathbb{M}_{g}(\nu_n)$ are not simple, ensuring that our analysis remains valid in the presence of spectral degeneracy.

A natural next step is to establish the convergence of the polynomials indexed by vectors in $\hat{\rm J}_g$ under a suitable topology, so that we can deduce the convergence of their zeros; see Appendix~\ref{SubSec:EstimationVanasingPoly} where we present a few consistency results in this direction.

\section{Estimation of the algebraic set $\mathcal{A}$}\label{subsection:recoverSet}

In this section we propose three methods to estimate the algebraic set $\mathcal{A}$, each suited to a different set of assumptions on $\mathcal{A}$. In our first method (Section~\ref{subsection:recoverSet-Method-1}) we take the common zero set of the estimated generators themselves; this approach yields an algebraic set and is consistent, but only under a strong assumption on $\mathcal{A}$ (see Theorem~\ref{Theorem:recoveringSet}). In our second approach (Section~\ref{Section:tubes}) we construct a semi-algebraic ``tube'' estimator for $\mathcal{A}$---we consider all points at which the fitted generators are approximately zero, producing a semi-algebraic set. This procedure is consistent under substantially weaker assumptions (see Theorem \ref{Theorem:completeRecovery}), though it does not yield an algebraic set. Our third approach (Section~\ref{sec: Prior}) incorporates prior structural knowledge on the algebraic set and projects the estimated generators onto a restricted polynomial class that encodes this prior knowledge. The resulting estimator remains algebraic; its consistency is established in Theorem \ref{Theorem:recoveringSet-Prior}.

\subsection{Recovering $\mathcal{A}$ from the common zero set of the fitted generators}\label{subsection:recoverSet-Method-1}

Let \(\hat{{\rm I}}_g\) denote the collection of polynomials whose coefficient vectors (ordered according to the monomial order \(\leq_g\)) belong to the estimated subspace \(\hat{{\rm J}}_g\), i.e.,
\begin{equation}\label{eq:hat-Ig}
\hat{{\rm I}}_g := \Big\{ P(x) = \sum_{{\bf i} \in \mathcal{B}_{d,g}} a_{\bf i} \, {\tilde{x}}^{{\bf i}} \in \mathbb{R}_{\leq g}[x_1, \dots, x_d] : (a_{\bf i})_{{\bf i}} \in \hat{{\rm J}}_g \Big\}.
\end{equation}
Once \({\rm I}_g\) (see~\eqref{eq:Ig}) is estimated by \(\hat{{\rm I}}_g\), a natural estimator for the algebraic set
\[
\mathcal{A}^{(g)} := \mathcal{V}({\rm I}_g) = \left\{ \mathbf{x} \in \mathbb{R}^d : P(\mathbf{x}) = 0 \;\; \text{for all } P \in {\rm I}_g \right\}
\]
is given by the estimated zero set
\begin{equation}\label{eq:A_ng}
\mathcal{A}^{(n,g)} := \left\{ \mathbf{x} \in \mathbb{R}^d : P(\mathbf{x}) = 0 \;\; \text{for all } P \in \hat{{\rm I}}_g \right\}.
\end{equation}

Figure~\ref{fig:IntroCircle} (leftmost plot) illustrates this estimator \(\mathcal{A}^{(n,2)}\) (shown in blue) for the ``circle'' example. Visually, the estimated set \(\mathcal{A}^{(n,2)}\) appears quite close to the true algebraic set \(\mathcal{A} \equiv \mathcal{A}^{(2)}\).\footnote{
Computing the common zero set of a system of multivariate polynomials can be computationally intensive. For two-dimensional problems, we evaluate over a fixed grid along one axis (e.g., the \(x\)-axis) and solve the resulting univariate polynomial equations. For higher dimensions, a similar strategy is employed. In more complex cases (e.g., when the kernel is not simple), numerical algebraic geometry software such as \texttt{Bertini} can be used.}
Our objective is to show that, for some \(g \geq 1\), the estimator \(\mathcal{A}^{(n,g)}\) converges to \(\mathcal{A}\) in a suitable sense.
Intuitively, one might expect \(\mathcal{A}^{(n,g)}\) to converge to \(\mathcal{A}^{(g)}\), which itself may or may not equal the true algebraic set \(\mathcal{A}\). The following remark clarifies a key requirement: for this approach to consistently recover the true set \(\mathcal{A}\), it is necessary to choose \(g = g^*\), where \(g^*\) is the minimal degree for which \(\mathcal{A} = \mathcal{V}({\rm I}_{g^*})\).

\begin{Remark}[When $g \ne g^*$]\label{Rem:g=g^*}
It is easy to see that for any $g < g^*$, we have $\mathcal{A}^{(g)} \ne \mathcal{A}$. For every $g\geq g^*$, Lemma \ref{Lemma:StableIdeal} implies $ \mathcal{A}^{(g)}=\mathcal{A}$. However, we provide a simple example to illustrate that if $g > g^*$ even then $\mathcal{A}^{(n,g)}$ may fail to converge to $\mathcal{A}$. Consider the algebraic set $\mathcal{A}= \{(x,y)\in \R^2:\ x=0\}$ for which $g^*=1$ (recall Definition~\ref{def:minDegree}). If we instead use degree $g=2$, the degree-2 component of the ideal is given by ${\rm I}_2={\rm span}\{x, xy,x^2\}$. Now suppose that, due to estimation error, we obtain $ \hat {\rm I}_2 :={\rm span}\{x, xy,x^2+1/n\} $, which is a small perturbation of ${\rm I}_2$. Although $ \hat {\rm I}_2$ converges to ${\rm I}_2$ as $n\to \infty$, the zero set $ \mathcal{A}^{(n,g)}$ is empty for every finite $n$. Hence, $\mathcal{A}$ cannot be recovered from $ \hat {\rm I}_2$ in this case. This example demonstrates that for our approach to consistently recover $\mathcal{A}$, it is necessary to select 
$g = g^*$.
\end{Remark}

We now formalize the notion of convergence of the sets $ \mathcal{A}^{(n,g)}$. As algebraic sets are closed (in the Euclidean topology) and might be unbounded, we use the notion of Painlevé-Kuratowski convergence (see Definition~\ref{def:PK}; also see \cite[Chapter~4]{RockafellarWets}) to study the convergence of the $\mathcal{A}^{(n,g)}$. The Painlevé-Kuratowski convergence can be metrized via the  distance 
$$ \mathbf{d}_{\mathrm{PK}}(A,B) := \int_{0}^\infty  d_{\mathrm{H}}(A\cap \mathbb{B}_t, B\cap \mathbb{B}_t   ) e^{-t} d \, t, \qquad \mbox{for} \;\;A,B\subset \R^d. $$
Here $\mathbb{B}_t$ denotes the ball with center $ 0$ and radius $t>0$ in $\R^d$ and $d_{\mathrm{H}}(A,B)$ is the Hausdorff distance between sets $A,B \subset \R^d$, i.e.,
$$
d_{\mathrm{H}}(A, B) :=\max\left( \sup_{x\in B}d(x, A), \sup_{x'\in A}d(x', B) \right),$$ where $d(x, A) := \inf_{a \in A} \|x - a\|$ is the distance from $x$ to its closest point in $A$.
The metric space $({\rm CL}_{\neq \emptyset}(\R^d), {\bf d}_{\mathrm{PK}})$, where 
${\rm CL}_{\neq \emptyset}(\R^d) :=\{ A\subset \R^d:  \overline{A}= A\neq \emptyset\}$,  is separable \cite{RockafellarWets}.

We recall some basic notions from real algebraic geometry (see \cite{Bochnak.1998.Book} for details).

An algebraic set $A\subset \R^d$ is {\it irreducible}---or, equivalently, a real algebraic variety---if, whenever $A$ can be written as a union of algebraic subsets $F\cup G$,  one must have $F=A$ or $G=A$. Every algebraic set $A$ can be uniquely decomposed as a finite union of irreducible  sets $A_1, \dots, A_M$ (called irreducible components of $A$) with the property that $A_i\not\subset \bigcup_{j\neq i} A_j$, for $i=1,\ldots, M$ (see~\cite[Theorem~2.8.3]{Bochnak.1998.Book}). 
The notion of dimension for an algebraic set is defined in a manner analogous to the dimension of a vector subspace; see~\cite[Definition~1.2]{perrin2008algebraic}.

\begin{Definition}[Dimension of an algebraic set]
The dimension of an algebraic set $A~\subset~\R^d$, ${\rm dim}(A)$, is the maximum length  $D$ of a chain $\emptyset \neq A_1\subsetneq  \dots \subsetneq A_D$ of distinct nonempty irreducible algebraic subsets of $A$ different from $A$. 
\end{Definition}
Observe that the above definition of dimension extends that of a linear subspace to any algebraic set. 
Let us illustrate the above notion for simple algebraic sets. Any circle has dimension $1$, since the only irreducible algebraic proper non-empty subsets of the circle are the singleton points (on the circle). Analogously, the dimension of the sphere in $\R^3$ is 2.
\begin{Definition}[Regular point]\label{defn:regular-pt}
Suppose an algebraic set $A \subset \R^d$ is irreducible and let $P_1, \dots, P_k$ be a set of generators of $ I(A) $, i.e., $ I(A)  = \langle \, P_1, \dots, P_k \,\rangle.$ A point  $x\in A$ is said to be {\it regular (or nonsingular)} if 
$$ {\rm dim} ({\rm sp}\{\nabla P_i(x): \ i=1, \dots,k\} ) =d-{\rm dim}(A) , $$
where ${\rm dim}(A)$ denotes the dimension of the algebraic set $A$. We denote the set of regular points of $A$ by ${\rm reg}(A)$. The value $d-{\rm dim}(A) $ is called the co-dimension of $A$.

If $A \subset \R^d$ is any algebraic set, then $x\in A$ is said to be regular if $x$ belongs to a single irreducible component $B$ of $A$ and it is regular point of $B$. 
\end{Definition}

Hence, a regular point of an algebraic set is one where the rank of the Jacobian matrix of its defining polynomials is the co-dimension of $A$. 
Around a regular point $x$, the algebraic set $A$ behaves like a $\mathcal{C}^\infty $ manifold of dimension ${\rm dim}(A)$, and the (standard) tangent space at $x$ is well-defined and agrees with the orthogonal complement 
$ ({\rm sp}\{\nabla P_i(x): \ i=1, \dots,k\} )^\perp $ (see \cite[Proposition 3.3.10]{Bochnak.1998.Book}). Now we are ready to state an important assumption that will be used in our main result in this subsection.

\begin{Assumption}[Complete intersection with same degree]\label{assumption:superStrong}
    We assume that $g$ is such that $\mathcal{A}=\mathcal{V}({\rm I}_{g})$ is irreducible and that 
    ${\rm dim}({\rm I}_g)+ {\rm dim}(\mathcal{A})=d$.\footnote{${\rm dim}({\rm I}_g)$ equals the number of linearly independent polynomials of degree $\le g$ vanishing at $\{\theta_i\}_{i\ge1}$ .} 
 \end{Assumption}

\begin{Remark}[On Assumption~\ref{assumption:superStrong}]\label{rem:Assump} Assumption~\ref{assumption:superStrong} is slightly restrictive. For instance, the circle $\mathcal{A} := \mathcal{V}(\{x,x^2+y^2+z^2-1\})$ in $\R^3$ does not satisfy Assumption~\ref{assumption:superStrong} as we explain below. Note that ${\rm dim}(\mathcal{A})=1$, and for $g\geq 2$, we have ${\rm dim}({\rm I}_g)\geq 5$ (note that $\{x,x^2+y^2+z^2-1, xy,xz, x^2 \} \subset {\rm I}_g$). This contradicts Assumption~\ref{assumption:superStrong}. 

Further, Assumption~\ref{assumption:superStrong} implies that $g$ must be $g^*$ 
and  that all nonzero polynomials of degree less than or equal to $g$, vanishing at $\mathcal{A}$, must have exact degree $g=g^*$.  To see this, we argue via contradiction. If the degree of $ P  \ne 0 \in {\rm I}_{g}$ is strictly smaller than $g$, for some $g$ such that Assumption~\ref{assumption:superStrong} holds, then the linearly independent polynomials $x_1 P(x), x_2 P(x), \ldots, x_d P(x)$ belong to ${\rm I}_{g}$. This implies that ${\rm dim}({\rm I}_{g}) = d$, which in turn implies, by Assumption~\ref{assumption:superStrong}, that $\mathcal{A}$ is the empty set. This contradicts the fact that $P$ vanishes at $\mathcal{A}$. Hence, $g^*$ is the unique degree for which Assumption~\ref{assumption:superStrong} can hold. Thus, Assumption~\ref{assumption:superStrong} indirectly assumes the exact knowledge of the degree $g^*$.
\end{Remark} 

The following theorem, proved in Appendix~\ref{Appendix:ProofSect5}, shows the consistency of the estimated algebraic set $\mathcal{A}^{(n,g)}$ to $\mathcal{A}$, under the Painlev\'{e}-Kuratowski (PK) topology. Moreover, the convergence locally around regular points happens at the parametric rate $n^{-1/2}$.

\begin{Theorem}\label{Theorem:recoveringSet}
Let \(\{X_i, \varepsilon_i, \theta_i\}_{i=1}^n\) follow the model in \eqref{Model}, and let \(\{r_n\}_{n \in \mathbb{N}}\) be a sequence such that \(r_n \to 0\) and \(n^{1/2} r_n \to \infty\) as \(n \to \infty\). Suppose Assumptions~\ref{Assumption:CLT} and \ref{assumption:superStrong} hold. Then:
\begin{enumerate}
    \item for every \(x_0 \in \operatorname{reg}(\mathcal{A})\), there exists a sequence \(\{x_n\}_{n \in \mathbb{N}}\), with \(x_n \in \mathcal{A}^{(n,g)}\), such that \(x_n \xrightarrow{\mathbb{P}} x_0\);
    
    \item any limit point in probability of a sequence \(\{x_n\}_{n \in \mathbb{N}}\), with \(x_n \in \mathcal{A}^{(n,g)}\), lies in \(\mathcal{A}\).
\end{enumerate}
If, in addition, \(\mathcal{A} = \overline{\operatorname{reg}(\mathcal{A})}\), then, as $n \to \infty$,
\[
\mathbf{d}_{\mathrm{PK}}(\mathcal{A}^{(n,g)}, \mathcal{A}) \xrightarrow{\mathbb{P}} 0.
\]
Moreover, for any \(x_0 \in \operatorname{reg}(\mathcal{A})\), there exists an open Euclidean neighborhood \(\mathcal{U}\) of \(x_0\) such that
\[
d_{\mathrm{H}}(\mathcal{A} \cap \mathcal{U},\; \mathcal{A}^{(n,g)} \cap \mathcal{U}) = \mathcal{O}_{\mathbb{P}}(n^{-1/2}).
\]
\end{Theorem}

The theorem asserts that the estimator $\mathcal{A}^{(n,g)}$ never ``misses'' any regular part of the algebraic set $\mathcal{A}$; conversely, any probabilistic limit of a sequence $x_n \in \mathcal{A}^{(n,g)}$ must lie in $\mathcal{A}$. Further, if $\mathcal{A}$ is the closure of its regular locus---i.e., it contains no isolated singular components---then the entire set $\mathcal{A}^{(n,g)}$ converges (in probability) in the Painlev\'{e}-Kuratowski probabilistic sense, not just pointwise. 
In addition, near any regular point, the estimator converges at the parametric rate—the same rate observed when estimating a linear subspace in PCA. For example, in the leftmost panel of Figure~\ref{fig:IntroCircle}, all points of the underlying algebraic set (the unit circle) are regular, allowing for convergence at the parametric rate.

\begin{Remark}[On the proof of Theorem~\ref{Theorem:recoveringSet}] The proof of Theorem~\ref{Theorem:recoveringSet}, given in Appendix~\ref{Appendix:ProofSect5}, relies on a local analytic argument centered around the regular points of the algebraic set $\mathcal{A}$. At such points, Assumption~\ref{assumption:superStrong} enables the application of the implicit function theorem, which locally parametrizes the algebraic set in terms of a smooth function. The key step then shows that small perturbations of the generators---obtained from the estimated kernel of the empirical moment matrix $\hat{\mathbb{M}}_{g}$---lead to perturbed zero sets that remain close to the true set in a uniform sense. 
Global consistency in the Painlevé–Kuratowski topology follows by combining the local arguments with the assumption that $\mathcal{A}$ is the closure of its regular part.
\end{Remark}

Note that the condition $\mathcal{A}=\overline{{\rm reg}(\mathcal{A})}$ in Theorem~\ref{Theorem:recoveringSet} is not automatic for real algebraic sets. A classical counter-example is the {\it Whitney umbrella} ${\mathcal{V}(x^2-y^2 z)}~\subset~\R^3$, which is irreducible, has all its regular points in the half-space $\{z\geq 0\}$, and contains the entire line $\{x=0,y=0,z<0\}$ of singular points. Thus, the closure of all its  regular points does not coincide with the whole algebraic set $\mathcal{A}$.

\bgroup

\subsection{Approximation by semi-algebraic sets}\label{Section:tubes}

Theorem~\ref{Theorem:recoveringSet} establishes that the algebraic set \(\mathcal{A}\) can be consistently estimated under Assumptions~\ref{Assumption:CLT} and \ref{assumption:superStrong}. However, as noted following the theorem, Assumption~\ref{assumption:superStrong} is quite strong. This motivates the development of an alternative estimation approach that remains valid under weaker conditions.

We propose approximating \(\mathcal{A}\) using semi-algebraic sets. A \emph{semi-algebraic set} (see \cite[Definition~2.1.4]{Bochnak.1998.Book}) is any set that can be written as a finite union and intersection of sets of the form
\[
\{ x \in \mathbb{R}^d : P(x) \leq 0 \} \quad \text{or} \quad \{ x \in \mathbb{R}^d : P(x) < 0 \},
\quad \text{for some } P \in \mathbb{R}[x_1, \dots, x_d].
\]

Let \({u}_1, \dots, {u}_{\hat{k}}\) be an orthonormal basis of the estimated kernel $\hat{{\rm J}}_g $ obtained from Algorithm~\ref{alg:two}, and let \(\hat{P}_1, \dots, \hat{P}_{\hat{k}}\) be the corresponding polynomials. For any \(\lambda > 0\), we define the semi-algebraic ``tube'' estimator by
\[
\mathcal{A}_{\lambda}^{(n,g)} := \left\{ x \in \mathbb{R}^d : |\hat{P}_i(x)| \leq \lambda \;\text{ for all } i = 1, \dots, \hat{k} \right\}.
\]
Thus, \(\mathcal{A}_{\lambda}^{(n,g)}\) is a semi-algebraic neighborhood of the zero set of the fitted generators. We consider a sequence \(\lambda_n \to 0\) as \(n \to \infty\). The center panel of Figure~\ref{fig:IntroCircle} illustrates the semi-algebraic estimator \(\mathcal{A}_{\lambda}^{(n,2)}\) for Example~\ref{ex: 2}. In this example, all points on the algebraic set are regular except for the singular intersection point at \((0, 0)\), where the two lines meet.

The next result, Theorem~\ref{Theorem:completeRecovery} (proved in Appendix~\ref{App:Thm-Semi-Alg}), demonstrates that this semi-algebraic tube estimator converges to the target algebraic set under no assumptions on \(\mathcal{A}\). Moreover, the convergence occurs in Hausdorff distance at rate \(\lambda_n\), uniformly over compact neighborhoods of each regular point.

\begin{Theorem}
\label{Theorem:completeRecovery}
Fix \( g \geq g^* \). Let \(\{X_i, \varepsilon_i, \theta_i\}_{i=1}^n\) follow the model in~\eqref{Model}, and let \(\{r_n\}_{n \in \mathbb{N}}\) and \(\{\lambda_n\}_{n \in \mathbb{N}}\) be sequences such that \(r_n \to 0\), \(\lambda_n \to 0\), and \(n^{1/2} r_n \to \infty\), \(n^{1/2} \lambda_n \to \infty\) as \(n \to \infty\). Suppose Assumption~\ref{Assumption:CLT} holds. Then,
\[
\mathbf{d}_{\mathrm{PK}}\left(\mathcal{A}, \mathcal{A}_{\lambda_n}^{(n,g)}\right) \xrightarrow{\mathbb{P}} 0.
\]
Moreover, for any \(x_0 \in \mathrm{reg}(\mathcal{A})\), there exists a Euclidean neighborhood \(\mathcal{U}\) of \(x_0\) such that
\[
d_{\mathrm{H}}\left(\mathcal{A} \cap \mathcal{U},\, \mathcal{A}_{\lambda_n}^{(n,g)} \cap \mathcal{U} \right) = \mathcal{O}_{\mathbb{P}}(\lambda_n).
\]
\end{Theorem}
\noindent

The condition \(n^{1/2} \lambda_n \to \infty\) ensures that the tube radius \(\lambda_n\) shrinks slowly enough to contain the stochastic variability in the data, while still concentrating tightly around the underlying algebraic set. Theorem~\ref{Theorem:completeRecovery} offers a practical guideline for selecting \(\lambda_n\): it should decay at a slower rate than the parametric benchmark \(n^{-1/2}\), to guarantee that every point in \(\mathcal{A}\) is eventually covered, but not too slowly, in order to preserve a sharp rate of convergence. A convenient and theoretically justified choice is $\lambda_n = (\log n)\,n^{-1/2}$. Under this setting, for any regular point \(x_0 \in \operatorname{reg}(\mathcal{A})\), there exists a Euclidean neighborhood \(\mathcal{U}\) of \(x_0\) such that
\[
d_{\mathrm{H}}\!\left(\mathcal{A} \cap \mathcal{U},\;
           \mathcal{A}^{(n,g)}_{\lambda_n} \cap \mathcal{U}\right)
    = \mathcal{O}_{\mathbb{P}}\!\left((\log n)\,n^{-1/2}\right).
\]
Furthermore, if \(\mathcal{A}\) is bounded and composed entirely of regular points, a standard compactness argument extends this local consistency to a global one:
\[
d_{\mathrm{H}}\!\left(\mathcal{A},\mathcal{A}^{(n,g)}_{\lambda_n}\right)
   = \mathcal{O}_{\mathbb{P}}\!\left((\log n)\,n^{-1/2}\right).
\]

\subsection{Estimation of $\mathcal{A}$ with known prior knowledge on its structure}
\label{sec: Prior}
In this subsection we focus on the estimation of the algebraic set $\mathcal{A}$ with some predefined structure. For instance, in many applications we may a priori  know that $\mathcal{A}$ is defined by a polynomial that factors into several lower-degree components. The following gives a simple example.

\begin{Example}[The `cross' example revisited] 
In Example~\ref{ex: 2} (also see the center plot of Figure~\ref{fig:IntroCircle}) the algebraic set $\mathcal{A}$ under consideration is the zero set of the polynomial $P_1(x,y) = x^2 - y^2 \in \R[x,y]$, which can be factorized as $(x-y)(x+y)$. Suppose we knew that the unknown polynomial factors as the product of two affine functions, i.e., $P_1 \in S := \{P(x,y) = U(x,y) V(x,y)  \in \R[x,y]:  U, V \in \R_{\le 1} [x,y]  \}$. Our goal is to recover $\mathcal{A}$ subject to this information. Our estimation procedure (in Algorithm~\ref{alg:two}) will, with probability 1, result in a non-degenerate conic which will not factor as the product of two affine functions. Usually, the conic that one recovers for this example is a hyperbola (see the recovered algebraic set in red in the leftmost plot of Figure~\ref{fig:Proj}). Given the estimated coefficient vector $u_1 \in \R^6$ of the fitted polynomial (obtained via Algorithm~\ref{alg:two}) we must instead find its nearest polynomial in $S$. Approximate factorization of this kind is a well-studied problem in algebraic geometry; see~\cite{kaltofen2008approximate} for an algorithm when the polynomial factors into two pieces. Their approach observes that multiplying two polynomials is linear in the coefficients of one factor, i.e., if $b$ collects the coefficients of $V$ and
$$ M(a)= 
\begin{pmatrix}
a_1 & a_2 &  a_3 & 0   & 0   &0 \\
0   & a_1 & 0    & a_3 & a_2 & 0 \\
0   &0    & a_1  & a_2 & 0   & a_3 
\end{pmatrix},
$$
with $a= (a_1,a_2,a_3)$ the coefficients of $U$, then 
$P=UV$ corresponds to $M(a)^\top b$. The closest factorization is therefore obtained by solving
\[
\inf_{a,b \in \R^3} \| 
M(a)^\top b - u_1\|.
\]
Although computationally heavier than our spectral step, this optimization delivers factors whose product lies in 
$S$ and is optimally close to the polynomial with coefficient vector $u_1$.  
\end{Example}
\begin{figure}[h!]
    \centering
\includegraphics[width=0.32\linewidth]{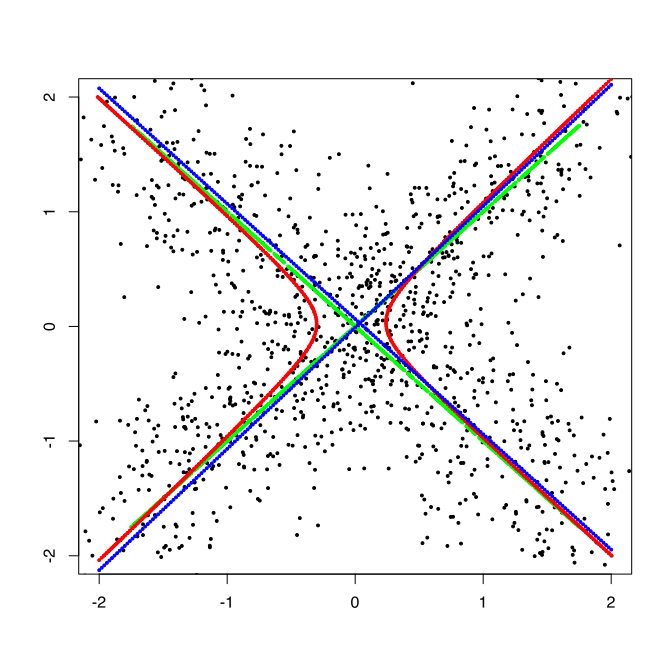}  
\includegraphics[width=0.32\linewidth]{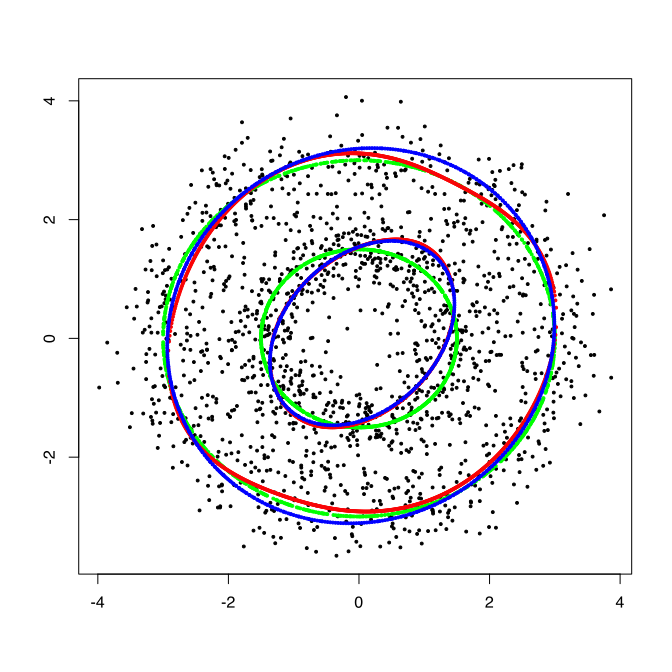}    
\includegraphics[width=0.32\linewidth]{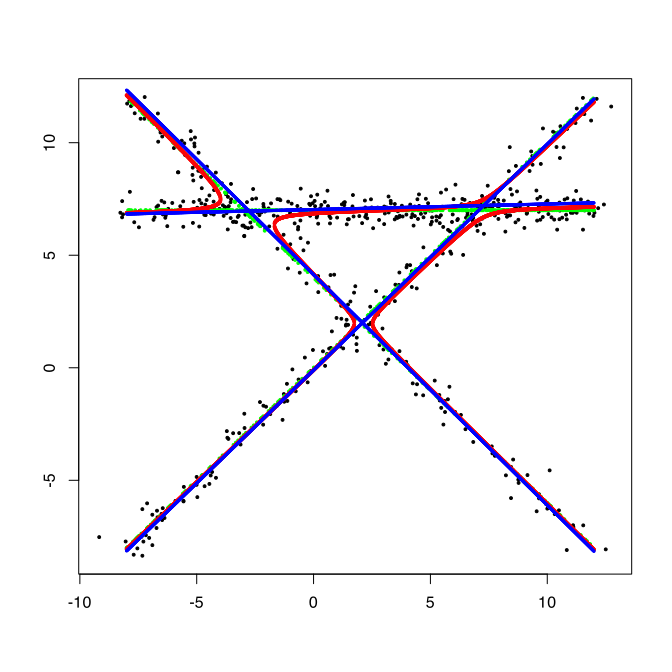}
     \fbox{ 
 \textcolor{green}{ $\bullet$} {\footnotesize $\theta_i$'s}\quad
 \textcolor{black}{$\bullet$} {\footnotesize $X_i$'s}\quad
 \textcolor{red}{\textemdash}{\footnotesize estimated algebraic set }
  \textcolor{blue}{\textemdash}{\footnotesize estimated algebraic set (proj. coeffs.) }}
    \caption{\label{fig:Proj}  In these experiments, the noise is $\sigma=0.4$. We exhibit in red the algebraic set corresponding to the zero locus of the estimated polynomial (via Algorithm~\ref{alg:two}). In dark blue, we show the output of the procedure incorporating the knowledge that the coefficients should factorize. The sample sizes from left to right are 600, 1200 and 600.  }
\end{figure}

More generally, suppose there exists a collection of polynomials $S\subset \R[x_1, \dots, x_d]$ such that $\mathcal{A}$ is the set of common zeros of one or more elements of $S$. Our principal interest lies is the case where every $P \in S$ factors into $m$ components of prescribed degrees $g_1, \dots, g_m$. Let $\tilde{S} \subset \R^\nelem$, where $g = g_1+\ldots+g_m$, denote the set of coefficient vectors (ordered according to the total order $\le_g$) corresponding to the polynomials in $S$. We assume that $\tilde{S}$ is closed in the Euclidean topology on $\R^\nelem$. Since the Euclidean projection 
onto a closed set such as $\tilde{S}$ need not be unique, we fix a {\it measurable selection}
 \[
\Pi_S: \R^\nelem\to \tilde{S}, \qquad \mbox{such that} \qquad u\mapsto \Pi_S(u) \in \arg \min_{v \in \tilde{S}} \| v- u\|.
\]
Our starting point is the output of Algorithm~\ref{alg:two}, and in particular, the set $\mathcal{A}^{(n,g)}$ (see~\eqref{eq:A_ng}). As emphasized in Remark~\ref{Rem:g=g^*}, knowledge of the true value $g^*$ is essential for recovering the algebraic set $\mathcal{A}$; therefore, in this subsection, we assume that $g=g^*$.

Figure~\ref{fig:Proj} illustrates the performance of the structure-aware projection method across three examples: (i) a cross,  (ii) two concentric circles—corresponding to the zero set of a quartic polynomial that factors into two quadratic components, and (iii) three intersecting lines—corresponding to the zero set of a cubic polynomial that factors into three affine components. In each panel of Figure~\ref{fig:Proj}, the estimated algebraic set obtained directly from the zero set of the fitted polynomial using Algorithm~\ref{alg:two} (see Section~\ref{subsection:recoverSet-Method-1}) is shown in red, while the improved estimate—obtained by projecting the coefficients onto the space of polynomials with the known factorization structure—is shown in {blue}.

Recall the estimated eigenspace $\hat{\rm J}_{g^*}$ of the debiased moment matrix $\hat{\mathbb{M}_g}$, as defined in~\eqref{eq:hat-J_g}. Given knowledge of the set $S$, we seek to project the estimated coefficient vectors of the fitted polynomials onto $S$, to obtain the set
$$ \hat{\rm J}_{g^*}^S :=\{\Pi_S(u): \ u\in  \hat{\rm J}_{g^*}, \;  \|u\|\leq 1\}. $$
Our next result, Lemma~\ref{lemma:estimation-cross-ideal}, establishes the convergence of  $\hat{\rm J}_{g^*}^S$  to the closed unit ball of ${\rm J}_{g^*}$, under the following assumption.

\begin{Assumption}\label{Assumption-prior-knloegde}
We assume that the vanishing ideal ${\rm I}(\mathcal{A})$ (see Definition~\ref{def:VanishingIdeal}) is generated by polynomials $P_1, \dots, P_k\in \R_{\leq g^*}[x_1, \dots, x_d]$  and that $\tilde{S}\subset \R^{\kappa_{d,g^*}}$ is closed in the Euclidean topology and contains ${\rm J}_{g^*} $ (recall the definition of ${\rm J}_{g} $ from Section~\ref{sec:Est-Kernel}). 
\end{Assumption}
\begin{Lemma}\label{lemma:estimation-cross-ideal}
Let $\{X_i, \eps_i, \theta_i \}_{i=1}^n$ be as in~\eqref{Model} and $\{r_n\}_{n\in \N}$ be a sequence such that $r_n\to 0$ and $n^{{1}/{2}}r_n \to \infty$. Suppose that Assumptions~\ref{Assumption:CLT} and \ref{Assumption-prior-knloegde} hold. Then, as $n \to \infty$: 
 \begin{enumerate}
     \item any limit point (in probability) of a sequence $u_n\in \hat{\rm J}_{g^*}^S$  lies in $ {\rm J}_{g^*}$; 
     \item for any $u\in  {\rm J}_{g^*} $ with $\|u\|\leq 1$, there exists a sequence $u_n\in \hat{\rm J}_{g^*}^S$ such that $u_n\xrightarrow{\mathbb{P}} u$. 
 \end{enumerate}
\end{Lemma}
The lemma above follows from the observation that for any $u\in \tilde{S}$ and $v\in \R^\nelem$, the triangle inequality yields
\begin{equation}
    \label{desigualdad-Gilles}
    \| \Pi_S(v) - u\| \le \| u - v \| + \| v -  \Pi_S(v)\| \le 2\| u - v \|.
\end{equation}
We now define the algebraic set 
$$ \mathcal{A}^{(n,g^*)}_S :=\{\x\in \R^d: P(x)=0 \text{ for all $P$ with coefficients in }    \Pi_S(\hat{\rm J}_{g^*})\}, $$
which will be our estimator of $\mathcal{A}$. Next, we show that under the following assumption, our method achieves asymptotic recovery of 
$\mathcal{A}$ at all regular points of its irreducible components. Note that this set of points (cf.~\eqref{eq:regular-prior} below) is larger than the set of regular points for reducible sets defined in Definition~\ref{defn:regular-pt}. For instance, the origin in the “cross” example (see Example~\ref{ex: 2}) is not a regular point in the sense of Definition~\ref{defn:regular-pt}, but it does satisfy condition~\eqref{eq:regular-prior}.

\begin{Assumption}\label{Assumption-prior-knloegde-set}
Suppose that
\begin{equation}
 S := \left\{ P = Q_1 \cdots Q_m :\; Q_j \in \mathbb{R}_{\leq {g_j}}[x_1, \dots, x_d] \quad \text{for } j = 1, \dots, m \right\},
\end{equation}
for a known sequence of degrees  \(g_1, \dots, g_m \in \mathbb{N}\) and a fixed \(m \in \mathbb{N}\).  
Further, assume that \({\rm I}_{g^*} \subset S\) and that \(\dim({\rm I}_{g^*}) = 1\).
\end{Assumption}
We note that Assumption~\ref{Assumption-prior-knloegde-set} differs from Assumption~\ref{assumption:superStrong}. Specifically, Assumption~\ref{assumption:superStrong} permits 
$\dim({\rm I}_{g^*}) > 1$, while Assumption~\ref{Assumption-prior-knloegde-set} encompasses algebraic sets that are not necessarily irreducible.

\begin{Theorem}\label{Theorem:recoveringSet-Prior}
Let \(\{X_i, \varepsilon_i, \theta_i\}_{i=1}^n\) follow the model in~\eqref{Model} and recall~\eqref{eq:hat-J_g} where \(\{r_n\}_{n\in \mathbb{N}}\) is a sequence such that \(r_n \to 0\) and $n^{1/2} r_n \to \infty$. Suppose that Assumptions~\ref{Assumption:CLT} and~\ref{Assumption-prior-knloegde-set} hold, and that the measurable selection \(\Pi_S\) is homogeneous, i.e., \(\Pi_S(\lambda v)=\lambda\,\Pi_S(v)\) for all \(\lambda>0\).\footnote{Assumption~\ref{Assumption-prior-knloegde-set} implies that \(S\) is a cone: if \(Q \in S\), then \(\lambda Q \in S\) for all \(\lambda >0\). As a result, for any \(\lambda \neq 0\) and \(u \in \mathbb{R}^\nelem\),
$\arg\min_{v \in \tilde{S}} \|v - \lambda u\| = \arg\min_{v \in \tilde{S}} \left\| \frac{1}{\lambda}v - u \right\| = \left\{ \lambda w : w \in \arg\min_{v \in \tilde{S}} \|v - u\| \right\},
$
which justifies the assumption that the measurable selection \(\Pi_S\) satisfies \(\Pi_S(\lambda v) = \lambda \Pi_S(v)\).}
Fix any nonzero polynomial \(P = Q_1 \cdots Q_m \in {\rm I}_{g^*} \setminus \{0\}\). Then, 
for any point
\begin{equation}
    \label{eq:regular-prior}
    x_0 \in \bigcap_{\{j: \, x_0 \in \mathcal{V}(Q_j)\}} \mathrm{reg}\bigl(\mathcal{V}(Q_j)\bigr),
\end{equation}
there exists an open Euclidean neighborhood \(\mathcal{U}\) of \(x_0\) such that
$d_{\mathrm{H}}\bigl(\mathcal{A} \cap \mathcal{U},\; \mathcal{A}^{(n,g^*)}_S \cap \mathcal{U}\bigr) \xrightarrow{\mathbb{P}} 0.$
Moreover, for any \(x_0 \in \operatorname{reg}(\mathcal{A})\), there exists an open Euclidean neighborhood \(\mathcal{U}\) of \(x_0\) such that
\[
d_{\mathrm{H}}(\mathcal{A} \cap \mathcal{U},\; \mathcal{A}^{(n,g)}_S \cap \mathcal{U}) = \mathcal{O}_{\mathbb{P}}(n^{-1/2}).
\]
\end{Theorem}
The above result shows the consistency of this procedure at regular points of the irreducible components, and provides a parametric rate at regular points of \(\mathcal A\). Let us interpret Theorem~\ref{Theorem:recoveringSet-Prior} in the context of the “cross” from Example~\ref{ex: 2}. In this case, we take $S = \left\{ Q_1 Q_2 \in \mathbb{R}_{\leq 2}[x, y] :\; \deg(Q_1) = \deg(Q_2) = 1 \right\},$
so that \(g^* = 2\) and
${\rm I}_2 = \left\{ \lambda(x + y)(x - y) : \lambda \in \mathbb{R} \right\}.$
Clearly, Assumption~\ref{Assumption-prior-knloegde-set} is satisfied. Moreover, one can verify that every point on \(\mathcal{A}\) satisfies the condition in~\eqref{eq:regular-prior}. In particular, the origin \((0, 0)\)—which lies at the intersection of the two lines—satisfies the condition even though it is not a regular point in the sense of Definition~\ref{defn:regular-pt} (cf.~Figure~\ref{fig:Proj}).

\begin{Remark}\label{rem:No-rate} 
The convergence rates around non-regular points satisfying  \eqref{eq:regular-prior}
are difficult to derive as one can have a polynomial $P_n=Q_nR_n$ converging  to $P=QR$ for which  $Q_n$ converges slower than $R_n$. Hence, the proofs, which are based on the implicit function theorem, are not easily adaptable to each of the reducible components of $\mathcal{A}$.  %
\end{Remark}

\section{Conclusion and outlook}  
\label{sec: conclusion}

In this paper, we introduced a statistically principled method for estimating the underlying structure of noisy data under the assumption that the signal lies on an algebraic set. By suitably debiasing the moment matrix derived from the Vandermonde matrix, we proposed an estimator for the algebraic set and established its consistency, convergence rates, and a central limit theorem. These results place our approach on firm theoretical footing and open several promising directions for future research.

Our current theoretical results assume a fixed ambient dimension \(d\) as the sample size \(n\) grows. A natural question is how the method behaves when the number of monomials \(\nelem = \binom{g+d}{d}\) grows with \(n\), for instance in the regime \(\nelem \asymp n\). Is it possible to develop a natural \emph{shrinkage estimator} for the empirical moment (Vandermonde) matrix in such settings? What regularization strategies might preserve algebraic structure while ensuring statistical stability?

Another open question is whether one can simultaneously estimate the \emph{noise covariance} $\Sigma$ and the algebraic set when the degree \(g\) is unknown. This would require combining algebraic structure estimation with model selection techniques, potentially leading to more adaptive procedures. Can such methods be made computationally efficient?

A known challenge in eigen decomposition-based methods is that, by definition as a minimization problem, especially in the presence of noise or when the spectrum is nearly degenerate, the eigenvectors appear unstable in practice. Can one devise \emph{robust alternatives} to extract the approximate kernel of the moment matrix? Recent work, such as that of \cite{oliveira2024improved}, offers refined perturbation bounds for eigenspaces. Can such results be adapted or extended to our setting to provide robust guarantees?

While Theorems~\ref{Theorem:recoveringSet} and~\ref{Theorem:completeRecovery} provide explicit local rates of convergence for the estimated algebraic set, Theorem~\ref{Theorem:recoveringSet-Prior}---which leverages prior structural information---does not currently yield any such quantitative convergence guarantees for \(\mathcal{A}^{(n,g^*)}_S\); see Remark~\ref{rem:No-rate}. Establishing \emph{local rates of convergence} in this setting remains an open problem. 

An alternative approach to solving the  problem considered in this paper may involve \emph{direct optimization over the latent points}. Specifically, one may consider the problem $
\min_{\theta_1,\dots,\theta_n} \sum_{i=1}^n \| X_i - \theta_i \|^2,$
subject to the constraint that the latent points \(\theta_1,\dots,\theta_n\) lie on a common algebraic set. Such an approach has been explored, for instance, in \cite{himmelmann2020generalized}, and although computationally more involved, it has the key advantage of being \emph{agnostic to the noise distribution}---unlike our current method, which relies on moment debiasing and thus implicitly assumes knowledge of the noise structure. This naturally raises the question: \emph{Can we design a hybrid approach that combines the statistical efficiency and computational simplicity of moment-based methods with the robustness and flexibility of projection-based formulations?} In particular, is it possible to develop a statistically consistent and computationally tractable estimator of the underlying algebraic set that does not require explicit modeling or estimation of the noise distribution?

More broadly, investigating \emph{minimax rates} under various structural and noise assumptions, as well as providing \emph{algorithmic guarantees} for constrained projection methods (e.g., convergence of alternating minimization or gradient descent schemes on non-convex algebraic varieties), represent valuable directions for future research.

\bibliographystyle{imsart-number}
\bibliography{Biblio}

\newpage

\appendix
\section{Numerical experiments}\label{sec: AddSimus}

In this section, we illustrate the performance of the proposed methods through numerical experiments, focusing on how they behave as we vary the sample size \(n\) and the noise level \(\sigma\). We consider three representative models:  
(i) a cross,  
(ii) the intersection of three lines, and  
(iii) two concentric circles.

\subsection{Increasing the sample size}

Figure~\ref{fig:Tube} shows how the estimated semi-algebraic set (described in Section~\ref{Section:tubes}) evolves with increasing sample size for the cross example in Example~\ref{ex: 2}. Here, we set \(\lambda_n \propto n^{-1/2} \log(n)\). As expected, the size of the estimated semi-algebraic set contracts and converges toward the true set as the sample size $n$ increases.

\begin{figure}[h!]
    \centering
    \includegraphics[width=0.42\linewidth]{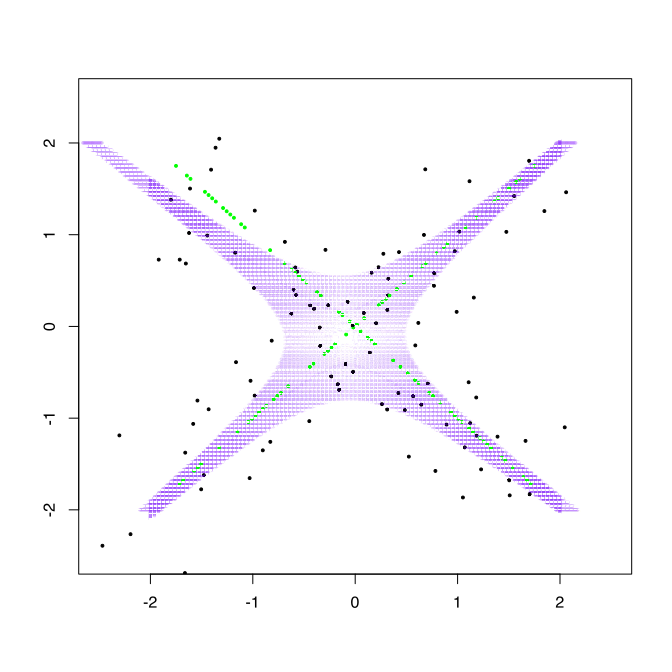}    
    \includegraphics[width=0.42\linewidth]{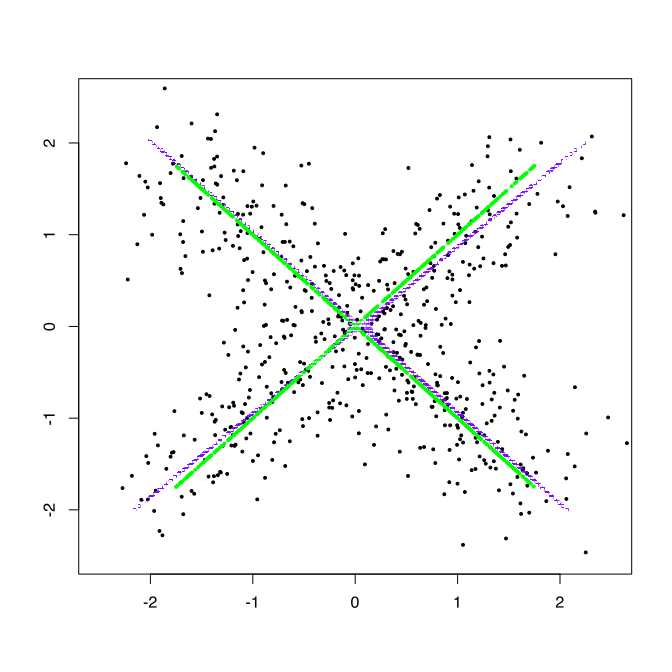}
    \fbox{ 
        \textcolor{green}{$\bullet$} {\footnotesize latent points \(\theta_i\)}\quad
        \textcolor{black}{$\bullet$} {\footnotesize observed points \(X_i\)}\quad
        \textcolor{purpleR}{\textemdash} {\footnotesize estimated semi-algebraic set} 
    }
    \caption{\label{fig:Tube}Estimated semi-algebraic sets (in purple) for two sample sizes \(n \in \{100, 600\}\) with \(\lambda_n \propto n^{-1/2} \log(n)\) and \(\sigma = 0.4\). The size of the semi-algebraic set shrinks as \(n\) increases.}
\end{figure}

\subsection{The impact of noise}
Figure~\ref{fig:Noise} demonstrates the impact of increasing noise levels on the recovery of an algebraic set defined by two concentric circles. Despite significant noise, the method from Section~\ref{subsection:recoverSet-Method-1} remains highly effective in capturing the true structure.

\begin{figure}[h!]
    \centering
    \includegraphics[width=0.42\linewidth]{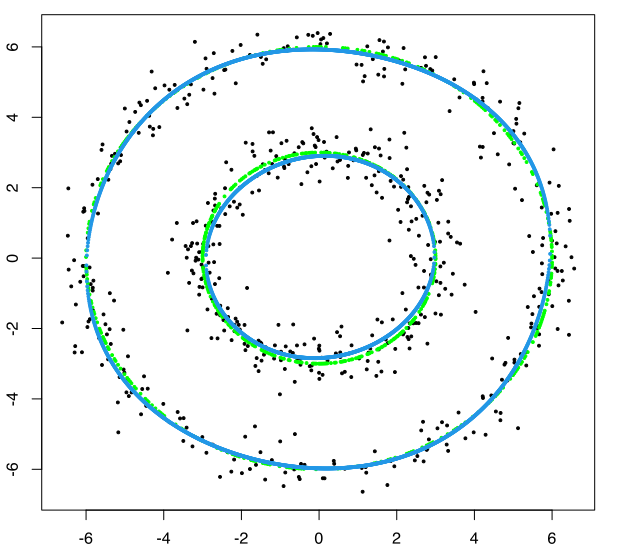}    
    \includegraphics[width=0.42\linewidth]{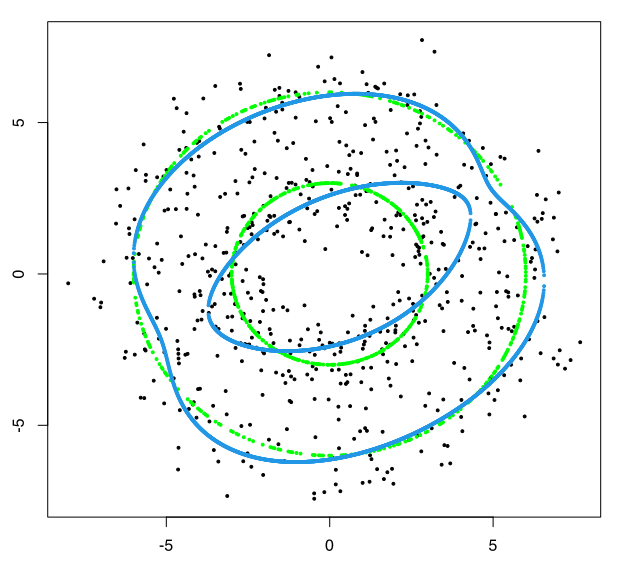}
    \fbox{ 
        \textcolor{green}{$\bullet$} {\footnotesize latent points \(\theta_i\)}\quad
        \textcolor{black}{$\bullet$} {\footnotesize observed points \(X_i\)}\quad
        \textcolor{blueR}{\textemdash} {\footnotesize estimated algebraic set} 
    }
    \caption{\label{fig:Noise}Estimated algebraic sets (in blue) obtained using Algorithm~\ref{alg:two} (see Section~\ref{subsection:recoverSet-Method-1}) as the noise level increases from \(\sigma = 0.4\) (left) to \(\sigma = 0.8\) (right). Here, \(n = 600\) and green dots represent the latent points.}
\end{figure}

\section{Recovery of  ${\rm I}_g$, the  subspace of vanishing polynomials of degree $g$}\label{SubSec:EstimationVanasingPoly}
Note that in Theorem~\ref{Theorem:main} we have already established the consistent recovery of ${\rm J}_{g}$, the limit of the kernel of $\mathbb{M}_g(\nu_n)$. Observe now that the set ${\rm I}_{g}$ of polynomials of degree $g$ vanishing at the latent points $\{\theta_i\}_{i\in \N}$ precisely have their vector of coefficients (following the monomial order $\leq_g$) that are elements of ${\rm J}_{g}$ (cf.~\eqref{eq:J_g} and \eqref{eq:Ig}).   

A natural next step is to show the convergence of the polynomials indexed by vectors in $\hat {\rm J}_g$  in a suitable topology. To study the convergence of polynomials, we need the notion of uniform convergence on compact sets. Recall that a sequence of continuous functions $\{f_n\}_{n\in \N}$  converges to a continuous function $f$ uniformly on compact sets if for every compact set $K\subset \R^d $, 
$ \sup_{\x\in K}|f_n(\x)-f(\x)|\to 0$, as $n \to \infty.$
In such a case we say that $f_n\to f$ in $\mathcal{C}_{c}(\R^d)$. 

The following result  states that each polynomial of ${\rm I}_g$ (see~\eqref{eq:Ig}) can be approximated  by elements of $\hat{\rm I}_g$ (see~\eqref{eq:hat-Ig}); moreover, each limit point of elements of $\hat{\rm I}_g$ belongs to ${\rm I}_g$. This readily follows from the fact that the convergence of a sequence of polynomials on compact sets is equivalent to the convergence of the coefficients.

\begin{Corollary}\label{Coro:ConsistencyPolynomials}
    Under the setting of Theorem~\ref{Theorem:main} the following holds: 
    \begin{enumerate}
        \item For any {$P\in {\rm I}_g$}
        there exists a sequence $\{P^{(n)}\}_{n\in \N}  $ with $ P^{(n)} \in \hat{{\rm I}}_g $, such that 
        $  P^{(n)} \xrightarrow{\mathbb{P}} P $ in $\mathcal{C}_{c}(\R^d)$.
       \item  If  $\{P^{(n)}\}_{n\in \N}  $ is a sequence  with $ P^{(n)} \in \hat{{\rm I}}_g $, such that 
        $  P^{(n)} \xrightarrow{\mathbb{P}} P $ in $\mathcal{C}_{c}(\R^d)$, then {$P\in {\rm I}_g$}. 
    \end{enumerate}
\end{Corollary}
The above corollary implies that for $g$ larger or equal to the minimum degree $ g^*$, we are able to recover the ideal ${\rm I}_\infty$ of polynomials  vanishing on $\{\theta_i\}_{i\in \N}$ (cf.~Definition~\ref{def:minDegree}). This is formalized in the following corollary.

\begin{Corollary}
\label{cor: limPoly}
   Fix $g\geq g^*$. Under the setting of Theorem~\ref{Theorem:main}, for any {$P\in {\rm I}_\infty$}
        there exists a sequence $\{P^{(n)}\}_{n\in \N}  $ with $ P^{(n)} \in \langle\,\hat{{\rm I}}_g \,\rangle $, such that 
        $  P^{(n)} \xrightarrow{\mathbb{P}} P $ in $\mathcal{C}_{c}(\R^d)$.
\end{Corollary}

\section{Proofs}
\label{sec: Proofs}
\subsection{Proofs of Section~\ref{Sec:Algebra}}\label{Appendix:ProofSect3}
\begin{proof}[Proof of Lemma~\ref{Lemma:StableIdeal}]
Note that as ${\rm I}_g^{(n+1)}\subset {\rm I}_g^{(n)} $ for all $n\in \N$, the chain of  subspaces $\{{\rm I}_g^{(n)}\}_{n\in \N}$ is descending. It is easy to check that 
$$ {\rm I}_g= \bigcap_{n\in \N} {\rm I}_g^{(n)},$$
by recalling that $P\in {\rm I}_g$ if and only if $P\in \R_{\leq g}[x_1, \dots, x_d]$ and $P(\theta_i)=0$ for all $i\in \N$.
We now use the fact that every finite-dimensional vector space $H$ is Artinian, meaning that for any sequence $\{H_n\}_{n\in \N}$ of subspaces of $H$  such that $H_{n+1}\subset H_n$ there exists $n_0$ such that $H_n=H_{n+1} $ for all $n\geq n_0$ (called descending chain condition; cf.,~\cite[Corollary~10.16.]{Anderson.1992.Modules}). Therefore,  there exists $n_g$ such that ${\rm I}_g^{(n)}={\rm I}_g^{(n+1)}$ for all $n\geq n_g$. Thus, the first claim follows.

If $g\geq g^*$, by definition we have ${\rm I}_\infty=\langle \, {\rm I}_g \, \rangle$ and $\mathcal{V}( {\rm I}_g )=\mathcal{A}$, which completes the proof.  
\end{proof}

\begin{proof}[Proof of Lemma~\ref{lem: isomorp}]
The proof is constructive and the construction is given in the following result. 
\begin{Lemma}
\label{lem: isomorpLongAppendix}
We define the following equivalence relation over $\{0, \dots, {d}\}^g$: ${\bf j}=(j_1, \dots, j_g)\sim {\bf j'}=(j_1', \dots, j_g')$ if and only if  there exists a permutation $ \tau\in [[g]]$ such that $ (j_{\tau(1)}, \dots, j_{\tau(g)})=(j_1', \dots, j_g') $. Denote the equivalence class of ${\bf j}$ by $[{\bf j}]$. Let ${e}_{l} := (0,\ldots, 0, 1, 0,\ldots,0) \in \R^{d+1}$ be the $(l+1)$'th canonical vector, for $l=0,\ldots, d$. Then the following hold:  
    \begin{enumerate}
        \item  The map 
$$\mathcal{U}_g:\quotient{\{0, \dots, {d}\}^g}{\sim}\ni [{\bf j}]\mapsto \frac{1}{g! }\sum_{{\bf l}\in [{\bf j}]}\otimes_g({ e}_{l_1}, \dots, { e}_{l_g}) \in {\rm sym}((\R^{d+1})^{\otimes_g})   $$
is well-defined and  a basis of 
$ {\rm sym}((\R^{d+1})^{\otimes_g})$ is given by 
$ \{ \mathcal{U}_g([{\bf j}]): \ {\bf j} \in \{0, \dots, {d}\}^g \} . $

\item Recall the definition of $\mathcal{B}_{d,g}$ from~\eqref{eq:B-gN}. For each ${\bf i}\in \mathcal{B}_{d,g}$ there exists a unique $ [{\bf j}{(\bf i)}]\in \quotient{\{0, \dots, {d}\}^g}{\sim}$ such that ${\bf i}=\sum_{k=0}^{d} e_{{\bf j}{({\bf i})}_k}$.

\item 
The linear map $\gamma_g: {\rm sym}((\R^{d+1})^{\otimes_g})  \longrightarrow  \R^\nelem$ defined by
\begin{align*}
\gamma_g:     \sum_{{\bf i}\in \mathcal{B}_{d,g} } a_{\bf i} \mathcal{U}_g([{\bf j}{({\bf i}) }]) &\longmapsto  (a_{\bf i})_{{\bf i}\in \mathcal{B}_{d,g}}
\end{align*}
is bijective and satisfies  \eqref{gammag}.
 \end{enumerate}
\end{Lemma}
\begin{proof}
    A basis 
    of $(\R^{d+1})^{\otimes_g}$ is given by $$B=\{\text{ $\otimes_g({ e}_{j_1}, \dots, { e}_{j_g})$ such that $(j_1, \dots, j_g)\in \{0, \dots, {d}\}^g$} \}.$$
    Therefore, 
 ${\rm sym}((\R^{d+1})^{\otimes_g})$ is spanned by 
 $$
 \left\{\frac{1}{g! }\sum_{{\bf j}\in [{\bf j}]}\otimes_g({ e}_{j_1}, \dots, { e}_{j_g}) : {\bf j} \in \{0, \dots, {d}\}^g \right\}. 
 $$
 We observe that  the dimension of  ${\rm sym}((\R^{d+1})^{\otimes_g})$ is $\kappa_{d,g}=\binom{d+g}{g}$ (cf.\ Proposition~2.4 in \cite{ComonetAl.SIMAA.2008}), which agrees with the number of elements on the space $ \{0, \dots, {d}\}^g/\sim$. Therefore, the first  claim follows.  For each ${\bf i}=(i_{0}, \dots, i_{d})\in \mathcal{B}_{d,g}$, the representation 
 $$
 {\bf i}=\sum_{j=0}^d i_j { e}_{j} 
 $$
 is unique. Then we can find the vector 
 $$ {\bf j}({{\bf i}})=( \underbrace{0, \dots, 0}_{i_0 \ {\rm times}}, \underbrace{1, \dots, 1}_{i_1 \ {\rm times}}, \dots, \underbrace{d, \dots, d}_{i_d \ {\rm times}})\in \{0, \dots, {d}\}^g, $$ 
 which satisfies $ {\bf i}=\sum_{k=0}^{d} { e}_{j{({\bf i})}_k} . $ Any element of $[{\bf j}{({\bf i})}]$ satisfies the same property. If  ${\bf j}'\notin [{\bf j}({{\bf i}})]$, then there exists $s\in \{0, \dots, d\}$ such that the number  of indexes $k\in \{ 1, \dots, g\}$ such that $s=j_k'$ differs from the  number of indexes $k\in \{ 1, \dots, g\}$ such that $s=j_k({\bf i})$. Therefore, the representations $ \sum_{k=0}^{d} { e}_{j'_k}  $ and $ \sum_{k=0}^{d} { e}_{j{({\bf i})}_k}  $  are different and the second claim follows. 
 
 The last claim follows easily from the fact that  $\gamma_g$ maps elements of a basis to elements of a basis and both spaces have the same dimension $\kappa_{d,g}$. 
\end{proof}
\end{proof}
\begin{proof}[Proof of Lemma~\ref{lem: RepSq}]
It is well-known (see e.g., \cite[p.~7]{Ryan.2002}) that if
 $U,V,W,X$ are  vector spaces, $F: V\to W$ and $G: U\to X$ are linear maps, then there exists a unique linear map  $ F\otimes G\to W\otimes X $ such that
 \begin{equation}
     \label{eq:propertiesotimes}
     (F \, v) \otimes (G\, u)= (F\otimes G)(v\otimes u ) \quad \text{for all }\  (u,v)\in U\times V.
 \end{equation}
By \eqref{ExampleTranspose} there exists a  linear isomorphism $$h_{\nelem}:\R^\nelem\otimes \R^\nelem \to \mathcal{M}_{\nelem\times \nelem}(\R)$$ such that  $  { u}\, { u}^\top =  h_{\nelem}({ u}\otimes { u})$ for all $u\in \R^\nelem$. Using \eqref{gammag} we get
   $$
   \phi(x)\left(\phi(x)\right)^\top =  h_{\nelem}(\phi(x)\otimes \phi(x)) = h_{\nelem}(\gamma_g(\tilde{\x}^{\otimes_g})\otimes \gamma_g(\tilde{\x}^{\otimes_g})). 
   $$
The relation \eqref{eq:propertiesotimes} implies 
   $$ \phi(x)\left(\phi(x)\right)^\top = h_{\nelem}((\gamma_g\otimes \gamma_g) (\tilde{\x}^{\otimes_g}\otimes \tilde{\x}^{\otimes_g}))= h_{\nelem}((\gamma_g\otimes \gamma_g) (\tilde{\x}^{\otimes_{2g}})).
   $$
   Therefore the result holds for any counting probability measure $\mu_n =\frac{1}{n}\sum_{i=1}^n \delta_{x_i} $, i.e., 
   $$ \frac{1}{n}\sum_{i=1}^n  \phi(x_i)\left(\phi(x_i)\right)^\top = \frac{1}{n}\sum_{i=1}^n h_{\nelem}((\gamma_g\otimes \gamma_g) (\tilde{\x_i}^{\otimes_{2g}}))= h_{\nelem}\left((\gamma_g\otimes \gamma_g) \left(\frac{1}{n}\sum_{i=1}^n\tilde{\x}_i^{\otimes_{2g}}\right)\right). $$
   We now conclude  by taking a sequence of counting measures $\mu_n$ converging to $\mu$ in the $2g$-Wasserstein distance (see \cite{villani2003topics}).  
\end{proof}

\begin{proof}[Proof of Theorem~\ref{theorem:representation}]
Theorem~4.1 in        \cite{pereira2022tensor} states that if $Y\sim \mathcal{N}( {\bf \mu}, \Sigma)$, for $\mu \in \R^d$ and $\Sigma \in \R^{d \times d}$ is a positive definite matrix and $g \in \N$,  then, 
\[
\E \left[Y^{\otimes g}\right] = \sum_{k=0}^{\lfloor g/2\rfloor} C_{g,k} \sym \left( \mu^{\otimes(g-2k)} \otimes \Sigma^{\otimes k} \right).
\]
Using this  together with Lemma~\ref{lem: RepSq} we obtain the first claim. The second one follows from  \cite[Theorem 5.1]{pereira2022tensor} combined with the linearity of the isomorphism.
\end{proof}

\begin{proof}[Proof of Corollary~\ref{prop: EasyFormula}]
In the particular case where the noise covariance is of the form $\sigma^2 I$ (for $\sigma >0$
 remark that the moment matrix requires computing quantities of the form 
\[
\frac1n \sum_{i=1}^n \prod_{j=1}^d \theta_{i,j}^{{\bf{a}}(j)}, 
\]
where ${\bf{a}}(j)$ is a shorthand notational for the corresponding exponents. 
 Under the assumption of diagonal noise, we remark that 
 \[
 \E \left[\frac1n \sum_{i=1}^n \prod_{j=1}^d f_j(X_{i,j})\right] =   \frac1n \sum_{i=1}^n \prod_{j=1}^d \E \left[ f_j(X_{i,j})\right],
 \]
for any polynomial functions $f_j$'s. 
In the particular case when $d=2$, the unbiased estimator of $n^{-1} \sum_{i=1}^n \theta_{i,1}^K \theta_{i,2}^L$, for $K,L \in \mathbb{N}$, is 
\[
\sum_{j=0}^{\lfloor K/2 \rfloor }\sum_{\lambda=0}^{\lfloor L/2 \rfloor } C_{K,j}C_{L,\lambda} (-1)^{j+\lambda} \sigma^{2j+2\lambda} \frac1n \sum_{i=1}^n X_{i,1}^{K-2j}  X_{i,2}^{L-2\lambda},
\]
where we recall the definition
\[
C_{K,j}:={K \choose 2j} \frac{(2j)!}{j!2^j}.
\]
This can be checked thanks to the equality, for $X\sim \mathcal{N}(m, \sigma^2)$, and $K\in \N$,
\[
\E \left[ X^K \right] = \sum_{j=0}^{\lfloor K/2 \rfloor} {K \choose 2j} \frac{(2j)!}{j!2^j} m^{K-2j} \sigma^{2j},
\]
combined with the proof of Theorem~5.1 in~\cite{pereira2022tensor}.
\end{proof}


\subsection{Proofs of Section~\ref{sec: Cons+CLT}} \label{Appendix:ProofSect4}

\begin{proof}[Proof of Theorem~\ref{Theorem:CLTMatrix}]
We start with the proof of the central limit theorem, relying on decomposition \eqref{eq: MomDec}. As each matrix in the sum depends on a different $\theta_i$, we will recourse to the Lindeberg-Feller central limit theorem~\cite[Proposition~2.27]{vdv98}. For a triangular array of independent random vectors $Y_{n,1}, \ldots,Y_{n,n}$, the two conditions 
\begin{align*}
&\frac1{n} \sum_{i=1}^ n\E \left[\|Y_{n,i}\|^2 \mathds{1}_{\{\|Y_{n,i}\| > \eps n^{1/2}\}}\right] \to 0, \text{every } \eps>0;\\
&\frac1n\sum_{i=1}^n \Cov(Y_{n,i}) \to \mathfrak{S}', 
\end{align*}
for some matrix $\mathfrak{S}'$ suffice to establish that $\frac1{\sqrt{n}}\sum_{i=1}^n\left(    Y_{n,i} - \E[Y_{n,i}]\right)\xrightarrow{\ w\ } \mathcal{N}(0, \mathfrak{S}')$. The only condition to establish is that point 2.\ of Assumption~\ref{Assumption:CLT} implies the first condition in the display above, setting $Y_{n,i}:=M_{n,i}$.
Looking at the terms 
\[
\|M_{n,i}\|_{\rm Fr}^2 
\mathds{1}_{\big\{\|M_{n,i}\|_{\rm Fr}^2 > \eps^2 \big\}},
\] one sees, recalling the definition of $M_{n,i}$ above and seeing that $C_{2g,k}(-1)^k$ as well as $\Sigma$ are bounded and $ h_{\nelem}(\gamma_g\otimes \gamma_g)$ is a fixed linear map, that the size of $\|M_{n,i}\|_{\rm Fr}^2$ is driven by that of
$\|{X}_i\|^{4g}= \|\theta_i+\eps_i\|^{4g}$. In the light of this and as $\eps_i$ is Gaussian,  it is clear that the assumption on the $\theta_i$'s guarantees the first condition of the Lindeberg-Feller theorem.

To prove the almost sure convergence, we remark that proving it for each entry of the matrix suffices.  Under our assumptions, the statement then follows directly from the result of \cite{kolmogorov1930loi} discussed in \cite[Notes to Section~8.3]{dudley2018real}. 
\end{proof}

\begin{proof}[Proof of Proposition~\ref{Prop:ConsistencyEigenValues}]
The first claim (resp.\ the second claim) follows from an application of Theorem~\ref{Theorem:CLTMatrix}-(i)  (resp.\ Theorem~\ref{Theorem:CLTMatrix}-(ii)) coupled with Weyl’s inequalities and the fact that spectral norm of $\hat{\mathbb{M}}_g- \mathbb{M}_g(\nu_n)$ is upper bounded by its Frobenius norm.
\end{proof}

\begin{proof}[Proof of Lemma~\ref{Lemma:StableKernel}]
We claim that  for every $g\in \N$ there exists $n_g\in \N$ and $\Delta_g>0$ such that, for every $n\geq n_g$ 
\begin{equation}
    \label{claim:lemmaStableKernel}\inf_{k>{ k_g}}\lambda_k( {\mathbb{M}}_g(\nu_n)) \geq \Delta_g.
\end{equation}
Note that from the claim the result follows  by Proposition~\ref{Prop:ConsistencyEigenValues} and the stabilization of the kernel ${\rm J}_g$ of ${\mathbb{M}}_g(\nu_n)$ (Lemma~\ref{Lemma:StableIdeal}).

We show the claim {\it ad absurdum}. Assume that there exists a sequence  $\{v_n\}_{n}$ of norm one vectors with $v_n=(a_{\bf i}^{(n)})_{{\bf i}\in \mathcal{B}_{d,g} }$ such that $v_n$ is orthogonal to  
$ {\rm J}_g $ and ${\mathbb{M}}_g(\nu_n)v_n =\lambda_n v_n$ with $\lambda_n\to 0$. Then, along a subsequence, we might assume that $v_n\to v$ with $v$ being norm one and orthogonal to $ {\rm J}_g $.  Let $P_n$ and $P$ be the polynomials with entries $v_n$ and $v$ in the order $\leq_g$, respectively. From here we derive that   
$ P(\theta_i)=0$ for all $i\in \N$, which  implies that $P\in {\rm I}_g$. Equivalently, this implies that  $v\in {\rm J}_g$ which contradicts the fact that $v$ has norm one and is orthogonal to $ {\rm J}_g $. Hence \eqref{claim:lemmaStableKernel} follows. 
\end{proof}

\begin{proof}[Proof of Theorem~\ref{Theorem:main}]
Recall that $k_g = {\rm dim}({\rm I}_g)$. We claim that with probability one, there exists  some (probably random)  $n_0$ such that 
    $$k_g= k^*_{n}:= {\rm dim}\left( \bigoplus_{j\leq j_{n}} E_j( \hat{\mathbb{M}}_g)\right) \quad \forall \; n \geq n_0. $$
Note that if the claim holds,  then 
Theorem~2 in \cite{YuDavisKahan}  implies that 
\begin{equation}
    \label{eq:YuDavisKahan}
    {\rm d}\left( \bigoplus_{j\leq j_{n}} E_j( \hat{\mathbb{M}}_g), \bigoplus_{j\leq j_{n}} E_j( \mathbb{M}_g(\nu_n)) \right) \leq  \frac{2 \, \nelem \| \hat{\mathbb{M}}_g-\mathbb{M}_g(\nu_n)\|_{{\rm Fr}}}{{\Delta}_n},
\end{equation}
where 
${\Delta}_n := \inf_{j>k_g}  \lambda_i\left( \mathbb{M}_g(\nu_n)\right).$ By  Lemma~\ref{Lemma:StableKernel} there exist $ \Delta>0$ and $n_0\in \N$ such that 
$  {\Delta}_n > \Delta $ and $\bigoplus_{j\leq j_{n}} E_j( \mathbb{M}_g(\nu_n))={\rm I}_g$ for all $n\geq n_0$. Hence, the result follows by Theorem~\ref{Theorem:CLTMatrix}. 

Now we prove the claim and finish the proof. We argue by contradiction and assume that one the events 
$$ \mathcal{E}_1= ( \text{there exists $\{n_k\}$ such that $k^*_{n_k}\leq k_g-1$ })$$
or 
$$ \mathcal{E}_2= ( \text{there exists $\{n_k\}$ such that $k^*_{n_k}\geq k_g+1$ })$$
has positive probability. If $\mathcal{E}_2$ has positive probability, then $\lambda_{k_g+1}( \hat{\mathbb{M}}_g) \leq r_{n_{k}}  $ with positive probability and we contradict \eqref{eq:BoundBelowLambdas}. 
 If $\mathcal{E}_1$ is not negligible, then $\lambda_{k_g}( \hat{\mathbb{M}}_g) \geq  r_{n_{k}}$ holds with positive probability, which implies that 
 $$ n_k\E[ (\lambda_{k_g}( \hat{\mathbb{M}}_g))^2 ]\geq  n_k\E[ (\lambda_{k_g}( \hat{\mathbb{M}}_g))^2 \mathds{1}_{\mathcal{E}_1} ] \geq \mathbb{P}(\mathcal{E}_1) n_k r_{n_k}^2 \to \infty,  $$
 contradicting \eqref{eq:convergenceLambdaKernel}. 
\end{proof}

\begin{proof}[Proof of Corollary~\ref{Coro:ConsistencyPolynomials}]
    Fix $P\in {\rm I}_g$. Then  $P=\Pi_{ {\rm I}_g}(P)$. By Theorem~\ref{Theorem:main}, the sequence $\{P^{(n)}\}_{n\in \N}$, with $P_n=\Pi_{ \hat{\rm I}_g}(P)$ converges in probability to $P$ as elements of the linear space $\R_{\leq g}[x_1, \dots, x_d]$, meaning that the components $(a_1^{(n)}, \dots, a_{\nelem}^{(n)})$ of $P^{(n)}$ in any basis of monomials converge in probability to the components  $(a_1, \dots, a_{\nelem})$ of $P$ (in the same basis of monomials). Since this topology over $\R_{\leq g}[x_1, \dots, x_d]$ is equivalent to that of  $\mathcal{C}_K(\R^d)$, we get the first point. The second one holds by the same means.
\end{proof}
\begin{proof}[Proof of Corollary~\ref{Coro:ConsistencyPolynomials}]
     We recall an important property of ideals of the commutative and unital ring $\R[x_1, \dots, x_d]$. If $I\subset \R[x_1, \dots, x_d]$ is an ideal generated by $A\subset \R[x_1, \dots, x_d]$, then (see \cite[Exercise~2.2.1.]{Trifkovi2013})  for every $P\in I$ there exist two finite sequences $Q_1, \dots, Q_k $ and $P_1, \dots, P_k$  of polynomials such that $P_i\in A$, for all $i=1, \dots, k$, and 
$ P=\sum_{i=1}^k Q_i P_i. $
This property and Corollary~\ref{Coro:ConsistencyPolynomials} yield the result.
\end{proof}

\subsection{Proofs of Section~\ref{subsection:recoverSet-Method-1}} \label{Appendix:ProofSect5}
We denote by $$ \tau:\R_{\leq g}[x_1, \dots, x_d]\to \mathcal{B}_{d,g}$$ the linear isomorphism mapping a polynomial $P \in \R_{\leq g}[x_1, \dots, x_d]$ to the vector of its entries in the monomial order $\leq_{g}$. In the following we identify any polynomial in $\R_{\leq g}[x_1, \dots, x_d]$ with its corresponding vector of coefficients (in the monomial order $\leq_{g}$) and use both notions interchangeably.  Note that $\Pi_S$ denotes the orthogonal projection onto the subspace $S$. We write as 
\begin{equation}
    \label{Simplification-tau}
    \Pi_{\hat{\rm I}_g}P=\tau^{-1}\Pi_{\hat{\rm J}_g} \tau P,\quad \Pi_{{\rm I}_g}P=\tau^{-1}\Pi_{{\rm J}_g} \tau P \quad {\rm and}\quad \|P\|= \|\tau P\|
\end{equation}
 to simplify notation.

In the proof, we use 
$N$ to denote the number of generators of 
$\mathcal{A}$, rather than $k$; the symbol $k$ is reserved as a dummy index for subsequences.

Let us first  recall the standard definition of  Painlevé--Kuratowski convergence (see \cite{RockafellarWets}).
\begin{Definition}[Painlevé-Kuratowski]
\label{def:PK}
A sequence of subsets $\{S_k\}_{k\in \N}$ of $\R^d$ converges to $S\subset \R^d $ in the Painlevé-Kuratowski sense if
$$ \limsup S_k :=\{ \x\in \R^d: \ \exists \,\{\x_{k_\eta}\}_{\eta\in \N} \  {\rm s.t.}\  \|\x_{k_\eta}-\x\|\to 0\ {\rm with}\  \x_{k_\eta}\in S_{k_\eta} \  \forall \eta\in \N\}$$
and 
$$ \liminf S_k :=\{ \x\in \R^d: \ \exists \,\{\x_{k}\}_{k\in \N} \  {\rm s.t.}\  \|\x_{k}-\x\|\to 0\ {\rm with}\  \x_{k}\in S_{k} \  \forall k\in \N\} $$ 
agree with $S$.
\end{Definition}

We now state a technical lemma that will be pivotal for handling reducible algebraic sets. It extends Proposition 3.3.10 of~\cite{Bochnak.1998.Book} by giving explicit control on the degrees of the auxiliary polynomials $R_1, \dots , R_{D}\in {I}(\mathcal{A})$ chosen from the vanishing ideal 
$I(\mathcal{A})$. (Recall that 
$I(\mathcal{A})$ is the set of all real polynomials that vanish on  $\mathcal{A}$; see Definition \ref{def:VanishingIdeal}.)

\begin{Lemma}[Local description with bounded degree]
\label{lemma:Local-ring-same-degree}
Let $\mathcal{A}=\mathcal{V}(P_1,\dots,P_N)\subset\R^{d}$ be an algebraic set generated by
polynomials $P_j\in\R_{\le g}[x_1,\dots,x_d]$, for $j=1,\ldots,N$.
Fix a regular point $x_0\in\operatorname{reg}(\mathcal{A})$, and let
$\mathcal{A}_{i^\star}$ be the unique irreducible component of $\mathcal{A}$ that contains $x_0$.
Denote by $D=\operatorname{codim}\bigl(\mathcal{A}_{i^\star}\bigr)$ its (real) codimension. Then there exist a Euclidean neighborhood $\mathcal{U}$ of $x_0$, and polynomials $R_1,\dots,R_{D}\in I(\mathcal{A})\cap\R_{\le g}[x_1,\dots,x_d]$ such that
\begin{equation}\label{eq:R_i}
\bigl\{\nabla R_i(x_0)\bigr\}_{i=1}^{D}\;\text{ are linearly independent,}\qquad
\mathcal{A}\cap\mathcal{U}
=\mathcal{A}_{i^\star}\cap\mathcal{U}
=\mathcal{V}(R_1,\dots,R_{D})\cap\mathcal{U}.
\end{equation}

Moreover, after shrinking $\mathcal{U}$ if necessary, there exist an open set $\mathcal{V}\subset\bigl(\operatorname{span}\{\nabla R_i(x_0)\}_{i=1}^{D}\bigr)^{\!\perp}$ and a $\mathcal{C}^1$ map $\phi:\mathcal{V}\to\operatorname{span}\{\nabla R_i(x_0)\}_{i=1}^{D}$
such that
\begin{equation}
    \label{generate-Local-Ps}
\mathcal{A}\cap\mathcal{U}\;=\;\bigl\{\,u+\phi(u)\;:\;u\in\mathcal{V}\bigr\}.
\end{equation}
\end{Lemma}
\begin{proof} $ $\\
\emph{Step 1: A local generating set without degree control.}
Proposition 3.3.10 of \cite{Bochnak.1998.Book} yields polynomials
$\widetilde R_1,\dots,\widetilde R_{D}\in I(\mathcal{A})$
and a neighbourhood $\widetilde{\mathcal{U}}$ of $x_0$ such that the gradients $\{\nabla\widetilde R_i(x_0)\}_{i=1}^{D}$ are linearly independent, and $\mathcal{A}\cap\widetilde{\mathcal{U}}
   =\mathcal{V}(\widetilde R_1,\dots,\widetilde R_{D})\cap\widetilde{\mathcal{U}}$.  Hence, by the implicit function theorem, after shrinking $\widetilde{\mathcal{U}}$ if necessary, there exist an open set $\widetilde{\mathcal{V}}\subset\bigl(\operatorname{span}\{\nabla \widetilde {R}_i(x_0)\}_{i=1}^{D}\bigr)^{\!\perp}$ and a $\mathcal{C}^1$ map $\phi:\widetilde{\mathcal{V}}\to\operatorname{span}\{\nabla \widetilde R_i(x_0)\}_{i=1}^{D}$
such that
\[
\mathcal{A}\cap\widetilde{\mathcal{U}}\;=\;\bigl\{\,u+\widetilde {\phi}(u)\;:\;u\in\widetilde{\mathcal{V}}\bigr\}.
\]
However, $\deg\widetilde R_i$ may exceed $g$.

\medskip
\noindent
\emph{Step 2: Replacing the $\widetilde R_i$ by degree-$\le g$ generators.}
Because $P_1,\dots,P_N$ generate $I(\mathcal{A})$, we can write
\(
  \widetilde R_i=\sum_{j=1}^{N}S_{i,j}\,P_j
\)
with $S_{i,j}\in\R[x_1,\dots,x_d]$.
Taking gradients at $x_0$ gives
\[
  \nabla\widetilde R_i(x_0)=\sum_{j=1}^{N}S_{i,j}(x_0)\,\nabla P_j(x_0).
\]
Since the set $\{\nabla\widetilde R_i(x_0)\}_{i=1}^{D}$ is linearly independent,
we can choose indices $1\le j_1<\cdots<j_D\le N$ such that
\(
  \{\nabla P_{j_\ell}(x_0)\}_{\ell=1}^{D}
\)
are already linearly independent. Relabel these generators as
\(
  R_\ell:=P_{j_\ell},\ \ell=1,\dots,D.
\)
Each $R_\ell$ lies in $I(\mathcal{A})$ and satisfies $\deg R_\ell\le g$ by assumption on the~$P_j$. Moreover, it also follows that 
$$\mathbb{V}:=\operatorname{span}\{\nabla \widetilde R_i(x_0)\}_{i=1}^{D} = \operatorname{span}\{\nabla R_i(x_0)\}_{i=1}^{D}.$$

\medskip
\noindent
\emph{Step 3: Local equality of zero sets.}  We write $x_0=v_0+v_0^\perp$, where     $v_0\in \mathbb{V}$ and $v_0^\perp\in \mathbb{V}^\perp$ and fix an orthonormal basis $\{e_i\}_{i=1}^D$  of $\mathbb{V}$. 
The implicit-function theorem
applied to the map
\[
\Gamma:\mathbb{V}\times \mathbb{V}^\perp \to \mathbb{V},\quad
(v,w)\mapsto  \sum_{i=1}^D R_i(v+w)  e_i,
\]
yields the existence a neighbourhood $\mathcal{U}'$ of $x_0$ and a neighborhood $\mathcal{V}'$ of $v_0^\perp$ in which
\[
\mathcal{V}(R_1,\dots,R_D)\cap\mathcal{U}'
  =\;\bigl\{\,u+ {\phi}(u)\;:\;u\in {\mathcal{V}}'\bigr\}
\]
for some $\mathcal{C}^1$ map $\phi$. Call $\mathcal{U}= \mathcal{U}'\cap \widetilde{\mathcal{U}} $ and $\mathcal{V}= \mathcal{V}'\cap \widetilde{\mathcal{V}} $. 
Because each $R_\ell$ vanishes on $\mathcal{A}$, the inclusion
$$\mathcal{A}\cap \mathcal{U} \subset\mathcal{V}(R_1,\dots,R_D)\cap\mathcal{U}  $$ 
is immediate. Therefore, by Step~1, we get 
\begin{align*}
   \;\bigl\{\,u+ \widetilde{\phi}(u)\;:\;u\in {\mathcal{V}}\bigr\}= \mathcal{V}(\widetilde R_1,\dots, \widetilde R_D)\cap\mathcal{U}&=\mathcal{A}\cap\mathcal{U}\\
    &\subset\mathcal{V}(R_1,\dots,R_D)\cap\mathcal{U} =\;\bigl\{\,u+ {\phi}(u)\;:\;u\in {\mathcal{V}}\bigr\},
\end{align*}
so that $\phi=\widetilde{\phi}$ and the contention is an equality. Hence, the result follows. 
\end{proof}

\begin{proof}[Proof of Theorem~\ref{Theorem:recoveringSet}] Let $g = g^*$, which follows from Assumption~\ref{assumption:superStrong} and Remark~\ref{rem:Assump}.  

\noindent{\it Consistency around regular points}. Let \({\bf P} := \{P_1, \dots, P_N\} \subset \mathbb{R}_{\leq g}[x_1, \dots, x_d]\) be an orthonormal system (in the sense that the vectors of coefficients are pairwise orthonormal) such that the span of ${\bf P}$ is \({\rm I}_g\). Suppose that the sequence of subspaces spanned by the set of polynomials
\[
{\bf Q}_n := \{Q_{1,n}, \dots, Q_{N_n,n}\} \subset \mathbb{R}_{\leq g}[x_1, \dots, x_d]
\]
converges to \({\rm I}_g\) in the subspace distance (see Definition~\ref{defn:Subspace-dist}). Further, under Assumption~\ref{assumption:superStrong} we have $\mathcal{A} = \mathcal{V}({\rm I}_g)$. Let \(\mathcal{A}_n\) denote the algebraic set defined as the common zero set of \({\bf Q}_n\). Then the following hold:

\begin{enumerate}
    \item For every \(x_0 \in \operatorname{reg}(\mathcal{A})\), there exists a sequence \(\{x_n\}_{n \in \mathbb{N}}\) with \(x_n \in \mathcal{A}_n\) such that \(x_n \to x_0\).
    \item Every limit point of a sequence \(\{x_n\}_{n \in \mathbb{N}}\) with \(x_n \in \mathcal{A}_n\) belongs to \(\mathcal{A}\).
\end{enumerate}

To prove the second claim, let \(\{x_n\}_{n \in \mathbb{N}}\) be a sequence such that  
\[
Q_{1,n}(x_n) = \cdots = Q_{N_n,n}(x_n) = 0,
\]
and suppose that \(x_n \to x\) along a subsequence \(\{x_{n_k}\}_{k \in \mathbb{N}}\). Without loss of generality, we may assume that the polynomials in \({\bf Q}_n\) form an orthonormal system (i.e., their coefficient vectors are pairwise orthogonal). Then, for each fixed \(i \in \{1, \dots, N\}\), the sequence \(\{Q_{i,n_k}\}_{k \in \mathbb{N}}\) admits a further subsequence \(\{Q_{i,n_{k_\ell}}\}_{\ell \in \mathbb{N}}\) that converges to a limit polynomial \(Q_i\). 

Since the subspace spanned by \({\bf Q}_n\) converges to \({\rm I}_g\) (by our assumption), we have:
\begin{itemize}
    \item For sufficiently large \(n\), the dimension of \(\text{span}({\bf Q}_n)\) stabilizes, i.e., \(N_n = N\).
    \item The limiting polynomials \(\{Q_1, \dots, Q_N\}\) span the same subspace as \({\rm span}({\bf P})\), i.e.,  \[\text{span}(\{Q_1, \dots, Q_N\}) = {\rm I}_g.\]
    \item Each \(P_i\) is a linear combination of \(\{Q_1, \dots, Q_N\}\).
\end{itemize}

Because the coefficient vectors of \(Q_{i,n_{k_\ell}}\) converge to those of \(Q_i\), the polynomials \(Q_{i,n_{k_\ell}}\) converge to \(Q_i\) uniformly on compact sets. Therefore,
\[
|Q_i(x)| \leq |Q_i(x) - Q_i(x_{n_{k_\ell}})| + |Q_i(x_{n_{k_\ell}}) - Q_{i,n_{k_\ell}}(x_{n_{k_\ell}})| \to 0,
\]
where the first term vanishes by continuity of \(Q_i\), and the second by uniform convergence on compacts. It follows that \(Q_i(x) = 0\) for all \(i = 1, \dots, N\), so \(x\) is a common zero of \(\{Q_1, \dots, Q_N\}\), and hence also a zero of \(\{P_1, \dots, P_N\}\). Therefore, \(x \in \mathcal{A}\), completing the proof.

We now prove the first claim stated above (i.e., 1.). By definition, the span of 
\({\bf Q}_n := \{Q_{1,n}, \dots, Q_{N_n,n}\}\) converges to the span of \({\bf P} := \{P_1, \dots, P_N\}\) if the orthogonal projection matrix \(\Pi_n \in \mathbb{R}^{\nelem \times \nelem}\) onto \(\{\tau(Q_{1,n}), \dots, \tau(Q_{N_n,n})\}\) converges to the projection \(\Pi\) onto \(\{\tau(P_1), \dots, \tau(P_N)\}\), where \(\tau\) denotes the vectorization of polynomial coefficients.

Let \(x_0\) be a regular point of \(\mathcal{A}\). By Assumption~\ref{assumption:superStrong}, and possibly after a change of coordinates, we may assume that
\begin{equation}\label{eq:Span-P_i}
\operatorname{span} \{ \nabla P_i(x_0) : i = 1, \dots, N \} = \operatorname{span} \{ e_1, \dots, e_N \},
\end{equation}
where \(\{e_1, \dots, e_d\}\) is the canonical basis of \(\mathbb{R}^d\). Let us write \(x_0 = (u_0, v_0) \in \mathbb{R}^N \times \mathbb{R}^{d-N}\).

We define a mapping \(\Gamma : (\mathbb{R}_{\leq g}[x_1, \dots, x_d])^N \times \mathbb{R}^N \to \mathbb{R}^N\) by
\[
\Gamma(Q_1, \dots, Q_N, u) = \left( Q_1(u, v_0), \dots, Q_N(u, v_0) \right).
\]
This map is \(\mathcal{C}^\infty\) when the space of polynomials is endowed with the standard Euclidean topology on the coefficient vectors, as \(\Gamma\) involves only algebraic operations (evaluation of polynomials). Note that since \(x_0 = (u_0, v_0)\) lies in the zero set of \(P_1, \dots, P_N\), we have
\[
\Gamma(P_1, \dots, P_N, u_0) = (P_1(u_0, v_0), \dots, P_N(u_0, v_0)) = 0.
\]
This, together with the invertibility of the Jacobian matrix \(\frac{\partial \Gamma}{\partial u}(P_1, \dots, P_N, u_0)\), will allow us to apply the implicit function theorem, which we do next.

By construction, the partial derivative of \(\Gamma\) with respect to the last argument, evaluated at \((P_1, \dots, P_N, u_0)\), is the identity matrix \(I_N\), by our assumption~\eqref{eq:Span-P_i} on the gradients. Hence, the implicit function theorem guarantees the existence of neighborhoods \(U \subset (\mathbb{R}_{\le g}[x_1, \dots, x_d])^N\) of \((P_1, \dots, P_N)\) and \(V \subset \mathbb{R}^N\) of \(u_0\), and a unique continuously differentiable function
\[
\phi : U \to V
\]
such that for all \((Q_1, \dots, Q_N) \in U\),
\[
\Gamma(Q_1, \dots, Q_N, \phi(Q_1, \dots, Q_N)) = 0,
\]
i.e.,
\[
(Q_1(\phi(Q_1, \dots, Q_N), v_0), \dots, Q_N(\phi(Q_1, \dots, Q_N), v_0)) = 0.
\]
This implies that the point \((\phi(Q_1, \dots, Q_N), v_0) \in \mathbb{R}^d\) lies in the algebraic set defined by \(Q_1, \dots, Q_N\), and that
\[
(\phi(Q_1, \dots, Q_N), v_0) \to (u_0, v_0) = x_0 \quad \text{as} \quad (Q_1, \dots, Q_N) \to (P_1, \dots, P_N).
\]

Since \(\text{span}({\bf Q}_n) \) converges to \(\text{span}({\bf P})\), it follows that the projections \(\Pi_n(P_i) \to P_i\) for each \(i = 1, \dots, N\), where \(\Pi_n(P_i)\) denotes the projection of the coefficient vector of \(P_i\) onto the subspace spanned by the coefficient vectors of \({\bf Q}_n\). Hence, the tuple \((\Pi_n P_1, \dots, \Pi_n P_N)\) converges to \((P_1, \dots, P_N)\), and by continuity of \(\phi\), we have
\[
x_n := (\phi(\Pi_n P_1, \dots, \Pi_n P_N), v_0) \to x_0.
\]

We now verify that \(x_n \in \mathcal{A}_n\) for all large enough \(n\). Since \(\Pi_n P_1, \dots, \Pi_n P_N\) are projections of linearly independent polynomials, they remain linearly independent for large \(n\), by continuity and the non-vanishing of the determinant
\[
\det\left( [\tau P_1, \dots, \tau P_N]^\top [\tau P_1, \dots, \tau P_N] \right) \neq 0.
\]
By continuity of the Gram matrix, the determinant of
\[
[\Pi_n \tau P_1, \dots, \Pi_n \tau P_N]^\top [\Pi_n \tau P_1, \dots, \Pi_n \tau P_N]
\]
is also nonzero for sufficiently large \(n\), implying that \(\Pi_n P_1, \dots, \Pi_n P_N\) are linearly independent.

Furthermore, each \(\Pi_n P_i\) vanishes at \(x_n\) by construction. Since the span of \(\{\Pi_n P_1, \dots, \Pi_n P_N\}\) is equal to the span of \({\bf Q}_n\) for all large \(n\) (this follows from Weyl’s inequality, which implies that the rank of the projection matrix \(\Pi_n\) stabilizes), it follows that each \(Q_{i,n}\) also vanishes at \(x_n\). Hence,
\[
x_n \in \mathcal{V}(Q_{1,n}, \dots, Q_{N_n,n}) = \mathcal{A}_n \quad \text{for all } n \geq n_0,
\]
for some \(n_0 \in \mathbb{N}\), and \(x_n \to x_0\). This proves the claim and establishes the first part of Theorem~\ref{Theorem:main}. \qed \newline

\noindent{\it Global consistency.} We now establish consistency in the PK topology. To this end, we assume that \(\mathcal{A} = \overline{\operatorname{reg}(\mathcal{A})}\), i.e., the algebraic set is the closure of its regular locus. We have already shown that if the span of \({\bf Q}_n\) converges to the span of \({\bf P}\), then the sequence of estimated algebraic sets \(\mathcal{A}_n\) satisfies
\[
\operatorname{reg}(\mathcal{A}) \subset \liminf \mathcal{A}_n \subset \limsup \mathcal{A}_n \subset \mathcal{A}.
\]
Since both \(\liminf \mathcal{A}_n\) and \(\limsup \mathcal{A}_n\) are closed subsets of \(\mathbb{R}^d\) (cf.~\cite[Chapter~4]{RockafellarWets}), we conclude that
\[
\mathcal{A} = \overline{\operatorname{reg}(\mathcal{A})} \subset \liminf \mathcal{A}_n \subset \limsup \mathcal{A}_n \subset \mathcal{A}.
\]
Hence, all inclusions are equalities, and it follows that
\[
{\bf d}_{\mathrm{PK}}(\mathcal{A}_n, \mathcal{A}) \to 0,
\]
i.e., the sequence \(\mathcal{A}_n\) converges to \(\mathcal{A}\) in the PK topology. Therefore, global consistency follows from Theorem~\ref{Theorem:main}. \qed \newline

\noindent{\it Convergence rates.} Fix a regular point \( x_0 = (u_0, v_0) \in \mathcal{A} \). Define the map $\tilde{\Gamma} : (\mathbb{R}_{\leq g}[x_1, \dots, x_d])^N \times \mathbb{R}^{d-N} \times \mathbb{R}^N \to \mathbb{R}^N$ as \[\tilde{\Gamma}(Q_1, \dots, Q_N, v, u) := (Q_1(u, v), \dots, Q_N(u, v)).
\]
As in the previous application of the implicit function theorem (see the map \(\Gamma\)), the same regularity assumptions hold for \(\tilde{\Gamma}\). In particular, the Jacobian of \(\tilde{\Gamma}\) with respect to \(u\) at \((P_1, \dots, P_N, v_0, u_0)\) is invertible. Therefore, by the implicit function theorem, there exist:
\begin{itemize}
    \item an open neighborhood \(W \subset (\mathbb{R}_{\leq g}[x_1, \dots, x_d])^N \times \mathbb{R}^{d-N}\) of \((P_1, \dots, P_N, v_0)\),
    
    \item an open neighborhood \(W' \subset \mathbb{R}^N\) of \(u_0\), and

    \item a unique continuously differentiable function \(\psi: W \to W'\) such that for any \((Q_1, \dots, Q_N, v) \in W\) and \(u \in W'\),
\[
Q_1(u, v) = \cdots = Q_N(u, v) = 0 \quad \iff \quad u = \psi(Q_1, \dots, Q_N, v).
\]
\end{itemize}

Without loss of generality, we may write \(W = W_1 \times W_2\), where \(W_1\) is a neighborhood of \((P_1, \dots, P_N)\) in the coefficient space, and \(W_2\) is a neighborhood of \(v_0\) in \(\mathbb{R}^{d-N}\). Then for any \((Q_1, \dots, Q_N) \in W_1\) and \(v \in W_2\),
\[
Q_1(u, v) = \cdots = Q_N(u, v) = 0 \quad \iff \quad u = \psi(Q_1, \dots, Q_N, v).
\]
Recall that $\hat{\rm I}_g$ (see~\eqref{eq:hat-Ig}) is a subspace of the space of real polynomials of degree at most $g$, and $\Pi_{\hat{\rm I}_g}$ is the orthogonal projection operator (with respect to the standard inner product on the coefficient vectors). 

Applying the implicit function theorem to a perturbed system of polynomial equations defined by the projected polynomials $\Pi_{\hat{\rm I}_g} P_1,\ldots, \Pi_{\hat{\rm I}_g} P_N$ (recall~\eqref{Simplification-tau}), we obtain:
\[
(\Pi_{\hat{\rm I}_g} P_1)(u, v) = \cdots = (\Pi_{\hat{\rm I}_g} P_N)(u, v) = 0 
\quad \iff \quad u = \psi(\Pi_{\hat{\rm I}_g} P_1, \dots, \Pi_{\hat{\rm I}_g} P_N, v),
\]
and, for \(n\) large enough,
\[
(u, v) \in \mathcal{A}^{(n,g)} \cap (W' \times W_2) 
\quad \iff \quad u = \psi(\Pi_{\hat{\rm I}_g} P_1, \dots, \Pi_{\hat{\rm I}_g} P_N, v) \text{ and } v \in W_2.
\]
This shows that points in the estimated set \( \mathcal{A}^{(n, g)} \) near $x_0$ can be parametrized similarly via $\psi$.

Now fix a bounded neighborhood $\mathcal{U} = \tilde{W}' \times \tilde{W}_2 \subset \mathbb{R}^d$ such that its closure is contained in $W' \times W_2$. Then, for all $n$ sufficiently large, $x = (u, v) \in \mathcal{A} \cap \mathcal{U}$ corresponds to a point $x_n = (u_n, v)$ in the estimated set \( \mathcal{A}^{(n,g)} \cap \mathcal{U} \), where:
\begin{equation}
  u = \psi(P_1, \dots, P_N, v), \qquad \mbox{and} \qquad 
  u_n = \psi(\Pi_{\hat{\rm I}_g} P_1, \dots, \Pi_{\hat{\rm I}_g} P_N, v).
\end{equation}
Hence,
\[
\sup_{x \in \mathcal{A} \cap \mathcal{U}} \inf_{y \in \mathcal{A}^{(n,g)} \cap \mathcal{U}} \|x - y\| 
\leq \sup_{v \in \tilde{W}_2} \left\| \psi(\Pi_{\hat{\rm I}_g} P_1, \dots, \Pi_{\hat{\rm I}_g} P_N, v) 
- \psi(P_1, \dots, P_N, v) \right\|,
\]
and similarly,
\[
\sup_{y \in \mathcal{A}^{(n,g)} \cap \mathcal{U}} \inf_{x \in \mathcal{A} \cap \mathcal{U}} \|x - y\| 
\leq \sup_{v \in \tilde{W}_2} \left\| \psi(\Pi_{\hat{\rm I}_g} P_1, \dots, \Pi_{\hat{\rm I}_g} P_N, v) 
- \psi(P_1, \dots, P_N, v) \right\|.
\]

Note that \(\psi(\cdot,v) \in \mathcal{C}^1\) is  Lipschitz with a constant that does not depend on $v \in \overline{\tilde{W}_2}$.
Therefore, the supremum difference above is of the same order as the difference in coefficient projections:
\[
\sup_{v \in \tilde{W}_2} \left\| \psi(\Pi_{\hat{\rm I}_g} P_1, \dots, \Pi_{\hat{\rm I}_g} P_N, v) 
- \psi(P_1, \dots, P_N, v) \right\| = O(\|\Pi_{\hat{\rm I}_g} - \Pi_{{\rm I}_g}\|_{\rm Fr}).
\]
Now, by Theorem~\ref{Theorem:main} we obtain the desired parametric rate. This establishes a local convergence rate near \(x_0\). 
\end{proof}

\subsection{Proofs of Section~\ref{Section:tubes}}\label{App:Thm-Semi-Alg}

We divide the proof of Theorem~\ref{Theorem:completeRecovery} into two parts: the first establishes consistency, and the second derives the convergence rates.

\begin{proof}[Proof of Theorem~\ref{Theorem:completeRecovery} (Consistency)]

The consistency result follows from the auxiliary lemma stated below.

\begin{Lemma}[Consistency of a tubular approximation]
\label{Lemma:auxiliaryRecoveringSet}
Fix a degree \(g\) and consider, for each \(n\), a collection of polynomials ${\bf Q}_n \;=\;\bigl\{Q_{1,n},\dots,Q_{N_n,n}\bigr\}
\;\subset\; \R_{\le g}[x_1,\dots,x_d],$
whose coefficient vectors are orthonormal in the standard monomial basis
of total degree \(\le g\).
For \(\lambda>0\) set
\[
\mathcal{V}_{\lambda}({\bf Q}_n)
\;:=\;
\Bigl\{x\in\R^{d} \;:\; |Q(x)|\le\lambda
      \;\text{ for every }\; Q\in{\bf Q}_n\Bigr\}.
\]

Let \(\tilde\Pi_n\) denote the orthogonal projection (in coefficient space) onto
\(\operatorname{span}\{\tau Q_{1,n},\dots,\tau Q_{N_n,n}\}\),
and write \(\Pi_n:=\tau^{-1}\tilde\Pi_n\tau\).
Suppose there exists a deterministic sequence \(r_n'\downarrow0\) such that
\[
\frac{\|\Pi_n-\Pi_{{\rm J}_g}\|_{\mathrm{Fr}}}{r_n'}\;\longrightarrow\;0,
\qquad n\to\infty.
\]
Then for any sequence \(\lambda_n\downarrow0\) that satisfies
\(\lambda_n / r_n' \to\infty\), we have
\[
{\bf d}_{\mathrm{PK}}\bigl(\mathcal{V}_{\lambda_n}({\bf Q}_n),\;\mathcal{A}\bigr)
\;\longrightarrow\;0,
\]
where \(d_{\mathrm{PK}}\) denotes the Painlevé–Kuratowski distance between sets.
\end{Lemma}

We now explain how to deduce Theorem~\ref{Theorem:completeRecovery} from Lemma~\ref{Lemma:auxiliaryRecoveringSet}. Since \(\lambda_n \to 0\) and \(n^{1/2} \lambda_n \to \infty\) by assumption, one can choose a sequence \(r_n' \to 0\), as $n \to \infty$, such that
\[
n^{1/2} r_n' \to \infty \qquad \text{and} \qquad (r_n')^{-1} \lambda_n \to \infty.
\]
From Theorem~\ref{Theorem:main}, we have
\[
{\rm d}(\hat{\rm J}_g, {\rm J}_g) = \mathcal{O}_{\mathbb{P}}\left(n^{-1/2}\right),
\]
which implies
\[
(r_n')^{-1} {\rm d}(\hat{\rm J}_g, {\rm J}_g) \xrightarrow{\mathbb{P}} 0.
\]
By Lemma~\ref{Lemma:auxiliaryRecoveringSet}, applied to \(\mathcal{V}_{\lambda_n}({\bf Q}_n) = \mathcal{A}_{\lambda_n}^{(n,g)}\) and $r_n'$ as in the above display, it follows that \(\mathcal{A}_{\lambda_n}^{(n,g)}\) converges to \(\mathcal{A}\) in the Painlevé–Kuratowski sense. Since the metric space $({\rm CL}_{\neq \emptyset}(\R^d), {\bf d}_{\mathrm{PK}})$ is separable (see \cite[Chapter~4]{RockafellarWets}), convergence in probability in this space is equivalent to the property that every subsequence has an a.s.~converging subsubsequence. In our case, Lemma~\ref{Lemma:auxiliaryRecoveringSet} ensures that along any such subsequence, $\mathcal{A}_{\lambda_n}^{(n,g)}$ converges a.s.~to $\mathcal{A}$. Therefore, this implies convergence in probability:
\[
{\bf d}_{\mathrm{PK}}(\mathcal{A}_{\lambda_n}^{(n,g)}, \mathcal{A}) \xrightarrow{\mathbb{P}} 0.
\]

Now we proceed with the proof of the lemma.

{
\begin{proof}[Proof of Lemma~\ref{Lemma:auxiliaryRecoveringSet}]
We show the two inclusions that characterize Painlevé–Kuratowski convergence:

\begin{enumerate}\setlength\itemsep{4pt}
\item[(i)] \textit{(Inner approximation).}  
For every $x\in\mathcal A$ there exists $n_0=n_0(x)$ such that 
$x\in\mathcal V_{\lambda_n}({\bf Q}_n)$ for all $n\ge n_0$.
\item[(ii)] \textit{(Outer approximation).}  
Every limit point of a sequence $\{x_n\}$ with
$x_n\in\mathcal V_{\lambda_n}({\bf Q}_n)$ belongs to $\mathcal A$.
\end{enumerate}
Together, (i) and (ii) imply ${\bf d}_{\mathrm{PK}}\!\bigl(\mathcal V_{\lambda_n}({\bf Q}_n),\mathcal A\bigr)\to0$. The rest of the proof is devoted to showing (i) and (ii).  We denote $N=\dim {\rm J}_g$ and fix a basis $P_1,\dots,P_N$ of ${\rm I}_g$ whose
coefficient vectors $\tau P_1,\dots,\tau P_N$ form an orthonormal basis of
${\rm J}_g$.  We abuse of notations and denote by $\langle P, Q\rangle$ the inner product of the coeficients of $P$ and $Q$. 
Set $E_n:=\operatorname{span}{\bf Q}_n=\operatorname{range}\Pi_n$.
For $n$ large enough, $\dim E_n=N$ (as in the proof of
Theorem~\ref{Theorem:recoveringSet}), and
$\bigl\{\Pi_nP_i\bigr\}_{i=1}^{N}$ is itself a basis of $E_n$. In the rest of the proof we assume that $n$ is large enough so that  $\dim E_n=N$. Moreover, it follows that 
\begin{align}
    \begin{split}\label{eq:controled-norm-seialg}
        |\langle  \Pi_nP_i , \Pi_nP_j \rangle- \langle  P_i , P_j \rangle |& = |\langle  \Pi_nP_i , \Pi_nP_j \rangle- \langle \Pi_{{\rm J}_g} P_i , \Pi_{{\rm J}_g} P_j \rangle |\\
        &\leq |\langle  \Pi_nP_i -\Pi_{{\rm J}_g} P_i , \Pi_nP_j \rangle|+ |\langle  \Pi_{{\rm J}_g} P_i , \Pi_nP_j -\Pi_{{\rm J}_g} P_j \rangle| \\
        &\leq \|\Pi_n-\Pi_{{\rm J}_g}\|_{\mathrm{Fr}} (\| \Pi_n P_j\|+1) \\
        &\le 2 \|\Pi_n-\Pi_{{\rm J}_g}\|_{\mathrm{Fr}} \to 0,
    \end{split}
\end{align}
where in the last inequality we use the fact that an orthogonal projection is a contraction.

\medskip
\noindent
\emph{Proof of (i).}
Fix $x\in\mathcal A$.  
Since $\Pi_{{\rm J}_g} P_i (x)= P_i (x)=0  $, it follows that
\begin{equation*}
\bigl|(\Pi_nP_i)(x)\bigr|
   =\bigl|(\Pi_{{\rm J}_g}-\Pi_n)P_i(x)\bigr|\le C\|x\|^{g}\,\|\Pi_n-\Pi_{{\rm J}_g}\|_{\mathrm{Fr}},
\end{equation*}
where $C=C(d,g)$ depends only on $d$ and $g$.
By assumption,
$\|\Pi_n-\Pi_{{\rm J}_g}\|_{\mathrm{Fr}}=o(r_n')$, while
$\lambda_n/r_n'\to\infty$; hence $\bigl|(\Pi_nP_i)(x)\bigr|=o(\lambda_n)$. Fix $Q_{i,n}\in {\bf Q}_n$ and use the fact that  as $\bigl\{\Pi_nP_i\bigr\}_{i=1}^{N}$ is a basis of $E_n$ to get that 
\begin{equation}
    \label{eq:representation-Qn-semi}
    Q_{i,n}=\sum_{j=1}^N c_{j,n} \Pi_nP_i,
\end{equation}
for some real coefficients $c_{1,n}, \dots, c_{N,n} $. We note that (i) follows after showing the sequence of coefficients  is bounded. Taking norms in \eqref{eq:representation-Qn-semi} we get 
\begin{align*}
    1=\|Q_{i,n}\|&= \sum_{j\neq k} c_{j,n}^2 \| \Pi_nP_j\|^2 + \sum_{j\neq k} c_{j,n} c_{k,n} \langle \Pi_nP_j,\Pi_nP_k \rangle \\
    &\geq \max_j  c_{j,n}^2 \min_{j} \| \Pi_nP_j\|^2- N \max_j  c_{j,n}^2  \max_{j\neq k}(\langle \Pi_nP_j,\Pi_nP_k \rangle)^2 .
\end{align*}
By \eqref{eq:controled-norm-seialg} there exists $n_0$ such that 
\begin{align*}
    \| \Pi_nP_j\|^2 \geq  1/2 \quad \text{and}\quad  \max_{j\neq k}(\langle \Pi_nP_j,\Pi_nP_k \rangle)^2 \leq \frac{1}{4 N} \quad \text{for all $n\geq n_0$,}
\end{align*}
so that 
$ \max_j  c_{j,n}^2 \leq 2$ for $n$ large enough. This concludes  the proof of (i).

\medskip
\noindent
\emph{Proof of (ii)}. 
Let $x_n\in\mathcal V_{\lambda_n}({\bf Q}_n)$ and
assume $x_n\to\bar x$. Fix $j\in \{1, \dots, N\}$. 
We note that  $\|\Pi_nP_j\|\leq 1$. 
Since $\Pi_nP_j \in  E_n$ and ${\bf Q}_n$ is an orthonormal system, there exist real coefficients $c_{1,n}', \dots, c_{N,n}' $ with $ \sum_{i=1}^{N} (c_{i,n}')^2\leq 1$ and such that 
$ \Pi_nP_j=\sum_{i=1}^N c_{i,n}' Q_{i,n} $. Then it follows that  
$$| \Pi_nP_j(x_n)|  \leq \sum_{i=1}^N |c_{i,n}'| |Q_{i,n}(x_n)| \leq \lambda_n \sum_{i=1}^N |c_{i,n}'| \leq \sqrt{N}  \lambda_n ,$$
which implies
\begin{align}
 | P_j(\Bar{x})|&\leq | P_j(\Bar{x})-P_j(x_n)|+ | \Pi_nP_j({x}_n)-P_j(x_n)| + |\Pi_nP_j({x}_n)| \\
 &\leq  | P_j(\Bar{x})-P_j(x_n)|+ | \Pi_nP_j({x}_n)-P_j(x_n)| + \sqrt{N} \lambda_n \to 0,
\end{align}
where the first term tends to zero by the continuity of $P_j$, the second by $\|\Pi_n-\Pi_{{\rm J}_g}\|_{\mathrm{Fr}}\to 0$ and the last one by $
\lambda_n\to 0$.
Therefore,  (ii) holds.
\end{proof}}

\begin{proof}[Proof of Theorem~\ref{Theorem:completeRecovery} (Convergence rate)]
Recall that $\widehat P_1,\dots,\widehat P_{\widehat N}$ denote the
data–driven polynomials that define the semi-algebraic tube
\(
\mathcal A^{(n,g)}_{\lambda_n}
    :=\bigl\{x\in\R^d : |\widehat P_i(x)|\le\lambda_n
      \text{ for all }i=1,\dots,\widehat N\bigr\}.
\)
We prove, on a suitably chosen neighborhood~$\mathcal U\subset\R^d$,
\begin{align}
\sup_{x\in\mathcal A\cap\mathcal U}
      \inf_{y\in\mathcal A^{(n,g)}_{\lambda_n}\cap\mathcal U}\|x-y\|
    &=\mathcal O_{\mathbb P}(\lambda_n),\label{eq:first-claim}\\[4pt]
\sup_{y\in\mathcal A^{(n,g)}_{\lambda_n}\cap\mathcal U}
      \inf_{x\in\mathcal A\cap\mathcal U}\|x-y\|
    &=\mathcal O_{\mathbb P}(\lambda_n).\label{eq:second-claim}
\end{align}
The first inclusion requires only that $\mathcal U$ be bounded; the
second needs the local chart supplied by
Lemma~\ref{lemma:Local-ring-same-degree}.

\medskip
\noindent
\emph{Proof of \eqref{eq:first-claim}.}
Let $\mathcal U\subset\R^{d}$ be bounded.
Observe that
\[
{\mathbb P}\Bigl(\sup_{x\in\mathcal A\cap\mathcal U}
               \inf_{y\in\mathcal A^{(n,g)}_{\lambda_n}\cap\mathcal U}\|x-y\|>0\Bigr)
 \le
{\mathbb P}\Bigl(
        \sup_{x\in\mathcal A\cap\mathcal U}
        \max_{1\le i\le\widehat N}|\widehat P_i(x)|>\lambda_n\Bigr).
\]
Because $g\ge g^{\!*}$, a point $x$ lies in $\mathcal A$ if and only if
\(
(\Pi_{{\rm J}_g}P)(x)=0
\)
for every $P\in\R_{\le g}[x_1,\dots,x_d]$.
Hence
\begin{align*}
&{\mathbb P}\Bigl(
       \sup_{x\in\mathcal A\cap\mathcal U}
       \max_{i}|\widehat P_i(x)|>\lambda_n\Bigr)                                       \\
&\quad=\;
{\mathbb P}\Bigl(
       \sup_{x\in\mathcal A\cap\mathcal U}
       \max_{i}|(\widehat P_i-\Pi_{{\rm J}_g}\widehat P_i)(x)|>\lambda_n\Bigr)           \\
&\quad\le
{\mathbb P}\Bigl(
       \sup_{x\in\mathcal U}
       \max_{i}|((\Pi_{\widehat{\rm J}_g}-\Pi_{{\rm J}_g})\widehat P_i)(x)|
       >\lambda_n\Bigr).
\end{align*}
Evaluation at $x$ is a bounded linear functional on coefficient space, so
for some constant $C=C(\mathcal U,g,d)$,
\[
|((\Pi_{\widehat{\rm J}_g}-\Pi_{{\rm J}_g})\widehat P_i)(x)|
    \le C\,\|\Pi_{\widehat{\rm J}_g}-\Pi_{{\rm J}_g}\|_{\mathrm{Fr}}
    \quad\text{for all }x\in\mathcal U.
\]
Theorem~\ref{Theorem:main} gives
\(
\|\Pi_{\widehat{\rm J}_g}-\Pi_{{\rm J}_g}\|_{\mathrm{Fr}}
        =\mathcal O_{\mathbb P}(n^{-1/2}),
\)
and $\lambda_n/n^{-1/2} \to 0 $ by assumption; hence the probability above
tends to~$0$, proving \eqref{eq:first-claim}.

\medskip\noindent
\emph{Proof of \eqref{eq:second-claim}.}
Fix $x_0\in\operatorname{reg}(\mathcal A)$ and let
$\mathcal U$, $R_1,\dots,R_D$, and the chart
$(v,w)\mapsto v+\phi(w)$ be as in
Lemma~\ref{lemma:Local-ring-same-degree},
with $\{\nabla R_i(x_0)\}_{i=1}^{D}$ aligned with the first $D$
coordinate axes.
Writing $x=(v,w)\in\R^{D}\times\R^{d-D}$,
a first-order Taylor expansion about the point
$(\phi(w),w)\in\mathcal A$ yields
\[
\bigl(R_1(v,w),\dots,R_D(v,w)\bigr)^{\!\top}
     =J_{R}(u^\ast,w)\,(v-\phi(w)),
\]
where $u^\ast$ lies on the segment joining $v$ and $\phi(w)$ and
$J_{R}$ is the $D\times D$ Jacobian
$\bigl(\partial_{v_j}R_i\bigr)_{i,j=1}^{D}$.
By continuity of $J_R$ and
the non-singularity at $x_0$, we may shrink~$\mathcal U$ so that
$\|J_{R}(u^\ast,w)^{-1}\|\le L$ for all $(u^\ast,w)\in\mathcal U$.
Hence
\[
\|v-\phi(w)\|
  \;\le\;
  L\,\|(R_1(v,w),\dots,R_D(v,w))\|.
\]
Now let $(v,w)\in\mathcal A^{(n,g)}_{\lambda_n}\cap\mathcal U$.
By construction of $\mathcal A^{(n,g)}_{\lambda_n}$
and Theorem~\ref{Theorem:main},
\(
\max_{1\le i\le D}|R_i(v,w)|=\mathcal O_{\mathbb P}(\lambda_n),
\)
whence
\(
\sup_{(v,w)\in\mathcal A^{(n,g)}_{\lambda_n}\cap\mathcal U}
     \|v-\phi(w)\|
   =\mathcal O_{\mathbb P}(\lambda_n).
\)
Because $u+\phi(u)$ parametrizes
$\mathcal A\cap\mathcal U$, this proves \eqref{eq:second-claim}.
\end{proof}
\medskip
Combining \eqref{eq:first-claim} and \eqref{eq:second-claim} completes the
proof of the convergence rate asserted in
Theorem~\ref{Theorem:completeRecovery}.
\end{proof}

\begin{proof}[Proof of Lemma~\ref{lemma:estimation-cross-ideal}]
Throughout the proof we use the elementary fact that, for any closed set
\(\tilde S\subset\R^{\kappa_{d,g^{\!*}}}\) and any measurable projection
\(\Pi_S:\R^{\kappa_{d,g^{\!*}}}\to\tilde S\),
\begin{equation}\label{eq:projection-2Lipschitz}
\|\,\Pi_S(x)-s\|\;\le\;2\|x-s\|,
\qquad x\in\R^{\kappa_{d,g^{\!*}}},\;s\in\tilde S,
\end{equation}
because \(\|x-\Pi_S(x)\|\le\|x-s\|\) by optimality of the projection.

\vspace{4pt}
\noindent\emph{(a)  Every limit point lies in \({\rm J}_{g^{\!*}}\).}
Let \(u_n\in\hat{\rm J}_{g^*}^S\) and assume
\(u_n\stackrel{\mathbb P}{\longrightarrow}u\).
By definition of $\hat{\rm J}_{g^*}^S$ there exists
\(v_n\in\hat{\rm J}_{g^*}^S$ with \(\|v_n\|\le1\) such that
\(u_n=\Pi_S(v_n)\).
Corollary~\ref{Coro:ConsistencyPolynomials} (consistency of the
kernel estimator) guarantees
\(v_n\stackrel{\mathbb P}{\longrightarrow}v_\infty\)
for some \(v_\infty\in{\rm J}_{g^{\!*}}\subset\tilde S\).
Applying \eqref{eq:projection-2Lipschitz} with \(s=v_\infty\) yields
\[
\|u_n-v_\infty\|
   \le 2\|v_n-v_\infty\|
   \xrightarrow{\mathbb P} 0,
\]
hence \(u=v_\infty\in{\rm J}_{g^{\!*}}\).

\vspace{4pt}
\noindent\emph{(b)  Density of $\hat{\rm J}_{g^*}^S$ in
\({\rm J}_{g^{\!*}}\).}
Fix \(u\in{\rm J}_{g^{\!*}}\) with \(\|u\|\le1\).
By Corollary~\ref{Coro:ConsistencyPolynomials} there exists
$v_n\in\hat{\rm J}_{g^*}^S$ such that
\(v_n\stackrel{\mathbb P}{\longrightarrow}u\).
Set
\[
w_n\;:=\;\frac{\|u\|}{\|v_n\|}\,v_n
\quad\bigl(\text{define }w_n:=0\text{ if }v_n=0\bigr),
\]
so that \(\|w_n\|\le1\) and \(w_n\stackrel{\mathbb P}{\longrightarrow}u\).
Define \(u_n:=\Pi_S(w_n)\in\hat{\rm J}_{g^*}^S\).
Using \eqref{eq:projection-2Lipschitz} with \(s=u\) we obtain
\[
\|u_n-u\|
   \le 2\|w_n-u\|
   \xrightarrow{\mathbb P} 0,
\]
so \(u_n\stackrel{\mathbb P}{\longrightarrow}u\), completing the proof.
\end{proof}

\begin{proof}[Proof of Theorem~\ref{Theorem:recoveringSet-Prior}] First we underline some useful observations.  Assumption~\ref{Assumption-prior-knloegde-set} implies that ${\rm dim}({\rm J}_{g^*})=1$. We denote by $P$ a generator of ${\rm I}_{g^*}$ and $u\in {\rm J}_{g^*}$ its coefficients. For a polynomial $Q$ with vector of coefficients $q$,   we denote by $\|Q\|$ the Euclidean norm of the coefficients $q$ of $Q$ and $\Pi_S(Q)$ the polynomial associated to $ \Pi_S(q)$. 

We now show that $\tilde S$ is closed (in coefficient space). Let
\[
P^{(n)} \;=\; \lambda_n\, Q^{(n)}_1 \cdots Q^{(n)}_m \;\in\; \tilde S,
\qquad \|Q^{(n)}_i\|=1\ \ (i=1,\dots,m),
\]
and assume $P^{(n)} \to P$ as $n\to\infty$. Since each coefficient space is finite
dimensional, the unit spheres $\{\|Q\|=1\}$ are compact; by a diagonal argument
we may assume (passing to a subsequence) that $Q^{(n)}_i \to Q_i$ for every $i$.

If $(\lambda_n)$ is bounded, pass to a further subsequence with
$\lambda_n \to \lambda \in \mathbb{R}$. Polynomial multiplication is continuous
in coefficient space, hence
\[
P \;=\; \lim_{n\to\infty} P^{(n)}
  \;=\; \lim_{n\to\infty} \lambda_n\, Q^{(n)}_1 \cdots Q^{(n)}_m
  \;=\; \lambda\, Q_1 \cdots Q_m \;\in\; \tilde S.
\]
Suppose instead $|\lambda_n|\to\infty$. Then, as $P^{(n)} \to P$,
\[
Q^{(n)}_1 \cdots Q^{(n)}_m
 \;=\; \frac{1}{\lambda_n} P^{(n)} \;\longrightarrow\; 0
\]
in coefficient norm, so by continuity of multiplication we obtain
$Q_1 \cdots Q_m = 0$. Since $\mathbb{R}[x_1,\dots,x_d]$ is an integral domain,
this implies $Q_i \equiv 0$ for some $i$, contradicting $\|Q_i\|=1$.
Therefore $(\lambda_n)$ must be bounded, and the first case applies, yielding
$P\in\tilde S$. Hence $\tilde S$ is closed.

We now establish the parametric rate of convergence around regular points of~\(\mathcal A\). Exactly as in the proof of Theorem~\ref{Theorem:recoveringSet}, there exists a (data–dependent) index \(N\in\mathbb N\) such that
\[
\dim\!\bigl(\hat{\rm J}_{g^{\!*}}\bigr)=1
\qquad\text{for all }n\ge N .
\]
Hence, for \(n\ge N\) the projected space
\(\Pi_S\!\bigl(\hat{\rm J}_{g^{\!*}}\bigr)\subset\tilde S\) is also one–dimensional. Because \(\Pi_S\) is \(2\)-Lipschitz (inequality~\eqref{desigualdad-Gilles}),
\[
{\rm d}\!\bigl(\Pi_S(\widehat{\rm J}_{g^{\!*}}),\,{\rm J}_{g^{\!*}}\bigr)
   \;=\;\mathcal O_{\mathbb P}\!\bigl(n^{-1/2}\bigr).
\]
Repeating the argument of Theorem~\ref{Theorem:recoveringSet}, with Lemma~\ref{lemma:Local-ring-same-degree} in place of its irreducible counterpart, yields:
for every \(x_0\in\operatorname{reg}(\mathcal A)\) there exists an open Euclidean neighborhood \(\mathcal U\)  of $x_0$ such that
\[
d_{\mathrm H}\!\bigl(\mathcal A\cap\mathcal U,\,
        \mathcal A^{(n,g^{\!*})}_{S}\cap\mathcal U\bigr)
   =\mathcal O_{\mathbb P}\!\bigl(n^{-1/2}\bigr).
\]
This proves the desired local \(n^{-1/2}\) convergence rate.

We now show the consistency of the estimator.
Fix\(\Pi_S(P^{(n)})\xrightarrow{\mathbb P}P\) and recall that, for
\(n\ge N\),
\[
  \mathcal A^{(n,g^{\!*})}_{S}
  \;=\;\mathcal V\!\bigl(\Pi_S(P^{(n)})\bigr).
\]
It suffices to prove the following general statement: There exists an open Euclidean neighborhood
\(\mathcal U\ni x_0\) such that, for every sequence
\(Q^{(n)}\in S\) with \(Q^{(n)}\to P\),
\begin{equation}\label{eq:local-Hausdorff}
  d_{\mathrm H}\!\bigl(
      \mathcal A\cap\mathcal U,\;
      \mathcal V(Q^{(n)})\cap\mathcal U\bigr)
  \;\xrightarrow{\mathbb P}\;0.
\end{equation}

Denote \(\mathcal A^{n}_S:=\mathcal V(Q^{(n)})\).
To establish \eqref{eq:local-Hausdorff} we verify, for a suitably chosen
\(\mathcal U\),

\begin{equation}\label{eq:two-limits}
\sup_{z\in\mathcal A^{n}_S\cap\mathcal U}
          \inf_{x\in\mathcal A\cap\mathcal U}\|x-z\|\;\xrightarrow{\mathbb P} 0,
\qquad
\sup_{x\in\mathcal A\cap\mathcal U}
          \inf_{z\in\mathcal A^{n}_S\cap\mathcal U}\|x-z\|\;\xrightarrow{\mathbb P} 0.
\end{equation}

\medskip
\noindent{\it First limit in~\eqref{eq:two-limits}.}
Fix any bounded open neighborhood \(\mathcal U\) of \(x_0\).
Assume by contradiction that the first limit in
\eqref{eq:two-limits} fails: then there exist \(\varepsilon>0\) and a
subsequence \(z_{n_k}\in\mathcal A^{n_k}_S\cap\mathcal U\) such that
\(
\inf_{x\in\mathcal A\cap\mathcal U}\|x-z_{n_k}\|\ge\varepsilon.
\)
Because \(\overline{\mathcal U}\) is compact, \(z_{n_k}\) admits a further
subsequence converging to some \(z_\infty\in\overline{\mathcal U}\).
Since \(Q^{(n)}\to P\) and \(Q^{(n_k)}(z_{n_k})=0\), continuity yields
\(P(z_\infty)=0\); hence \(z_\infty\in\mathcal A\cap\overline{\mathcal U}\).
But then
\(
\|z_{n_k}-z_\infty\|\to0,
\)
contradicting the choice of \(\varepsilon\).
Thus the first limit in \eqref{eq:two-limits} holds.

\medskip\noindent{\it Second limit in~\eqref{eq:two-limits}.}
Factor \(Q^{(n)}=\lambda_n Q^{(n)}_1\cdots Q^{(n)}_m\) with
\(\|Q^{(n)}_i\|=1\).
By compactness, along a subsequence
\((\lambda_n,Q^{(n)}_1,\dots,Q^{(n)}_m)\to(\lambda,P_1,\dots,P_m)\) and
\(P=\lambda P_1\cdots P_m\).
Let \(P_1,\dots,P_s\) be the factors vanishing at the chosen regular
point \(x_0\).  Since
\(\mathcal V(\{P_{s+1},\cdots, P_m\})\) is closed and
\(x_0\notin\mathcal V(\{P_{s+1},\cdots, P_m\})\),
there exists a neighborhood
\(\mathcal U\subset\mathcal U'\) of \(x_0\) that does not intersect
\(\mathcal V(\{P_{s+1},\cdots P_m\})\).

Since \(\nabla P_{k}(x_0)\ne 0\) for \(k\le s\),
there exists a unit vector \(h\) with
\(\langle\nabla P_{k}(x_0),h\rangle\ne 0\) for all \(k\le s\).
Rotate so that \(h=e_{1}\) and write \(x=(v,w)\in\R\times\R^{d-1}\)
with \(x_0=(v_0,w_0)\).

\medskip\noindent
\emph{Implicit–function parameterization.}
For a polynomial \(R\) and \(x=(v,w)\), define the functional 
    $$ \Gamma: \R_{\le g_1}[ x_1, \dots, x_d] \times \R^{d-1}\times \R \ni (R,w, v)\mapsto  R(w, v). $$
Because \(\partial_v \Gamma(P_{k},w_0,v_0)\neq 0\) (\(k\le s\)),
the implicit function theorem yields open sets
\(\mathcal U_{k}\ni P_{k}\), \(\mathcal W_{k}\ni w_0\),
\(\mathcal V_{k}\ni v_0\) and \(\mathcal C^{1}\) maps
\(
  \phi_{k} : \mathcal U_{k}\times\mathcal W_{k}\to\mathcal V_{k}
\)
such that
\[
   R(v,w)=0 \quad 
   \;\Longleftrightarrow\; \quad
   v = \phi_{k}(R,w)
   \quad
   (R,w,v)\in\mathcal U_{k}\times\mathcal W_{k}\times\mathcal V_{k}.
\]
Because \(Q^{(n)}_{k} \to P_{k}\), for \(n\) large
\(Q^{(n)}_{k}\in\mathcal U_{k}\) and
\(Q^{(n)}_{k}(v,w)=0 \iff v=\phi_{k}(Q^{(n)}_{k},w)\).

As a consequence, letting $ \mathcal{U}=\mathcal{V} \times \mathcal{W} $  be an open neighborhood of $(w_0,v_0)$ compactly contained in 
$  \bigcap_{k=1}^s \mathcal{V}_k \times \mathcal{W}_k \cap \mathcal{U}' $,  it follows that 
$$ Q^{(n)}(v,w)=0 \quad \iff \left(  v=\phi_1(Q_1^{(n)},w) \ {\rm or} \ \dots \ {\rm or} \ v=\phi_s(Q_s^{(n)},w) \right)$$
for all $(v,w)\in \mathcal{U}$. 
Hence, 
\begin{align*}
    \sup_{x\in \mathcal{A}\cap \mathcal{U}}   \inf_{ z\in \mathcal{A}^{n}_S \cap \mathcal{U}}\| x-z\| &= \sum_{k=1}^s \sup_{w\in \mathcal{W}}   \inf_{ (v',w')\in \mathcal{A}^{n}_S \cap \mathcal{U}}\| w-w' \|+ |\phi_k(Q^{(n)}_k,w) -v'| \\
    &\leq \sum_{k=1}^s \sup_{w\in \mathcal{W}}  |\phi_k(Q^{(n)}_k,w) -\phi_k(P_k,w)| \xrightarrow{\mathbb P} 0,
\end{align*}
where the limit follows from the continuity of $\phi_k$. 
\end{proof}

\end{document}